\documentclass[11pt, a4paper]{article}

\usepackage[T1]{fontenc}
\usepackage{listings}
\usepackage{url}
\usepackage[margin=1in]{geometry} 
\usepackage{algorithm}
\usepackage{algorithmic}
\usepackage{subfigure}

\usepackage[noadjust]{cite}

\usepackage[colorlinks=true,citecolor=red,linkcolor=blue,urlcolor=blue]{hyperref}

\usepackage{booktabs, makecell, array, xcolor,tabularx}
\usepackage{amsmath}                              
\usepackage{graphicx}         
\usepackage{booktabs}         
\usepackage{paper}
\usepackage{amsmath,amssymb}
\usepackage{color}
\usepackage{authblk}
\usepackage{changes}
\usepackage{cancel}
\usepackage{caption}
\usepackage{pgfplots}
\usepackage{pgfplotstable}
\pgfplotsset{compat=1.18}
\title{
Nonparametric Inference on Unlabeled Histograms
}
\author{Yun Ma and 
Pengkun Yang\thanks{Y.\ Ma and P.\ Yang are with Department of Statistics and Data Science, Tsinghua University. 
P.\ Yang is supported in part by the National Key R\&D Program of China 2024YFA1015800, the NSFC Grant 12101353, and Tsinghua University Dushi Program 2025Z11DSZ001. 
}}

\begin{document}
\maketitle

\begin{abstract}

Statistical inference on histograms and frequency counts plays a central role in categorical data analysis.
Moving beyond classical methods that directly analyze labeled frequencies, we introduce a framework that models  the multiset of unlabeled histograms via a mixture distribution to better capture unseen domain elements in large-alphabet regime.
We study the nonparametric maximum likelihood estimator (NPMLE) under this framework, and establish its optimal convergence rate under the Poisson setting.
The NPMLE also immediately  yields flexible and  efficient plug-in estimators for functional estimation problems, where a localized variant further achieves the optimal sample complexity for a wide range of symmetric functionals.
Extensive experiments on  synthetic, real-world datasets, and large language models highlight the practical benefits of the proposed method.

\end{abstract}
\tableofcontents
\section{Introduction}
Histograms appear ubiquitously in real-world applications and have long been a key focus of frequentist inference, which arise  from partitioning numerical data into discrete bins. 
They also serve as natural summary statistics of nominal observations, where the data are typically collected as labels from a large population, such as species in ecology \cite{Corbet1941}, words in linguistics \cite{ET76}, and tokens in large language models~\cite{Vaswani2017attention}.

Statistical inference on histogram data is typically carried out under an underlying statistical model. For categorical data, a natural  approach, referred to as \textit{$P$-modeling}, aims to assign a probability mass to each category. Specifically, the observations \( X=(X_1,\dots,X_n) \) are modeled as independently and identically distributed (\iid) according to the distribution \( P = (p_1, p_2, \ldots) \). The frequency counts \( N 
= (N_1,N_2,\ldots) \) are then obtained by enumerating the occurrences of each category, where \( N_j = \sum_{i=1}^{n} \indc{X_i = j} \).

A major challenge of the $P$-modeling approach arises when many categories remain \textit{unseen}, as in large word corpora, genotype data, or species catalogs. 
Despite the inaccessibility of the unseen labels, the properties of the overall distribution can be inferred from the seen categories.
A long-standing problem in this context is estimating the number of unseen categories, with seminal work by Fisher in ecology \cite{fisher1943relation}, classical methods of the Good–Turing estimator \cite{GT53,GT56}, applications to vocabulary diversity \cite{ET76,thisted1987shakespeare}, and recent advances in large language models \cite{Kalai2024llmunseen,li2025evaluatingunseen}.
Clearly, methods directly based on the $P$-model such as the maximum likelihood estimator fail to detect unseen categories.
More generally, estimating the number of the unseen falls under \textit{symmetric functional estimation}, where the target remains invariant under permutations of category labels. 
A notable example is the Shannon entropy, which originates from Shannon’s seminal contributions \cite{shannon1948,Shannon1951} and widely studied in neuroscience \cite{strong1998entropy}, physics \cite{de2022entropy}, and large language models \cite{farquhar2024detecting}.
Other symmetric functionals, including distance to uniformity \cite{Batu2001Test,canonne2020survey} and R\'enyi entropy \cite{Acharya2017renyi,wang2024renyi}, have also been extensively studied.
Nevertheless, applying the $P$-modeling approach to symmetric functional estimation is reported to suffer from severe bias and can even be inconsistent (see, e.g., \cite{Efr82,Pan03}).

Another approach aims to fit an equivalent class of the \( P \)-model without labeling the categories \cite{OSVZ04}. To illustrate, consider the multiset $\{N_i\}_{i\in\naturals}$ of the frequency counts in $N$.
The likelihood of observing the multiset is given by 
\begin{align}
    \label{eq:PML}
    P(\{N_i\}) = \sum_{N': N'\sim N} P(N'),
\end{align}
where $N'\sim N$ represents that $N'$ and $N$ correspond to the same multiset, and \( P(N') \) denotes the likelihood of \( N' \) under the \( P \)-model.  
However, this method is encountered with significant computational challenge due to the combinatorial structure. Computing the likelihood  requires evaluating a matrix permanent (see \cite[Eq.~(15)]{PJW17}), which is known to be a \#P-complete problem \cite{VALIANT1979189}.  
As a result, maximizing the likelihood over all distributions, referred to as the \textit{profile maximum likelihood} (PML) \cite{OSVZ04}, is highly challenging and requires sophisticated algorithms for approximate computation \cite{PJW17,ACSS21}.  
In fact, even exactly solving the PML for a frequency sequence of length 10–20 is non-trivial \cite{P12}. 
In addition to the likelihood-based approach, other algorithms model $\{N_i\}_{i \in \mathbb{N}}$ based on the method of moments \cite{vv17,LMM,HS21}. However, they often rely on delicate moment-matching programs with performance sensitive to numerous tuning parameters, and the estimation of high-order moments could suffer from large variance.

In this paper, we introduce a novel framework that addresses both the statistical and computational challenges.  
We model the multiset of unlabeled histograms via a mixture formulation, which naturally leads to a maximum likelihood estimation procedure based on the \textit{nonparametric maximum likelihood estimator} (NPMLE).  
The formulation depends solely on the histogram without dependency on category labels. Moreover, the NPMLE is computationally tractable due to its convex structure.  
The resulting estimator is then applied to symmetric functional estimation using a plug-in approach, which is compatible with 
the classical \( P \)-model. Further methodological details are provided in the following subsections.

\subsection{Model and methodology}
\label{sec:intro-model}

In this subsection, we propose a mixture model for analyzing the frequency counts. 
Let \(q_n(\cdot, r)\) denote a prescribed distribution of the frequency counts with parameter $r\in [0,1]$.
For instance, under the $P$-model, the frequency count of a category with occurrence probability \( r \) across \( n \) \iid\ observations follows the binomial distribution   
\(\bin(x, n, r) \triangleq \binom{n}{x} r^x(1-r)^{n-x}\).
Another common choice is the Poisson distribution  
\( \poi(x, nr) \triangleq \frac{(nr)^x}{x!}e^{-nr} \), which applies when the sample size further follows a Poisson distribution with mean \( n \) (see, \eg, \cite[Sec.~2]{WY16}).

Apart from the $P$-model, we propose the $\pi$-modeling approach, in which each frequency count $N_i$ follows a mixture distribution:
\begin{align}
\label{eq:N_mixture}
    f_{\pi^\star}(\cdot)\triangleq \int q_n(\cdot, r) \diff \pi^\star(r),
\end{align}
where \(\pi^\star\) is a mixing distribution supported on \([0,1]\).
In particular, we refer to \eqref{eq:N_mixture} as the \textit{Poisson mixture} or \textit{binomial mixture} when $q_n$ is the Poisson or binomial distribution, respectively.

Given a multiset of frequency counts $N=\{N_1,N_2,\cdots\}$, the nonparametric maximum likelihood estimator is given by \cite{KW56}:
\begin{align}
     \label{eq:npmle-def}
    \hat{\pi} \in \argmax_{\pi \in \calP([0,1])} \  L(\pi;N),
\end{align}
where $\calP([0,1])$ denotes the set of all distributions supported on $[0,1]$, and the likelihood function \(L(\pi;N) \) is 
\begin{align}
\label{eq:Ni-likelihood}
    L(\pi;N) \triangleq \sum_{i} \log f_{\pi} (N_i). 
\end{align}
We will discuss other variations of the program in Sections~\ref{sec:theory} and~\ref{sec:discussion}.

The proposed model \eqref{eq:N_mixture} offers both statistical and computational advantages.
Under the $\pi$-model, the sequence of frequency counts \( (N_1, N_2, \dots) \) is \textit{exchangeable}, \ie, the joint distribution remains invariant under any permutation of indices.
This property is essential for tasks that are invariant to specific category labels and captures the potentially unseen elements. 
In contrast, under the $P$-model, each frequency \( N_i \) is associated with its own probability \( p_i \).
Moreover, the mixture formulation also benefits from computationally efficient procedures, which is facilitated by many recent advancements~\cite{KM13,rebayes17,Zhang2022NPMLE}.

To illustrate the advantage of our model in fitting the frequency multiset, we provide an example based on the butterfly dataset collected in the early 1940s by the naturalist Corbet \cite{Corbet1941}.  Table~\ref{tab:butterfly} shows the number of occurrence of each observed frequency count -- no specimens were observed for 304 species, 118 species were observed exactly once, and so on.
We fit both the Poisson mixture model (via~\eqref{eq:npmle-def} with $q_n(\cdot, r)=\poi(\cdot, nr)$) and the $P$-model (via the empirical distribution) using the frequency counts at most \(T\), and then perform the \(\chi^2\) goodness-of-fit test on the two models by calculating the testing statistic
$$\sum_{j=0}^T \frac{(\varphi_j - \Expect[\varphi_j])^2}{\Expect[\varphi_j]},$$
where \(\varphi_j \triangleq \sum_{i} \indc{N_i = j}\).
Under the null hypothesis that the model is correctly specified, it approximately follows the chi-squared distribution with $T$ degrees of freedom.
Figure~\ref{fig:butterfly} plots the histogram of observed frequency counts (from 0 to 30) along with the values of $\Expect[\varphi_j]$ under the fitted mixture, showing that the mixture model provides a good fit to the data\footnote{These $\Expect[\phi_j]$'s are computed from the conditional distribution of \eqref{eq:N_mixture} on \(\{0,1,\ldots, T\}\).}. 
Table~\ref{tab:p_values} displays the $p$-values under different levels $T$, which consistently reject the $P$-model and fail to reject the Poisson mixture model.

\begin{table}[ht]
\centering
\begin{tabular}{|c|*{16}{c|}}
\hline
\textbf{$j$} & 0 & 1 & 2 & 3 & 4 & 5 & 6 & 7 & 8 & 9 & 10 & 11 & 12 & 13 & 14 & 15 \\
\textbf{$\varphi_j$} & 304 & 118 & 74 & 44 & 24 & 29 & 22 & 20 & 19 & 20 & 15 & 12 & 14 & 6 & 12 & 6 \\
\hline
\textbf{$j$} & 16 & 17 & 18 & 19 & 20 & 21 & 22 & 23 & 24 & 25 & 26 & 27 & 28 & 29 & 30 & $>$30 \\
\textbf{$\varphi_j$} & 9 & 9 & 6 & 10 & 10 & 11 & 5 & 3 & 3 & 5 & 4 & 8 & 3 & 3 & 2 & 94 \\
\hline
\end{tabular}
\caption{Histogram of the frequency counts in the Corbet butterfly dataset \cite{Corbet1941}.}
\label{tab:butterfly}
\end{table}

\noindent
\begin{minipage}[h]{0.53\textwidth}
    \vspace{0pt} 
    \centering
    \includegraphics[width=\linewidth]{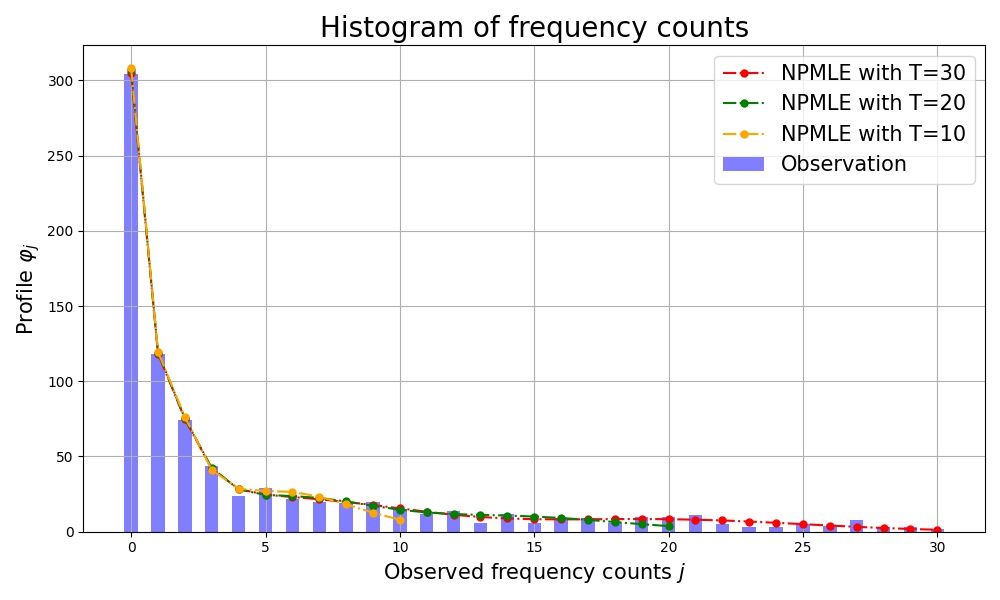}
    \captionof{figure}{Histogram and the NPMLE-fitted model.}
    \label{fig:butterfly}
\end{minipage}%
\hfill
\begin{minipage}[h]{0.4\textwidth}
    \vspace{0pt} 
    \centering
\begin{tabular}{c|c c}
  & Mixture & P-model \\
\hline
$T=10$ & 0.2215 & 1.86e-07 \\
$T=15$ & 0.6544 & 1.46e-05 \\
$T=20$ & 0.2954 & 9.57e-06 \\
$T=25$ & 0.9209 & 2.31e-04 \\
$T=30$ & 0.8778 & 6.28e-04 \\
\end{tabular}
\captionof{table}{$p$-values of the $\chi^2$ goodness of fit test at various models and truncation levels.}
    \label{tab:p_values}
\end{minipage}

\subsection{NPMLE under the P-model} 
\label{sec:mismatch}
In this subsection, we discuss the application of our methodology under the $P$-model.
Consider a $k$-atomic distribution $P=(p_1,\ldots,p_k)$\footnote{Here we focus on discrete distributions with support size no more than $k$. 
See Section~\ref{sec:pen-npmle} for further discussion on more general settings.}, and $N_i \sim q_n(\cdot,p_i)$.
The goal is to estimate the \textit{histogram distribution} of $P$ defined as
$$\pi_P \triangleq \frac{1}{k} \sum_{i=1}^k \delta_{p_i},$$ 
where $\delta_{x}$ denotes the Dirac measure at $x$.
From a Bayesian perspective, in the mixture formulation \eqref{eq:N_mixture}, each $p_i$ is then treated as an independent random effect drawn from the prior $\pi^\star$ without the normalization $\sum_i p_i = 1$.
Similar problem has been studied under the Gaussian sequence model, where the NPMLE is applied to estimate the density $f_{\pi_P}$ \cite{Zhang08}.

In this paper, we aim to show that the NPMLE in~\eqref{eq:npmle-def} yields reasonable estimates of $\pi_P$. 
For a quick insight, consider the expectation of the likelihood function
\begin{align}
\label{eq:npmle_opt_exp}
 \Expect  L(\pi;N) = 
 \sum_{i=1}^k \sum_{j=0}^\infty q_n(j,p_i) \cdot \log f_{\pi}(j) =    k \cdot \sum_{j=0}^\infty f_{\pi_P}(j) \log f_\pi(j),
\end{align}
which corresponds to the negative cross-entropy of \( f_\pi \) relative to \( f_{\pi_P} \) with unique maximizer \( f_\pi = f_{\pi_P}\). 
As the maximizer of $L(\pi; N)$, the NPMLE $\hat\pi$ is intuitively expected to be close to the maximizer of the expected log-likelihood $\Expect  L(\pi; N)$, which is $\pi_P$.
In Section~\ref{sec:theory}, we formalize this intuition via a different approach, and establish rigorous convergence guarantees for the Poisson NPMLE.

From the optimization perspective, the NPMLE also proves beneficial for fitting the $P$-model.
Due to the non-convex nature of the space of true underlying histograms $\pi_P$, where each probability mass is restricted to a multiple of $\frac{1}{k}$, directly optimizing the log-likelihood objective over this space is computationally challenging. In contrast, the NPMLE \eqref{eq:npmle-def} naturally provides a continuous relaxation into a convex program, allowing outputting a ``fractional histogram distribution'' whose probability masses can take continuous values. Also, the NPMLE  removes the normalization constraint that each histogram distribution has an expectation of \( \frac{1}{k} \) since the $p_i$'s sum to 1. In fact, the NPMLE is approximately self-normalized with its expectation converging to that of $\pi_P$ (see Theorem~\ref{thm:npmle-global}).

\subsection{Applications to functional estimation}
\label{sec:app_func}
We further explore the downstream task of symmetric functional estimation based on the observed multiset of frequencies.
Let $\calP$ be a family of distributions, and $G:\calP \mapsto \reals$ be a functional on it. $G$ is said to be Lipschitz continuous in $\calP$ under a metric $d$, if 
$$
|G(P) - G(P')| \leq d(P, P'), \quad \forall P, P' \in \mathcal{P}.
$$
This paper focuses on the following linear functional on $\calP([0,1])$:
\begin{align}
\label{eq:def_func}
    G(\pi) = \int g \, \diff\pi, \quad \pi \in \calP([0,1]),
\end{align}
where \( g: [0,1] \to \mathbb{R} \) is a given measurable function.
It follows that any such linear functional $G(\pi)$ is Lipschitz continuous under any \textit{integral probability metric} (IPM) with respect to any function class containing $g$; see Appendix~\ref{sec:ipm} for details. 

Particularly, set $\pi=\pi_P$ under the $P$-model. The resulting functional is a \textit{symmetric additive functional} of $P$, taking the form  
\begin{align}
\label{eq:def_func_P}
    G(P) = \sum_{i=1}^k g(p_i) = k \cdot \int g \, \diff\pi_P.
\end{align}
Typical examples include the  Shannon entropy $H(P) = \sum_{i=1}^k p_{i} \log \frac{1}{p_{i}}$, power-sum $F_\alpha(P) =  \sum_{i=1}^k p_{i}^\alpha$, $\alpha \in (0,1)$, and the support size $S(P) =  |\{i \in[k] \mid p_i > 0\}|$,  
which correspond to the functions $h(x)=-x\log x$,  $f_\alpha(x)=x^\alpha$, and $s(x)=\indc{x>0}$, respectively. See Section~\ref{sec:truncated_npmle} for further discussions.

A widely used strategy for estimating the symmetric additive functional~\eqref{eq:def_func_P} is the so-called \textit{plug-in approach}, which first obtains a histogram estimator $\hat{\pi}$, and then substitutes it into the functional to construct the desired estimator
\begin{align}
    \label{eq:F_hat}
    \hat G= k \cdot \int g \ \diff \hat\pi.
\end{align}
We consider the NPMLE plug-in estimator, where $\hat{\pi}$ is defined in \eqref{eq:npmle-def}, and demonstrate that it exhibits strong theoretical and practical performance in various functional estimation problems, particularly in the large-alphabet regime where $k$ grows with $n$. We preview that our NPMLE plug-in estimator has the following advantages: 
\begin{itemize}
    \item \textit{Flexibility}: Building on the general advantages of plug-in estimators, our approach provides a unified framework for statistical tasks, including estimating a broad class of symmetric functionals and characterizing the species discovery curve for determining the number of unseen species -- see Sections~\ref{sec:truncated_npmle} and~\ref{sec:exp_real} for details.
    \item \textit{Computational tractability}: The maximum likelihood estimation \eqref{eq:npmle-def} only involves a convex optimization program. As detailed in Section~\ref{sec:exp}, the NPMLE can be efficiently solved using standard convex optimization methods and softwares.
    \item \textit{Statistical efficiency}: According to classical asymptotic theory \cite{vdv00}, likelihood-based approaches typically achieve higher statistical efficiency compared to moment-based methods, such as polynomial  approximation methods and other expansion-based bias correction techniques.
\end{itemize}

The rest of the paper is organized as follows. Section~\ref{sec:preliminaries} introduces the Poisson model and key statistical properties of the NPMLE. Section~\ref{sec:theory} presents theoretical results on the convergence of the Poisson NPMLE and its variants, including localized and penalized formulations. Section~\ref{sec:exp} reports experiment results on synthetic data, real datasets, and large language models, demonstrating the accuracy and robustness of NPMLE-based estimators. Section~\ref{sec:discussion} concludes with extensions to broader settings. Additional details on background, proofs, and experimental are provided in the appendices.

\subsection{Related work}
\paragraph{Functional estimation}
Functional estimation plays a crucial role in statistics, computer science, and information theory, with broad applications across various disciplines. Entropy estimation, for instance, has been extensively applied in neuroscience \cite{strong1998entropy}, 
 physics \cite{de2022entropy},
and telecommunications \cite{plotkin1996entropy}.
See also \cite{PainskyEntropy2021} for a comprehensive review.
The problem of estimating support size and support coverage dates back to Fisher's seminal work \cite{fisher1943relation} on estimating the number of unseen species. Since then, it has  been explored in ecology \cite{chao1984nonparametric,bunge1993estimating,OSW16}, linguistics \cite{ET76,thisted1987shakespeare}, and database management \cite{li2022sampling}.
The $L_1$ distance between probability distributions is closely related to distribution testing problems \cite{Batu2001Test,canonne2020survey}.
More recently, a series of studies relate functional estimation problems to the analysis of language models for understanding their capacity and robustness \cite{farquhar2024detecting,nasr2025scalable,li2025evaluatingunseen}.

\paragraph{Plug-in and non-plug-in estimators}
The empirical distribution is the most commonly used choice for plug-in estimation. In the large-sample regime, its asymptotic efficiency and consistency are established in \cite{vdv00,AK01} under mild conditions on the functional. Basic refinements include first-order bias correction \cite{miller1955note}, the jackknife estimator \cite{BURNHAM1978}, the Laplace estimator \cite{SG96entropy}, and the James-Stein type estimator \cite{hausser09a}. More advanced methods model the histogram distribution for symmetric functional estimation, such as the fingerprint-based algorithm \cite{vv17}, moment matching program \cite{LMM,HJW20}, and the profile maximum likelihood (PML)  estimator \cite{OSVZ04}. The PML plug-in estimator is shown to achieve optimal sample complexity for various symmetric functionals in \cite{ADOS16}, with the result further refined in \cite{HO19,HS21}. Efficient convex relaxation algorithms for approximate PML computation are then explored in \cite{ACSS21,CSS22}.

We also briefly review several \textit{non-plug-in approaches} that aim to estimate the functional directly, without explicitly recovering the underlying distribution. A prominent example is the polynomial approximation method, which approximates the target functional by its best polynomial surrogate and constructs unbiased estimators for the resulting polynomial expression. This technique has been widely adopted to obtain minimax-optimal rates for a variety of functionals, including entropy \cite{Pan03, JVHW15, WY16}, support size \cite{WY15unseen}, power sum \cite{JVHW15}, support coverage \cite{OSW16}, total variation distance \cite{JHW18_L1}, and other Lipschitz functionals \cite{hao2019unified}, etc.\
Despite their strong theoretical guarantees, non-plug-in methods typically require functional-specific constructions that may limit their general applicability.
Also, the practical performance can be sensitive to hyperparameter choices (e.g., polynomial degree), and higher-order approximations often incur greater computational costs and risk of overfitting.
Other alternatives include Bayesian methods, which place a prior over the discrete distribution and compute the posterior distribution of the functional \cite{NSB01, archer2014bayesian}. More recently, neural network-based estimators have also been proposed for learning complex functionals from data \cite{shalev2022neural}.

\paragraph{Nonparametric maximum likelihood} 

Originally introduced by \cite{KW56}, the nonparametric maximum likelihood estimate (NPMLE) has been extensively studied during the past decades. Fundamental results on existence, uniqueness, and the discreteness of its support have been established in a series of works \cite{L78,L83,lindsay1993uniqueness}. Under the mixture model, \cite{KW56,Chen16Consistency} establish the consistency of NPMLE, and the asymptotic normality of functional plug-in estimators has been analyzed in \cite{sara1999applications}. See also \cite{L95} for a comprehensive review.
The convergence rate of NPMLE has also been extensively studied.  The Hellinger rate for density estimation has been developed for Gaussian mixtures \cite{GV01,Zhang08,MWY25} and for Poisson mixtures \cite{SW22,JPW25}. 
\cite{VKVK19,FP21} establish the minimax optimality of the NPMLEs for the Poisson and binomial mixtures under the 1-Wasserstein distance.

Various kinds of algorithms has been proposed for computing the NPMLE. The expectation-maximization (EM) algorithm is first proposed by \cite{L78} and further applied in \cite{JZ09}. Convex optimization algorithms are then considered, such as the interior point method \cite{KM13}  implemented by the \texttt{R} package \texttt{REBayes}  \cite{rebayes17},
and the minimum distance estimator \cite{JPW25} designed for Poisson NPMLE.
Delicate high-order optimization algorithms have also been developed for computing the NPMLE, including sequential quadratic programming (SQP) \cite{KIM2020NPMLE}, cubic regularization of Newton’s method \cite{wang2023nonparametric}, and the augmented Lagrangian method \cite{Zhang2022NPMLE}. These approaches demonstrate the ability to handle larger data sizes and broader value ranges while achieving higher accuracy compared to first-order methods. 
More recently, advanced techniques based on Wasserstein gradient flows are also developed \cite{yan2024flow}.

\subsection{Notation}
Let $[k] \triangleq \{1, \dots , k\}$ for $k \in \naturals$. 
Let $\Delta_{k-1}$ denote the collection of all probability measures with support size at most $k$.
For $x, y \in \reals$, $x \vee y \triangleq \max\{x, y\}$ and 
$x \wedge y \triangleq \min\{x, y\}$. 
Let $|I|$ denote the cardinality of $I$ if $I$ is countable, and the Lebesgue measure of $I$ if $I$ is uncountable.
Define $\calP(I)$ as the collection of all probability distributions that is supported on $I$. 
Let $\Bern(p)$, $\Bin(n,p)$, $\Poi(\lambda)$, and $\operatorname{Multi}(n,P)$ denote the Bernoulli distribution with
mean $p$, the binomial distribution with parameter $n,p$, the Poisson distribution with mean $\lambda$, and the multinomial distribution with parameters $n,P$, respectively. 
For a function \( g\) defined on \([0,1]\), we define the \( L_q \) norm as \(\|g\|_q \triangleq (\int_{[0,1]} |g(x)|^q \, \mathrm{d}x )^{1/q}\) for \( 1 \leq q < \infty \), the \( L_\infty \) norm as
\(\|g\|_{\infty} \triangleq \sup_{x \in [0,1]} |g(x)|\), and the truncated \( L_\infty \) norm on a subset $I\subseteq [0,1]$ as
\(\|g\|_{\infty, I} \triangleq \sup_{x \in I} |g(x)|\).
For two positive sequences ${a_n}$ and ${b_n}$, write $a_n \lesssim b_n$ or $a_n = O(b_n)$ when $a_n \leq Cb_n$ for some absolute constant $C > 0$, $a_n \gtrsim b_n$ or $a_n = \Omega(b_n)$ if $b_n \lesssim a_n$, and $a_n \asymp b_n$ or $a_n =\Theta (b_n)$ if both $b_n \gtrsim a_n$ and $a_n \gtrsim b_n$ hold. 
We write $a_n=O_\alpha(b_n)$ and $a_n \lesssim_\alpha b_n$ if $C$ may depend on parameter $\alpha$.

\section{The Poisson regime}
\label{sec:preliminaries}
In this section, we introduce the Poisson model, a widely used framework for analyzing frequency counts. We then study the corresponding Poisson NPMLE and present its key properties, including optimality conditions and statistical guarantees.

\subsection{Counting with Poisson processes}
\label{sec:poisson_npmle}
The Poisson model, also known as Poisson sampling, is a widely used framework for modeling frequency counts in scenarios such as customer arrivals \cite{Hall2007} and animal trapping \cite{fisher1943relation,OSW16}. 
When covariates are available, Poisson regression models the event rate as a function of these variables for purposes such as prediction or smoothing. In contrast, counting processes model events over time when only counts and event times are observed, among which the Poisson process assumes that the number of events in a given time interval follows a Poisson distribution.
Let $P=(p_1,\dots,p_k)$ be the normalized intensities of $k$ categories with $\sum_{i=1}^k p_i=1$ such that there is on average one arrival per unit time. 
Then, the frequency counts over $n$ units of time are distributed as
\begin{align}
    \label{eq:poisson-sampling}
    N_i \inddistr \Poi(n p_i), \quad \forall i \in [k].
\end{align}

Conditioned on the total number of counts $n'=\sum_{i} N_i$, the vector $N =(N_1,\dots,N_k)$ follows the $\operatorname{Multi}(n',P)$ distribution, which is equivalent to the i.i.d.\ sampling model from $P$.
The minimax risk under Poisson sampling is provably close to that under a fixed sample size across a wide range of distributional and functional estimation problems (see, e.g., \cite{JVHW15,WY16,LMM,HS21}).
In this paper, we refer to $n$ as the \textit{sample size} and $k$ as the \textit{alphabet size}.

\subsection{Basic properties of the Poisson NPMLE}
\label{sec:npmle-basic}
Under the Poisson model with $q_n(x,r)=\poi(x, nr)$ in \eqref{eq:N_mixture}, we consider the Poisson NPMLE 
\begin{align}
     \label{eq:pois-npmle-def}
    \hat{\pi} \in \argmax_{\pi \in \calP([0,1])} \sum_{i=1}^k \log f_{\pi} (N_i),
\end{align}
where $f_{\pi}$ is defined in \eqref{eq:N_mixture} with $q_n(\cdot, r)=\poi(\cdot, nr)$.

\paragraph{Existence and uniqueness}
Poisson NPMLE enjoys favorable properties such as the uniqueness
of the solution, whereas in other mixture models (\eg, the binomial mixture), the mixing distribution $\pi$ is not necessarily identifiable from $f_\pi$ \cite{Teicher63}.
The following proposition that is implied by {\cite[Theorem 1]{JPW25}} establishes the existence and uniqueness of the Poisson NPMLE. 
This result builds upon earlier work by \cite{L83} and involves a detailed analysis of the Poisson probability mass function.
\begin{proposition}
\label{prop:npmle-1}
     Let $\hat p_i\triangleq N_i/n$. The solution $\hat{\pi}$ in \eqref{eq:pois-npmle-def} exists uniquely and is a discrete distribution with support size no more than the number of distinct elements in $\{N_i\}_{i=1}^k$.
     In addition, $\hat{\pi}$ is supported on $[1 \wedge \min_{i\in [k]} \hat p_i, 1 \wedge \max_{i\in [k]} \hat p_i]$.
\end{proposition}

\paragraph{Optimality conditions}
By definition of $\hat\pi$ in~\eqref{eq:pois-npmle-def}, for any feasible $Q\in\calP([0,1])$, we have
\begin{equation}
    \label{eq:0th_opt}
    \sum_{i=1}^k \log \frac{f_{\hat \pi} (N_i)}{f_{Q} (N_i)}\ge 0.
\end{equation}
Letting $\pi_N \triangleq \frac{1}{k}\sum_{i=1}^k \delta_{N_i}$, we rewrite the likelihood function \eqref{eq:Ni-likelihood} as $L(\pi,N)= -k ( H(\pi_N) + \KL(\pi_N\|f_\pi) )$,
where $H(p)\triangleq \Expect_p \log \frac{1}{p}$ denotes the Shannon entropy and $\KL(p\|q) \triangleq \Expect_p \log \frac{p}{q}$ denotes the Kullback–Leibler (KL) divergence. 
Then, the NPMLE can be equivalently formulated as 
\begin{align}
     \label{eq:npmle-KL}
    \hat{\pi} \in \argmin_{\pi \in \calP([0,1])} \ \KL(\pi_N\|f_\pi).
\end{align}
This also provides a minimum-distance interpretation of the NPMLE, which aligns the empirical histogram $\pi_N$ with a smoothed density $f_\pi$ of bandwidth $O(\tfrac{1}{\sqrt{n}})$ under the $\KL$ divergence.

Next, we turn to the first-order optimality conditions.
For any $Q\in \calP([0,1])$, it follows from the zeroth-order optimality~\eqref{eq:0th_opt} that the directional derivative of the log-likelihood function at $\hat\pi$ in the direction of $Q$ is always non-positive:
\begin{align}
    \label{eq:1st_opt}
    D_{\hat \pi}(Q) \triangleq \lim_{\epsilon\to 0_+} \frac{L((1-\epsilon)\hat \pi + \epsilon Q)-L(\hat \pi) }{\epsilon} 
    = \frac{1}{k} \sum_{i=1}^k  \frac{ f_{Q}(N_i) }{ f_{\hat \pi}(N_i) }  -1 \leq 0.
\end{align}
Another useful necessary condition is that the NPMLE $\hat\pi$ is always an ascending direction:
\begin{align}
     \label{eq:1st_opt_2}
    D_Q(\hat \pi) 
    &= \frac{1}{k}\sum_{i=1}^k \frac{f_{\hat \pi}(N_i)}{f_{Q}(N_i)} -1 
    \stepa{\geq} \pth{\prod_{i=1}^k \frac{f_{\hat \pi}(N_i)}{f_{Q}(N_i)}}^{1/k} -1 
    \stepb{\geq} 0,
\end{align}
where (a) uses the AM-GM inequality,
and (b) follows from~\eqref{eq:0th_opt}.

\paragraph{Statistical properties}
In the following, we establish statistical properties for the Poisson NPMLE based on its optimality conditions, which  play a key role in proving the main results in Section~\ref{sec:theory}.
To begin with, let $r:[0,1] \mapsto [0,\infty)$ be a  nonnegative function. 
For a set $S \subseteq [0,1]$, define the \textit{$r$-fattening} of $S$ as
$$S_{r}\triangleq \bigcup_{x \in S} [x - r(x), x + r(x)].$$
In particular, if $r(x) \equiv r$ is constant, this reduces to the standard notion of fattening using a fixed radius $r$.

\begin{definition}[$r$-separation]
Two sets $S,S'\subseteq[0,1]$ are said to be \textit{$r$-separated} if $S_r \cap S'_r = \emptyset$. 
In particular, we define the \textit{$r$-complement} of $S$ as $S^{c,r} \triangleq \cup_{S'\subseteq [0,1]: S_r \cap S'_r = \emptyset} S'$, which is the largest subset of $[0,1]$ that is $r$-separated from $S$.
\end{definition}

We say that the radius function $r$ is \textit{$t$-large} if
\begin{equation}
    \label{eq:r_condition}
    \inf_{x\in [0,1]} \frac{r^2(x)}{x} \wedge r(x) \geq t.
\end{equation}
Equivalently, if the function $r$ is $t$-large, then $r(x)\ge t\vee \sqrt{tx}$ and thus $r(x)\ge \sqrt{t(x \vee r(x))}$ for all $x\in[0,1]$. 
Under this condition, $S_r$ characterizes the high-probability region of the Poisson distribution with parameter in $S$ (see Lemma~\ref{lem:pois_rn}).
The following proposition controls the support and probability mass of the Poisson NPMLE.

\begin{proposition}
\label{prop:weight_remainder_asymp}
There exist universal constants $C, c, c_0>0$ such that, for any $t$-large function $r:[0,1] \mapsto [0,\infty)$ with $t\geq C\frac{\log k}{n}$,     
with probability at least $1 - 2k\exp(-c_0nt)$, the following holds for all measurable sets $S\subseteq [0,1]$:
\begin{enumerate}
\item[(a)] $\hat \pi(S_r) \geq \pi_P(S)/(1+\exp(-cnt)) $;
\item[(b)] $\hat \pi(S_r) \leq 1-\pi_P(S^{c,r}) + \exp(-cnt)$;
\item[(c)] $\hat \pi (S_r ) = 1$ if $\pi_P(S) = 1$.
\end{enumerate}
\end{proposition}

Proposition~\ref{prop:weight_remainder_asymp} characterizes both the local and global behavior of the Poisson NPMLE. 
Part (a) shows that the NPMLE assigns at least $\pi_P(S)$ up to an exponentially small error term within a neighborhood $S_r$ of $S$. 
Conversely, part (b) upper bounds $\hat\pi(S_r)$ by the probability mass of $\pi_P$ on a larger set $ (S^{c,r})^c \supseteq S_r$ up to an error term.
Combining (a) and (b) implies that $\hat\pi$ nearly matches the mass of $\pi_P$ around $S$.
Finally, part (c) strengthens this result  by removing the error terms in the special case where $S$ is the full support, showing that $\hat\pi$ concentrates around the support of $\pi_P$ with high probability.
The proof constructs a high probability event on which these statistical properties are necessary to satisfy the optimality conditions \eqref{eq:1st_opt} and \eqref{eq:1st_opt_2} for all testing distributions $Q$.
The full proof is deferred to Appendix~\ref{app:npmle-weight}.

\begin{remark}
While Proposition~\ref{prop:weight_remainder_asymp} holds for $\hat \pi$ under \eqref{eq:poisson-sampling}, it fails under the mixture model $N_i \iiddistr \int \mathrm{Poi}(n\theta) \diff \pi^\star(\theta)$ when $\pi_P$ is replaced by $\pi^\star$.
To illustrate, consider $k=2$ and $\pi_P = \pi^\star = \tfrac{1}{2}\delta_{1/3} + \tfrac{1}{2}\delta_{2/3}$.
Under the mixture model, both $N_1$ and $N_2$ are drawn from the same component with probability $0.5$, in which case the NPMLE concentrates near a single point by Proposition~\ref{prop:npmle-1} and thus deviates substantially from $\pi^\star$.
\end{remark}

The next proposition provides an upper bound on the Hellinger distance between the mixture densities of the Poisson NPMLE, which directly follows from the density estimation result in \cite[Proposition 27]{SW22}.
\begin{proposition}
\label{prop:pois-hellinger-rate}
Let $\{N_i\}_{i=1}^k$ be drawn from the Poisson model \eqref{eq:poisson-sampling}, and
\begin{align*}
    \epsilon_{n,k}^2= \frac{n^\frac{1}{3} \log^8k }{k} \wedge 1.
\end{align*}
Then, there exist constant $s^\star>0$ such that for any $s \geq s^\star$,  
\begin{equation}
    \label{eq:npmle-rate}
        \pbb[H(f_{\hat{\pi}}, f_{\pi_P} )\geq s\epsilon_{n,k}]\leq 2\exp\pth{-s^2\log^2 k /8}.
\end{equation}
\end{proposition}


The (squared) Hellinger risk is a commonly used measure in density estimation problem.
The proof follows the classical metric entropy approach for analyzing $M$-estimators, which is applied to the NPMLE by constructing finite mixture approximation \cite{Zhang08,MWY25}.
In our setting, the alphabet size $k$ corresponds to the number of input counts in the NPMLE, while the sample size $n$ serves as a bandwidth scaling parameter of the Poisson distributions. 
Accordingly, in Proposition~\ref{prop:pois-hellinger-rate}, the Hellinger risk decreases with $k$ but grows with $n$.
The non-vanishing error in the fixed-$k$, large-$n$ regime is not an artifact of the analysis. 
Indeed, Proposition~\ref{prop:pi_p_minimax_lb} establishes a minimax lower bound on the Hellinger risk, which remains bounded away from zero for constant $k$ even as $n \to \infty$.
Intuitively, this is because the standard deviation of each Poisson distribution in the mixture model is of the same order as the estimation error of the mean parameters.

\begin{proposition}
\label{prop:pi_p_minimax_lb}
There exist universal constants $c,C>0$ such that for any $n\ge C \log k$,
\begin{align*}
 \inf_{\hat f}\sup_{P\in \Delta_{k-1}} \Expect H^2 (\hat f, f_{\pi_P}) \geq \frac{c}{k},
\end{align*}
where the infimum is over all $\hat f$ measurable with respect to $\{N_1,\ldots,N_k\}$.
\end{proposition}

Nevertheless, as we will show in the remaining sections, the NPMLE $\hat\pi$ remains meaningful despite the impossibility of consistent density estimation.
Figure~\ref{fig:cdf_asymp} provides a quick insight that $\hat\pi$ closely estimates $\pi_P$ in the sense of the cumulative distribution function (CDF) when $k$ is fixed and the sample size $n$ is large. 
In Section~\ref{sec:theory}, we establish theoretical guarantees under the 1-Wasserstein distance between mixing distributions.
The results can be further extended to the $p$-Wasserstein distance and the general \textit{integral probability metric} (IPM) \cite{muller1997integral}, which serves as a foundation for many functional estimation problems.

\begin{figure}[H]
    \centering
    \includegraphics[width=0.5\linewidth]{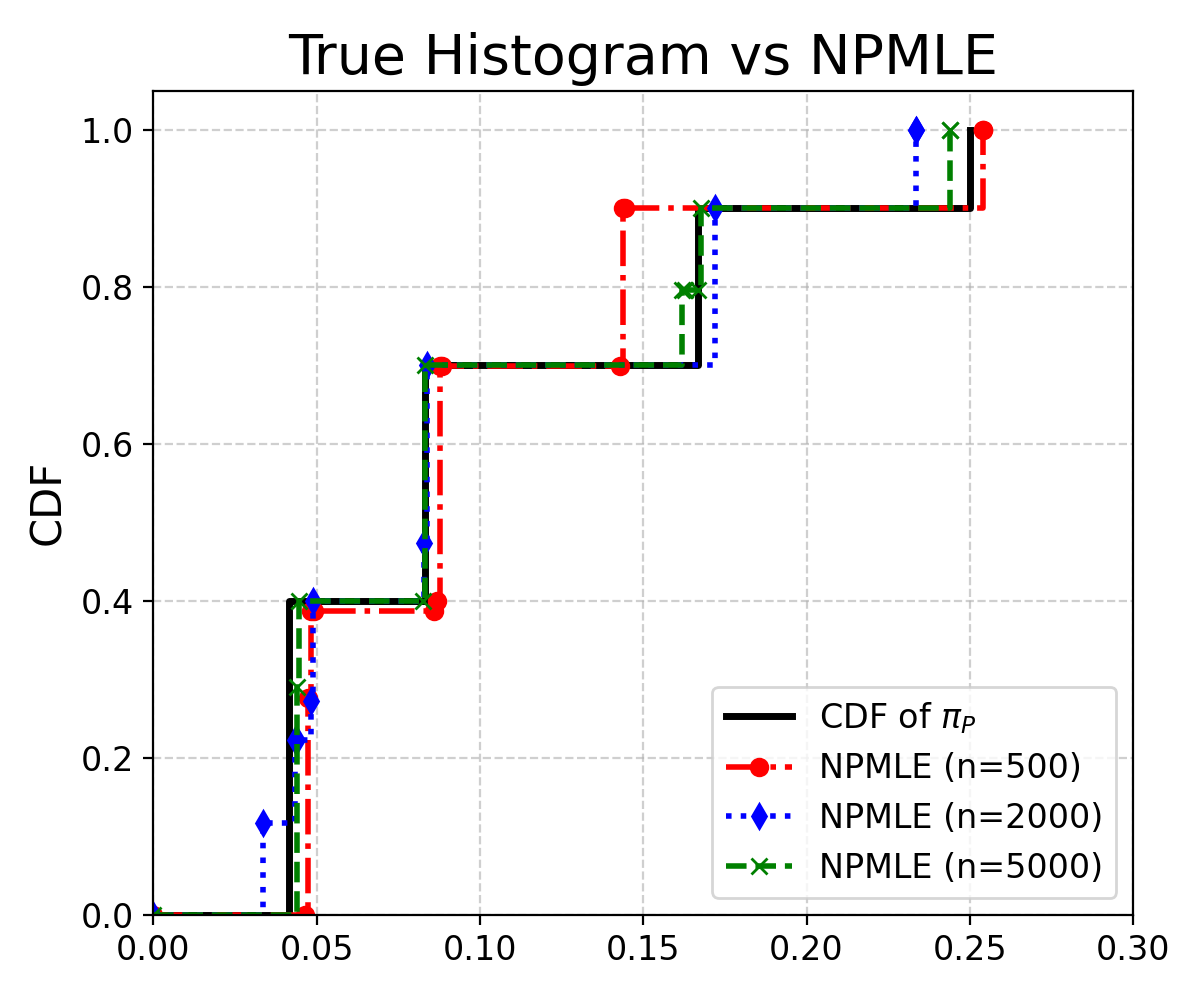}
    \caption{CDFs of the underlying distribution $\pi_P=\frac{1}{10}(4\delta_{\frac{1}{24}}+3\delta_{\frac{1}{12}}+2\delta_{\frac{1}{6}}+\delta_{\frac{1}{4}})$ with $k=10$ and the NPMLE fitted with $n=500, 2000, 5000$. The figure illustrates that the NPMLE assigns nearly the same probability mass as $\pi_P$ around each of its atom.
    }
    \label{fig:cdf_asymp}
\end{figure}

\section{Theoretical guarantees of the Poisson NPMLE}
\label{sec:theory}
In this section, we establish theoretical guarantees for the Poisson NPMLE \eqref{eq:pois-npmle-def}, building on the properties developed in the preceding discussion. We first show that the estimator achieves the parametric rate of asymptotic convergence under the $1$-Wasserstein distance.
Section~\ref{sec:theory-main} then investigates the non-asymptotic regime where the alphabet size grows with the sample size, showing that the NPMLE attains the minimax optimal rate. Section~\ref{sec:truncated_npmle} then addresses the estimation of symmetric functionals, where we combine the NPMLE plug-in estimator with a tailored bias-correction scheme to construct minimax rate-optimal estimators for specific functionals.
Finally, Section~\ref{sec:pen-npmle} develops a penalized version of the NPMLE to handle the case of unknown support size.

\subsection{Asymptotic rate of convergence} 
\label{sec:theory-asymp}
To begin with, we consider the regime where $P=(p_1,\dots,p_k)$ is fixed and establish an asymptotic guarantee for the Poisson NPMLE \eqref{eq:pois-npmle-def}. In particular, we focus on convergence in the $1$-Wasserstein distance \cite[Chapter~1]{villani03}, defined by
$$ 
W_{1}(P, P')
\triangleq \inf \{ \Expect |X-Y|: X\sim P, Y \sim P' \}. 
$$
 The next theorem shows that the estimator converges to the true histogram at the standard parametric rate under the $W_1$ distance.
\begin{theorem}
\label{thm:npmle-asymp}
    Fix $P=(p_1,\ldots,p_k)\in \Delta_{k-1}$. Let $\hat{\pi}$ be the NPMLE in \eqref{eq:pois-npmle-def}.
    Then, as $n\to\infty$,
    $$
    W_1(\hat\pi,\pi_P) = O_p\pth{\sqrt{\frac{1}{n}}}.
    $$
\end{theorem}

The proof of Theorem~\ref{thm:npmle-asymp} applies the quantile coupling formula \cite[Eq.\ (2.52)]{villani03} that expresses the Wasserstein distance in terms of differences between the quantile functions:
\[
W_{1}(\hat\pi,\pi_P)= \int_{0}^{1}\left|\hat Q (u)-Q_P(u)\right| \diff u,
\]
where $\hat Q$ and $Q_P$ denote the quantile functions of $\hat\pi$ and $\pi_P$, respectively.
The difference is then bounded by applying Proposition~\ref{prop:weight_remainder_asymp}, which implies that around each atom of $\pi_P$ the NPMLE assigns nearly the same probability mass within a neighborhood of length $\Theta(1/\sqrt{n})$ with high probability, as illustrated in Figure~\ref{fig:cdf_asymp}. 
The complete proof is provided in the Appendix~\ref{app:prf_thm_npmle-asymp}.
Using a similar quantile coupling formula, Theorem~\ref{thm:npmle-asymp} further extends to the $q$-Wasserstein distance defined as 
    $$ 
    W_{q}(P, P')
\triangleq \inf \{ (\Expect |X-Y|^q)^{1/q}: X\sim P, Y \sim P' \}, \quad q\ge 1.
    $$
Following the proof of Theorem~\ref{thm:npmle-asymp}, we obtain that $W_{q}(\hat\pi,\pi_P) \leq O_p(1/\sqrt{n})$, \ie, the same convergence rate  applies to the stricter $W_q$ distance for any constant $q \geq 1$.

\subsection{Non-asymptotic rate of convergence on large alphabet}
\label{sec:theory-main}

In this subsection, we move beyond the large-sample regime and consider the setting where the alphabet size $k$ grows with $n$. 
We focus on scenarios where categories can be grouped into subclusters in which occurrence probabilities (and thus frequency counts) are closely aligned. 
Such structures commonly arise in real-world datasets. 
For example, in species databases, abundances often follow hierarchical patterns reflecting positions in the food web \cite{Cohen2003food}; similarly, in statistics, co-citation and co-authorship networks \cite{Ji2022dataset} form multi-level hierarchical communities, where individuals at different levels exhibit distinct citation and collaboration counts.

Motivated by the subgroup structure, we investigate the performance of NPMLE under the specific  assumption of the underlying distribution. 
Consider the radius function 
\[
r_t^\star(x) \triangleq \sqrt{t x}+t,
\]
which by definition is $t$-large.
By Proposition~\ref{prop:weight_remainder_asymp}, the $r_t^\star$-fattening set captures the high-probability region of the Poisson model.  
Specifically, any two $r_t^\star$-separated points 
$p,p'\in(0,1)$  are \textit{heterogeneous} in the sense that $\TV(\Poi(np),\Poi(np'))\geq 1-\exp(-\Omega(nt))$. 
In contrast, they are \textit{homogeneous} when the probability masses are close (\eg, $p'\in \{p\}_{r_t^\star}$ with $nt\lesssim 1$), since the high-probability regions of $\Poi(np)$ and $\Poi(np')$ largely overlap.
To capture the subgroup structure, we consider the following assumption:
\begin{assumption}
\label{ass:separation}
There exists $q_1,\dots,q_L \in [0,1]$ that are distinct and pairwise $r^\star_t$-separated such that $\pi_P$ is supported on $\cup_{\ell=1}^L \{q_\ell\}_{r^\star_s}$ for some $t>s>0$.
\end{assumption}

Under Assumption~\ref{ass:separation}, the support of $P$ is partitioned into $L$ subgroups with the cluster centroid $q_\ell$ for each. 
Notably, Assumption~\ref{ass:separation} captures the emergence of categories with vanishingly small masses, a phenomenon that poses fundamental challenges for various large-alphabet problems. 
In particular, probability masses below $O(\tfrac{\log n}{n})$ often correspond to unseen categories with limited sample size, and thus constitute the hard instances in functional estimation \cite{WY16,JHW18_L1,WY15unseen} and histogram estimation \cite{vv17,LMM}. 
Addressing these problems typically requires tailored techniques, such as polynomial approximation and carefully designed linear programs. This regime is explicitly covered by Assumption~\ref{ass:separation}, where such small masses are covered by the subgroup with $q=0$ and $s,t \asymp \tfrac{\log n}{n}$.

\begin{theorem}
\label{thm:npmle-global}
Suppose that $n \geq \Omega(\tfrac{k}{\log k})$. 
There exist universal constants $C,C',c_0$ such that, for any $P\in\Delta_{k-1}$ satisfying Assumption~\ref{ass:separation} with 
$s=\tfrac{c_0\log n}{n}$ and $t = \tfrac{C\log n}{n}$, 
\begin{align}
    \label{eq:rate_global_W1_2}
     \Expect W_{1}(\hat\pi,\pi_{P}) 
     \leq C' 
     \sqrt{\frac{\log n}{kn}}\frac{1}{\log_+(\frac{k/\log^{3}n}{L\wedge n^{1/3}} ) },
\end{align}
where $\log_+(x)\triangleq 1\vee \log x$.
\end{theorem}

Theorem~\ref{thm:npmle-global} shows that NPMLE $\hat\pi$ attains the minimax lower bound (see \cite[Theorem 23]{LMM}) of estimating $\pi_P$ under the $W_1$ distance, where the worst-case distributions are covered by Assumption~\ref{ass:separation}.
Specifically, we consider the following regimes:
\begin{itemize}
    \item \textit{Large-sample and large-cluster-count regime.} When $k \lesssim (L\wedge n^{1/3}) \log^{3}n$,  \eqref{eq:rate_global_W1_2} provides an upper bound of $O(\sqrt{\tfrac{\log n}{kn}})$, which is optimal up to a logarithmic factor in $n$ compared with the minimax rate and the asymptotic rate $O(n^{-1/2})$ in Theorem~\ref{thm:npmle-asymp}. 
    
    \item \textit{Large-alphabet regime.}  The logarithmic factor in \eqref{eq:rate_global_W1_2} becomes effective as $k$ exceeds $(L \wedge n^{1/3}) \log^{3}n$. In particular, if $n\log n \gtrsim k \gtrsim (L \wedge n^{1/3})n^{\epsilon}$ for some $\epsilon>0$, the optimal rate $\Theta(\sqrt{\tfrac{1}{kn\log n}})$ is achieved, which improves upon the empirical histogram $\pi_{\hat P}\triangleq \tfrac{1}{k}\sum_{i=1}^k \delta_{\hat p_i}$ satisfying 
    \[
    \Expect W_1(\pi_{\hat P}, \pi_P)\le \frac{\Expect \Norm{\hat P-P}_1}{k} \le \sqrt{\frac{1}{kn}}.
    \]
    Hence, the empirical histogram is rate optimal only when all the $p_i$'s are heterogeneous and the underlying probability masses can be grouped into $L \approx k$ subclusters. 
    
    \item \textit{Trivial regime.} 
    Note that $W_1(\pi_P,\pi_Q) \leq \|P-Q\|_1/k \leq 1/k$ via the naive coupling between $\pi_P$ and $\pi_Q$.
    When $n \leq o(k/\log k)$, no estimator can achieve an error of $o(1/k)$. 
    Theorem~\ref{thm:npmle-global}  recovers the optimal sample complexity $\Theta(\tfrac{k}{\log k})$.
\end{itemize}

The proof of Theorem~\ref{thm:npmle-global} proceeds as follows.
First, by the dual representation of $W_1$ distance \cite[Theorem 1.14]{villani03}, 
it suffices to uniformly upper bound the plug-in estimation error of the NPMLE for 1-Lipschitz functions: 
\begin{align}
\label{eq:W1_dual}
W_1(\hat\pi,\pi_{P})
=\sup_{g \in\calL_1}\Expect_{\hat\pi} g- \Expect_{\pi_{P}} g,
\end{align}
where $\calL_1$ denotes the class of 1-Lipschitz functions.
We employ a Poisson deconvolution and construct a Poisson approximation taking form $\hat g (x) = a + \sum_{j} b_{j} \poi(j, n x)$, and decompose the error as
\begin{align*}
  \Expect_{\pi_{P}} g - \Expect_{\hat\pi} g = \int \hat g (\diff\pi_{P}-\diff {\hat\pi}) +  \int (g-\hat g)(\diff\pi_{P}-\diff {\hat\pi}) .   
\end{align*}
The first term is at most 
\begin{align*}
    \abs{\int \hat g (\diff \pi_{P}-\diff \hat\pi)} \leq \sum_{j=0}^{\infty} \abs{ b_{j}  (f_{\pi_{P}}(j)-f_{\hat\pi}(j)) } \leq \max _{j}\abs{ b_{j} } \| f_{\pi_{P}}-f_{\hat\pi}\|_1 \leq \max _{j}\abs{ b_{j} }  2H(f_{\pi_{P}},f_{\hat\pi}),
\end{align*}
where the density estimation error $H(f_{\pi_{P}},f_{\hat\pi})$ can be derived similar to Proposition~\ref{prop:pois-hellinger-rate} in each subgroup. 
Similar Poisson deconvolution has been used in \cite{VKVK19,FP21}, while our framework further reveals an interesting connection between density estimation and the estimation of $\pi_P$.
The $\log^3 n$ term in \eqref{eq:rate_global_W1_2} arises from the logarithmic factor in the Hellinger rate (see Lemma~\ref{lem:pois-hellinger-rate-c}) and is not optimized.
In particular, while $f_{\hat\pi}$ is fundamentally inconsistent for constant $k$, the error of $\hat\pi$ is weighted by $|b_j|$ that is proportional to the subgroup width. 
Moreover, in contrast to approximation-based approaches such as \cite{LMM,HS21}, which explicitly incorporate polynomial approximations and requires estimating higher-order moments, we apply polynomial approximation only implicitly through the analysis. 
The complete proof is presented in Appendix~\ref{app:prf_thm_npmle-global}.

By allowing $g$ to range over a functional class $\calF$ rather than the $\calL_1$ class in \eqref{eq:W1_dual}, the analysis naturally generalizes to distance measures in the integral probability metric (IPM) family \cite{muller1997integral} (see Appendix~\ref{sec:ipm}). This allows us to extend the histogram estimation guarantees to functional estimation problems, which is the central focus of Section~\ref{sec:truncated_npmle}.

\begin{remark}
To remove Assumption~\ref{ass:separation} and obtain theoretical guarantees for general distributions, one idea is to apply the localization argument: 
1) localize the subgroup of each probability mass using an independent sample;
2) solve the local NPMLE using the frequency counts in each subgroup;
3) analyze the local NPMLE and aggregate the estimators.
Similar ideas have been used to construct rate-optimal estimators through localized linear programs \cite{LMM} and piecewise polynomial approximation \cite{HO19}. 
In Section~\ref{sec:truncated_npmle}, we adopt localization for small masses in functional estimation. 
In practice, however, the performance of the localized methods depends on the tuning of additional parameters. 
A unified theory for the vanilla NPMLE without the separation condition is left for future work.
\end{remark}

\subsection{Symmetric functional estimation via the localized NPMLE}
\label{sec:truncated_npmle}
In this subsection, we focus on the problem of symmetric functional estimation introduced in Section~\ref{sec:app_func}, aiming at estimating the target functional
in \eqref{eq:def_func_P}:
$$G(P) = \sum_{i=1}^k g(p_i) = k \cdot \int g \, \diff\pi_P .$$
In the large-alphabet regime with many small probability masses, a major challenge in functional estimation arises when the target functional is non-smooth or even singular near zero.
To address this, we introduce a localized NPMLE plug-in estimator.
The proposed estimator consists of two parts.
For small probability masses, we solve the Poisson NPMLE \eqref{eq:pois-npmle-def} using only the subgroup with small frequency counts, and then construct the corresponding plug-in estimator as in \eqref{eq:F_hat}. 
For large frequency counts, we employ the empirical distribution with a bias correction. 
As we will show next, the localized NPMLE plug-in estimator attains minimax optimal rate for estimating a broad class of functionals.

To begin with, suppose that we observe two independent samples of frequency counts \( N = (N_1, \ldots, N_k) \) and \( N' = (N_1', \ldots, N_k') \) with \( N_i, N_i' \iiddistr \Poi(n p_i) \).
Following the formulation in Section~\ref{sec:poisson_npmle}, the two samples can be obtained via the thinning property (see, \eg, \cite[Sec.~3.7.2]{durrett2019probability}) of the Poisson process with observations over $2n$ units of time.

\paragraph{Localized NPMLE}
Consider the subgroup $I \triangleq \{0\}_{r_t^\star}=[0,t]$ with $t=C \tfrac{\log n}{n}$. 
The set $I$ corresponds the region of small probability masses that account for the unseen domain elements.
The second independent sample $N'$ is used to localize the masses based on a Poisson tail bound (see Lemma~\ref{lem:pois_rn}).
The localized NPMLE on $I$ is then estimated using the first sample $N$:
\begin{align}
     \label{eq:npmle-trunc-def}
    \hat{\pi}_I  = \argmax_{\pi\in \calP([0,1]) } & \sum_{i\in\calJ} \log f_\pi(N_i),
    \qquad 
    \calJ= \{i:\hat p_i' \in I\}.
\end{align}
By independence, conditioning on $N'$, the convergence of $\hat{\pi}_I$ to $\pi_{P,I}\triangleq \frac{1}{|\calJ|}\sum_{i\in \calJ}\delta_{p_i}$ following an analysis analogous to that of Theorem~\ref{thm:npmle-global}. 
Theorem~\ref{thm:npmle-truncated} in Appendix~\ref{app:prf_truncated_np} further establishes an upper bound on the general integral probability metric between $\pi_{P,I}$ and $\hat{\pi}_I$.

\paragraph{Bias-corrected estimator}
For large frequency counts, we apply the empirical plug-in estimator with first-order bias correction. 
Intuitively, for a smooth function $g:[0,\infty) \mapsto \reals$, the Taylor expansion at $p_i$ implies that
\begin{align*}
    \Expect g(\hat p_i)- g(p_i) = \frac{\var[\hat p_i]}{2}g''(p_i) + O(n^{-2}) = \frac{p_i}{2n}g''(p_i) + O(n^{-2}).
\end{align*}
 The bias-corrected estimator of $g$ is defined as 
\begin{align}
\label{eq:g_debias}
    \tilde g(x) = 
    \begin{cases}
        g(x)- \frac{x}{2n}g''(x), & x >0,\\
        g(0), & x=0.
    \end{cases}
\end{align}
For instance, when $g= x\log \frac{1}{x}$, $\tilde g=g+\frac{1}{2n}$ is the Poisson analogue of the well-known Miller-Madow estimator \cite{miller1955note}. 

\paragraph{Combine the estimators}
Given the index set $\calJ$, we can partition the functional $G$ as
\begin{align*}
    G(P)= \sum_{i\in \calJ} g(p_i) +  \sum_{i\in[k]\setminus\calJ} g(p_i) \triangleq G_1(P) + G_2(P).
\end{align*}
We apply the NPMLE to the frequency counts with indices \( i \in \calJ \) to estimate \( G_1(P) \), and use the bias-corrected plug-in estimator for the remaining indices to estimate \( G_2(P) \):
\begin{align}
\label{eq:G-combined}    
    \tilde G \triangleq |\calJ| \cdot \Expect_{\hat\pi_I} g + \sum_{i\in[k]\setminus\calJ} \tilde g(\hat p_i).
\end{align}
When the additional knowledge that $G(P)$ takes value in $[\underline{G}, \overline{G}]$ is available, the final estimator is then defined as $\hat{G} \triangleq (\tilde G \wedge \Bar{G})  \vee \underline{G}.$
This two-part structure aligns with the design of various approximation-based estimators (e.g., \cite{CL11, WY16, JHW18_L1}), where small frequency counts are handled by polynomial-based estimators. In contrast, our use of the NPMLE improves both stability and flexibility without the need for explicit high-order polynomial constructions.

Next, we apply the proposed estimator for specific symmetric functionals.
Consider the Shannon entropy $H(P) = \sum_{i=1}^k  h(p_i)\in [0,\log k]$ with $g(x)=h(x)\triangleq x \log \frac{1}{x}$, and the estimator 
$\hat H =  (\tilde H \wedge \log k) \vee 0$ with $\tilde H=\tilde G$ defined in \eqref{eq:G-combined}.

\begin{theorem}
\label{thm:np_trunc_H_estimate}
Suppose that $\log n \geq \Omega(\log k)$.
There exist a universal constants $C'$ such that, for any $P\in\Delta_{k-1}$,
\begin{align*}
\Expect |\hat{H}-H(P)| &\leq C'\pth{ \frac{k}{n\log n} + \frac{\log n}{\sqrt{n}}}.
\end{align*}
\end{theorem}

The approach that combines a histogram-based plug-in estimator for small probability masses with an empirical plug-in estimator for large probability masses was proposed in \cite{vv17}, which achieves an additive error of $\epsilon$ with a sample size of $\Theta(\frac{k}{\epsilon^2\log k})$.
In comparison, Theorem~\ref{thm:np_trunc_H_estimate} shows that combining the NPMLE plug-in estimator with a bias-corrected estimator attains the optimal sample complexity $\Theta(\frac{k}{\epsilon\log k})$.

To sketch the proof, we first decompose the estimation error of $\tilde H$ as 
\begin{align*}
    \tilde H - H(P) = |\calJ| (\Expect_{\hat\pi_I} h- \Expect_{\pi_{P,I}} h)  + \sum_{i\in[k]\setminus\calJ} \pth{\tilde h(\hat p_i) - h(p_i)},
\end{align*}
where $\tilde h$ is defined in \eqref{eq:g_debias} with $g = h$. Conditioning on $\calJ$, we control the first term using the uniform bound of the integral probability metric (see Theorem~\ref{thm:npmle-truncated}), and the second term by the bias-correction design. 
It turns out that  $|\tilde H - H|$ can be bounded at the desired rate with exponentially small failure probability. 
Finally, the bound on the mean absolute error follows from an additional truncation step applied in $\hat H$.

Similarly, \eqref{eq:G-combined} can also be applied to estimate other symmetric functionals including the power-sum $F_\alpha(P) =  \sum_{i=1}^k p_{i}^\alpha$, $\alpha \in (0,1)$ and the support size $S(P) =  |\{i \in[k] : p_i > 0\}|$, and attains the optimal sample complexity and the minimax rates of the mean absolute error established in \cite{JVHW15,WY15unseen}. 
The  precise results are provided in Appendix~\ref{app:prf_truncated_np}.

\subsection{Penalized NPMLE for unknown support size}
\label{sec:pen-npmle}
In the above discussions, the NPMLE program assumes knowledge of the true support size $k$. However, in practical sampling scenarios, we often have access only to the nonzero frequency counts, where the observed frequencies cover merely a fraction of the true support with many categories remaining unobserved.
A natural remedy is to augment the observed frequencies with zeros to an appropriate length, where the prescribed support size is selected through a data-driven procedure.
To address this issue, we develop a penalized variant of the NPMLE program that introduces a regularization term for support size selection, allowing joint optimization over both the histogram and the support size parameter.

Suppose that we have observed a multiset of $k$ non-zero frequency counts $N=\{N_i\}_{i=1}^k$ with $N_i\ge 1$.
We add zeros onto $N$ to length $k'\geq k$ to an extended multiset $N'=\{N_i\}_{i=1}^{k'}$ with $N_{k+1}=\ldots =N_{k'}=0$.
Let $H(p)=p\log \frac{1}{p} +(1-p)\log \frac{1}{1-p}$ denote the binary entropy function.
Consider the following penalized likelihood function
\begin{align}
    L(\pi;N,k') 
    & = \sum_{i=1}^{k} \log f_{\pi} (N_i) + (k' - k)\log f_{\pi}(0) + k'H(\frac{k}{k'}),
    \label{eq:Ni-pen-likelihood}
\end{align}
which avoids discrete optimization by accommodates non-integer $k'$.
Note that \eqref{eq:Ni-pen-likelihood} is concave in $\pi$ for any given $k'\geq k$.
Moreover, for any fixed $\pi$, the likelihood term
$\sum_{i=1}^{k} \log f_{\pi}(N_i) + (k' - k)\log f_{\pi}(0)$ decreases as $k'$ increases, while the regularization term
$k'H(\tfrac{k}{k'}) = -(k\log \tfrac{k}{k'} + (k'-k) \log \tfrac{k'-k}{k'})$ is strictly concave and grows with $k'$, thereby inducing a trade-off for support size selection.
The penalized NPMLE is then given by
\begin{align}
\label{eq:pen-NPMLE-k}
    \hat k, \hat \pi \in \argmax_{k'\geq k, \pi \in \calP([0,1])} L(\pi;N,k').
\end{align}

Next, we investigate the optimality conditions of \eqref{eq:pen-NPMLE-k}.
Let $\pi_{N}=\tfrac{1}{k}\sum_{i=1}^{k} \delta_{N_i}$  and $\pi_{N'} =  \frac{k}{k'}\pi_{N} + \frac{k'-k}{k'} \delta_0$. By definition, we have
\begin{align}
\label{eq:0th-opt-pen}
   L(\pi;N,k') - \sum_{i=1}^{k}  \log {\pi_{N}(N_i)}  & = -\sum_{i=1}^{k'}  \log \frac{\pi_{N'}(N_i)}{f_{\pi} (N_i)} 
    = -k' \KL(\pi_{N'}\| f_\pi).
\end{align}
Hence, the penalized NPMLE can be interpreted as minimizing the scaled $\KL$ divergence, where the regularization term naturally arises from this formulation.

For the first-order optimality, if $\hat k>k$, we have
$$\frac{\partial L(\pi;N,k')}{\partial k'} \mid_{k'=\hat k} = \log f_{\pi} (0) - \log \frac{\hat k - k}{\hat k}=0,$$ which implies that $f_{\hat \pi}(0)= \frac{\hat k - k}{\hat k}$ if $\hat k$ exists.
Hence, the regularization aligns the zero-probability mass of the optimized Poisson mixture with that of the empirical histogram.
For $k'\geq k$, let $\hat\pi_{k'} \triangleq \argmax_{\pi \in \calP([0,1])} L(\pi;N,k')$ denote the NPMLE with a fixed $k'$.
Similar to \eqref{eq:1st_opt}, we have for any $Q\in \calP([0,1])$,
\begin{align*}
     \sum_{i=1}^k  \frac{ f_{Q}(N_i) }{ f_{\hat\pi_{k'}}(N_i) } + (k'-k) \frac{ f_{Q}(0) }{ f_{\hat\pi_{k'}}(0) }  \leq k'.
\end{align*}
Particularly, if $\hat k>k$, we have with $\hat \pi= \hat\pi_{\hat k}$,
\begin{align}
\label{eq:1st-opt-pen-pi}
    \sum_{i=1}^k  \frac{ f_{Q}(N_i) }{ f_{\hat\pi}(N_i) } + \hat k f_{Q}(0) = \sum_{i=1}^k  \frac{ f_{Q}(N_i) }{ f_{\hat\pi}(N_i) } + (\hat k-k) \frac{ f_{Q}(0) }{ f_{\hat\pi}(0) }  \leq \hat k.
\end{align}

\begin{proposition}
\label{prop:pen-npmle}
For any given $\{N_i\}_{i=1}^k$, $N_i>0$, the following holds:
\begin{enumerate}
    \item[(i)] $L(\hat\pi_{k'};N,k')$ is monotone non-decreasing with respect to $k'$ over  $[k,\infty)$.
    \item[(ii)] 
    Suppose that $k<\hat k<\infty$. Then, $L(\hat \pi_{k'};N,k') = L(\hat \pi;N,\hat k)$ for any $k'\geq \hat k$. Moreover, 
    $\hat \pi_{k'} = \frac{\hat k}{k'}\hat\pi + (1-\frac{\hat k}{k'}) \delta_0$ if $k'\in\naturals$ and $k'\geq \hat k$. 
\end{enumerate}
\end{proposition}

Proposition~\ref{prop:pen-npmle} characterizes the convergence behavior of the penalized NPMLE. First, convergence is guaranteed since the penalized likelihood is non-decreasing and uniformly bounded above. Second, if $\hat k$ exists, increasing the support size beyond $\hat k$ only adds extra zeros to the NPMLE without increasing the penalized likelihood. 
Consequently, $\hat k$ can be chosen at the phase-transition point, that is, the smallest $k$ at which the penalized likelihood reaches its maximum or shows negligible increase with further growth. 
Figure~\ref{fig:K-times-KL} provides an illustration of the support size selection based on the scaled $\KL$ divergence under the uniform $P$-model; see Section~\ref{sec:simulation} for further numerical simulations.
The proof of Proposition~\ref{prop:pen-npmle} follows from the optimality conditions and is deferred to Appendix~\ref{app:prf_pen_np}.
\begin{figure}[htbp]
    \centering    \includegraphics[width=0.5\linewidth]{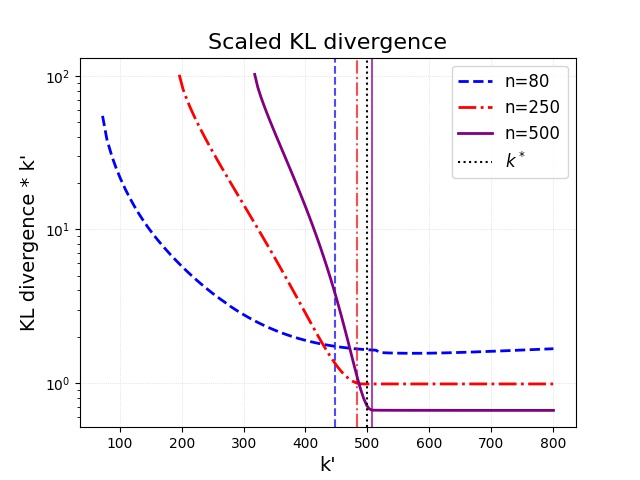}
    \caption{The scaled $\KL$ divergence $k' \cdot \KL(\pi_{N'}\| f_{\hat\pi_{k'}})$ under the uniform distribution with true support size $k^\star = 500$ and varying sample sizes $n$. Each curve starts at the number of observed non-zero counts $k$, and the vertical colored line indicates the selected $\hat k$ value.}
    \label{fig:K-times-KL}
\end{figure}

\begin{remark}[Countable support set]
Another relevant setting is the Poisson model with a countable support set,  where $N_i \inddistr \Poi(n p_i)$ with $P=(p_1,p_2,\ldots)$. Notably, such a model can be made statistically indistinguishable 
from a finite-support model \eqref{eq:poisson-sampling} by aggregating categories with sufficiently small probability masses\footnote{Let $\tilde p=(p_1,p_2,\ldots,p_k,\tilde p_{k+1},0,0,\ldots)$, where $k$ is chosen such that $\tilde p_{k+1} \triangleq \sum_{j=k+1}^\infty p_j \leq o(n^{-1})$. 
Applying the identity $1-\tfrac{1}{2}H^2(\otimes_{i} \Poi(np_i),\otimes_{i} \Poi(n\tilde p_i)) = \prod_{i} (1-\tfrac{1}{2}H^2(\Poi(np_i),\Poi(n\tilde p_i))) = \exp (- \tfrac{n}{2}\sum_{i=k+1}^\infty (\sqrt{p_i}-\sqrt{\tilde p_i})^2)$ \cite[Sec.~2.4]{Tsybakov2009IntroductionTN}
yields that $H^2(\otimes_{i} \Poi(np_i),\otimes_{i} \Poi(n\tilde p_i))\leq o(1)$, indicating the statistical indistinguishability.}.
Given the existence of the finite-support surrogate, the proposed method is then able to adaptively determine the effective support size.
\end{remark}

\begin{remark}[Model selection]
    An alternative perspective for selecting the support size is through model selection: each $k'$  defines a distribution family $\calM_{k'}$ in a nested sequence $\{\calM_{k'}\}_{k'\in\naturals}$ with $\calM_{k'}\subseteq \calM_{k'+1}$. For the Poisson model, the observation sequence can be expressed as $(N_1,N_2,\ldots,N_k,0,0,\ldots)$, and the model is $\bigotimes_{i=1}^{k'} \poi(N_i,np_i) \bigotimes_{i=k'+1}^{\infty} \delta_0$ for $k'\ge k$ and $P\in \Delta_{k'-1}$. As $k'$ increases, the gain in maximum likelihood within $\calM_{k'}$ can be controlled by the complexity (e.g., bracketing entropy) of the nested models, and a penalty can be added to ensure strong consistency of $\hat k$. This approach is used in \cite{GH12} for location mixture models with an \iid\ sample; a rigorous theoretical analysis for our model is left for future work.
\end{remark}

\section{Numerical experiments}
\label{sec:exp}
\subsection{Numerical simulation} 
\label{sec:simulation}
To begin with, we introduce the implementation of the Poisson NPMLE \eqref{eq:pois-npmle-def}.
Although the NPMLE program is convex in the mixing distribution, the primary challenge stems from its inherently infinite-dimensional formulation.
Following the approach of \cite{KM13,KIM2020NPMLE,Zhang2022NPMLE}, a standard strategy is to approximate the infinite-dimensional problem by restricting the mixing distribution $\pi$ to a finite grid $\{r_j\}_{j=1}^m$ and optimizing its weights over the simplex $\Delta_{m-1}$. While the previous works 
construct the grid $\{r_j\}_{j=1}^m$ using equally spaced support points, we adopt a data-dependent truncated scheme that pay more attention in small probability values that is crucial in large-alphabet estimation.
Then, we optimize the dual formulation of the NPMLE program  as suggested by \cite{KM13}.
The localized NPMLE and penalized NPMLE also follow from this procedure, where for localized NPMLE we set $N'=N$ in \eqref{eq:npmle-trunc-def} without using a second sample.
We implement the NPMLE-based estimators in Python using the commercial optimization software MOSEK \cite{andersen2000mosek}. 
See Appendix~\ref{app:exp_algo} for more implementation details.

In the following, we present experimental results on synthetic data.
We evaluate the performance of the proposed methods for entropy estimation. 
We let $n$ range from $10^2$ to $10^5$ and consider both the \textit{large-sample regime} with $k=10^2$ and the \textit{large-alphabet regime} with $k=10^5$, respectively.
Given each sample size $n$ and alphabet size $k$, we generate frequency counts via \iid\ sampling $N \sim \operatorname{Multi}(n, P)$\footnote{The \iid\ and Poisson sampling schemes resemble each other; see Section~\ref{sec:binomial-npmle} for further discussions.}.
The underlying distribution $P\in\Delta_{k-1}$ is selected 
among the follows to capture varying heterogeneity conditions: the \textit{uniform} distribution $p_i = k^{-1},\ i\in[k]$; the \textit{spike-and-uniform}  distribution $p_i = \frac{1}{2(k-3)}$ for $i \in [k-3]$, and $p_{k-2}= p_{k-1} = \frac{1}{8}, p_{k}=\frac{1}{4}$; and the \textit{Zipf(1)} distribution $p_i\propto i^{-1}$. 
See Appendix~\ref{app:add_exp} for additional experiments with other choices of $P$.
The NPMLE-based estimators, including the NPMLE plug-in estimator \eqref{eq:F_hat} (NP) and the localized NPMLE estimator (NP-L), are compared with several existing methods: the empirical distribution (EMP), the Miller–Madow (MM) estimators \cite{miller1955note}, the polynomial-based estimators (JVHW and WY) \cite{JVHW15,WY16}, Valiant and Valiant’s histogram plug-in estimator (VV) \cite{vv17}, and the PML plug-in estimator (PML) implemented by \cite{ACSS21}. 
For each $(n, k, P)$ and estimator, we conduct 50 independent trials and compute the root mean squared error (RMSE) of the estimates.

Figure~\ref{fig:fig-ent-1} presents the results of entropy estimation\footnote{For visualization clarity on the logarithmic scale, we cap the relative length of error bars at 50\% of the corresponding estimate.}.
Among the baseline methods, the classical EMP and MM estimators perform well in the large-sample regime but deteriorate significantly when the alphabet size grows. Advanced methods such as JVHW, WY, VV, and PML generally achieve better accuracy in large-alphabet scenarios, but their performance could be unstable as the underlying distribution or the ratio between $k$ and $n$ varies, since these methods often rely on linear programs or high-order polynomials with many tuning hyperparameters.
In comparison, NP and NP-L demonstrate fast and stable convergence across both regimes, achieving low RMSE in most cases. 
Moreover, NP and NP-L exhibit similar performance in practice.
Additional results for other functionals including the support size and R\'enyi entropy are presented in Appendix~\ref{app:add_exp}, showing the broad advantages of the NPMLE-based estimators.

\begin{figure}[H]
    \centering
    \subfigure[Uniform]{
    \includegraphics[width=0.3\linewidth]{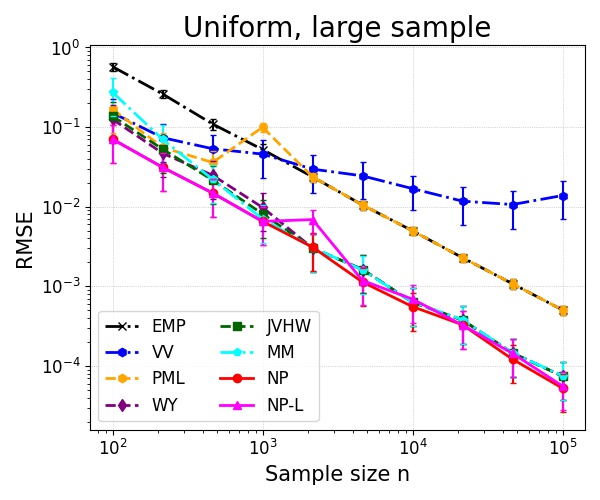}
    }
    \subfigure[Spike-and-uniform]{
    \includegraphics[width=0.3\linewidth]{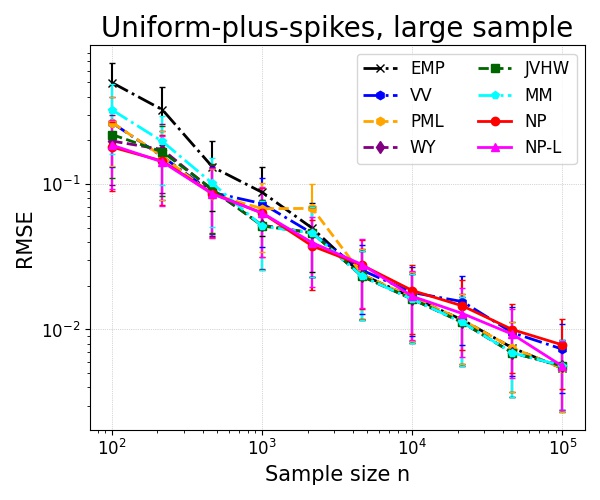}
    }
    \subfigure[Zipf(1)]{
    \includegraphics[width=0.3\linewidth]{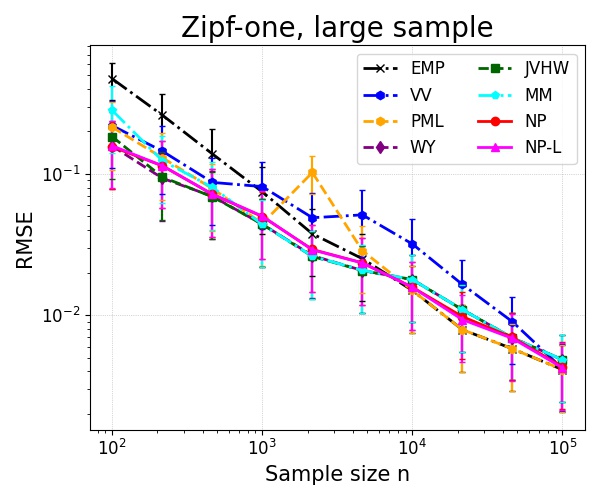}
    }
    \subfigure[Uniform]{
    \includegraphics[width=0.3\linewidth]{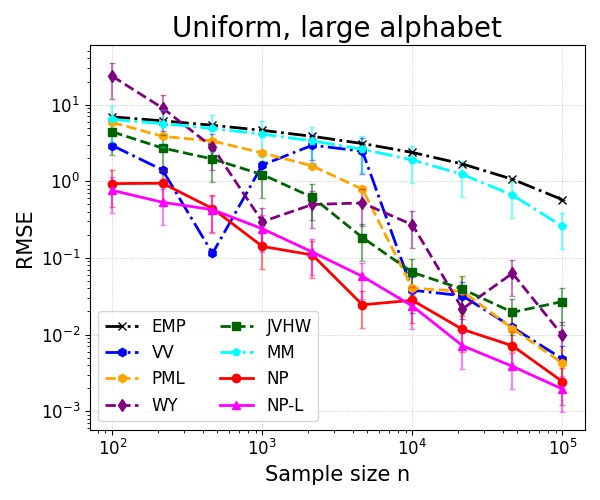}
    }
    \subfigure[Spike-and-uniform]{
    \includegraphics[width=0.3\linewidth]{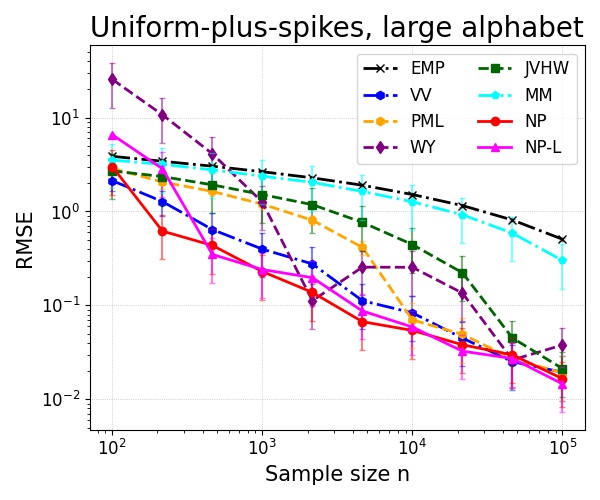}
    }
    \subfigure[Zipf(1)]{
    \includegraphics[width=0.3\linewidth]{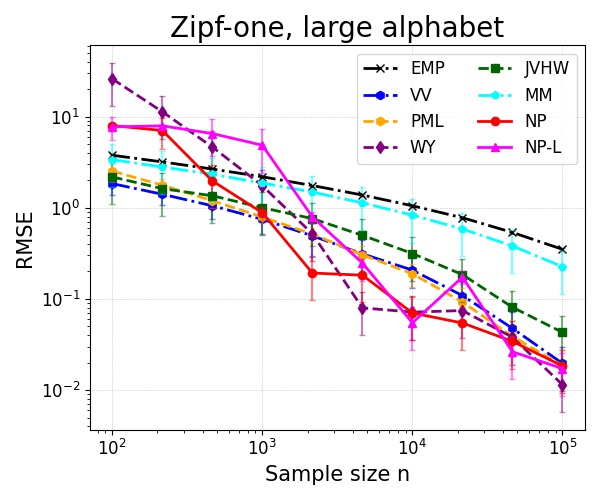}
    }
    
    \hspace{.0in}
    \hspace{.0in}
    \caption{Shannon entropy estimation: Panels (a)–(c) plot the RMSE of the large-sample regime, while panels (d)–(f) show the results of the large-alphabet regime.}
    \label{fig:fig-ent-1}
\end{figure}

Next, we apply the penalized NPMLE~\eqref{eq:pen-NPMLE-k} to the task of entropy estimation.
In Figure~\ref{fig:fig-reg-NP}, the plug-in estimator based on the penalized NPMLE (orange) is compared with that of the standard NPMLE using only the observed non-zero frequency counts (blue), as well as with the oracle estimator that uses the complete frequency counts with $k^\star = 500$.
All estimators are implemented with a grid size of $m = 500$ and evaluated over 50 independent trials.
Figures~\ref{fig:fig-reg-NP}(a)–(b) show that, when only non-zero counts are available, the penalized NPMLE markedly outperforms the standard version and achieves performance close to the oracle estimator across different underlying distributions.
Moreover, the boxplots of $\hat{k}$ in Figures~\ref{fig:fig-reg-NP}(c)–(d) indicates that the estimated support size $\hat{k}$ converges to $k^\star$ as $n$ increases.

\begin{figure}[htbp]
    \centering
    \subfigure[Uniform, entropy]{
    \includegraphics[width=0.33\linewidth]{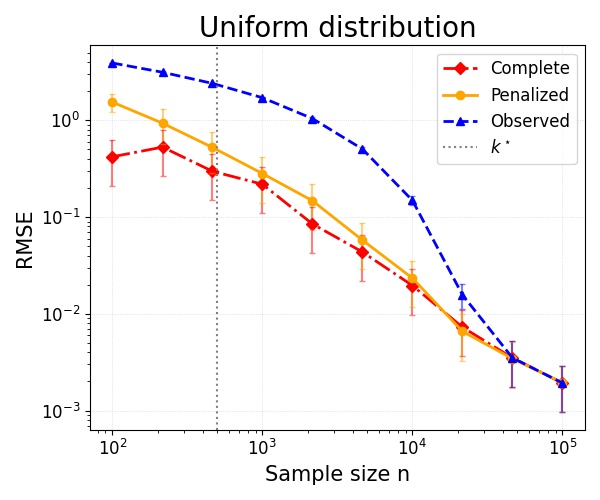}
    }
    \subfigure[Zipf(1), entropy]{
    \includegraphics[width=0.33\linewidth]{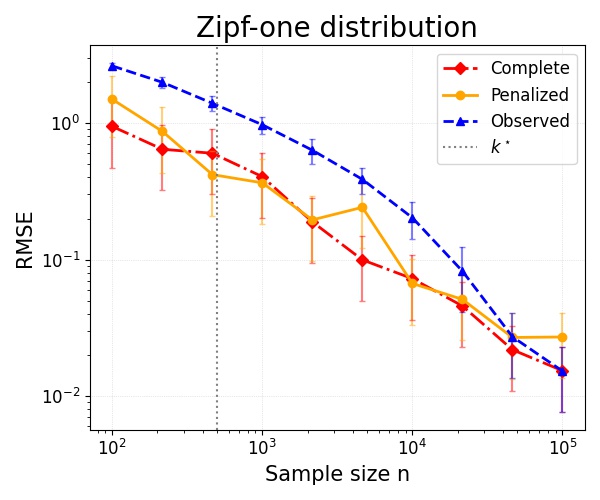}
    }
    \hspace{.0in}
    \subfigure[Uniform, $\hat k$]{
    \includegraphics[width=0.33\linewidth]{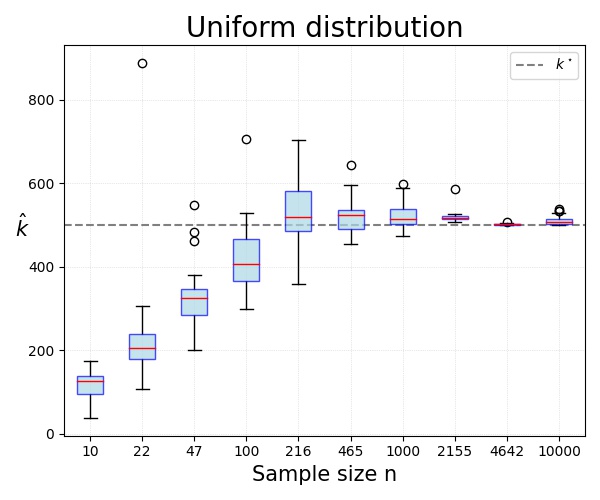}
    }
    \subfigure[Zipf(1), $\hat k$]{
    \includegraphics[width=0.33\linewidth]{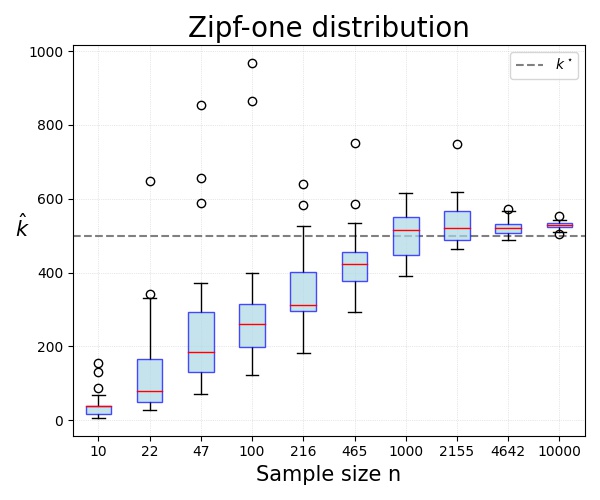}
    }
    \hspace{.0in}
    \caption{Performance of the penalized NPMLE.}
    \label{fig:fig-reg-NP}
\end{figure}

\subsection{Real-world data experiments} 
\label{sec:exp_real}
In this section, we evaluate the performance of the NPMLE plug-in estimator on the  application scenarios of computational linguistics and neuroscience.

\paragraph{Entropy estimation on linguistic corpus}
We begin by estimating the entropy per word in the novel \textit{Moby Dick} by Herman Melville. The text contains $n_{total}=210321$ words, with a total of $k = 16509$ distinct words.
In each of the 50 trials, we randomly sample $n$ words from the text without replacement and estimate the entropy based on the observed frequency counts.
\paragraph{Quantifying information content in neuronal signals} 
Entropy estimation on neural spike train data help assess how much information neurons convey about external stimuli or internal states \cite{strong1998entropy}.
We apply the dataset collected by \cite{UC2003spike}, which contains spike recordings from 2 ON and 2 OFF primate retinal ganglion cells responding to binary white noise stimuli. 
The spike times from the 4 neurons are grouped into time bins matching the stimulus frame rate (120 Hz) in the original data. We then combine them into 5-frame windows and encode the neuron spike counts for entropy estimation.

\paragraph{Estimating the number of unseen}
We revisit the problem of estimating the number of words Shakespeare likely knew but never used, a question explored in \cite{ET76, thisted1987shakespeare}.
This falls under the class of unseen-species estimation problems originally proposed by Fisher \cite{fisher1943relation}.
Under the Poisson scheme, the quantity of interest is the expected number of categories that have zero occurrences during the first $n$ units of time, but occur at least once during an additional $tn$ units of time for some $t > 0$.
With $g(x)=e^{-nx} (1-e^{-tnx})$, the target symmetric additive functional is
\begin{align}
\label{eq:unseen-def}
    G = \sum_{i} g(p_i)
    =\sum_{i} e^{-np_i} \cdot\left(1-e^{-tnp_i}\right) \triangleq \sum_{i} g(p_i).
\end{align}
We apply the NPMLE plug-in estimator \eqref{eq:F_hat} to this problem.
We use the corpus of Shakespeare’s 154 sonnets (14-line poems) for evaluation. In each of the trials, we randomly select 60 sonnets to form the observed sample of $n$ words, and then sample $nt$ additional words from the remaining sonnets using a range of values for $t$.
Other baseline estimators include the PML plug-in estimator, the Good–Toulmin (GT) estimator \cite{GT56}, and the smoothed Good–Toulmin (SGT) estimator \cite{OSW16}. For comparison, we apply SGT estimators with Poisson and binomial smoothing distributions, as suggested by \cite[Theorem 1]{OSW16}.

\begin{figure}[htbp]
    \centering
    \subfigure[Moby Dick]{
    \includegraphics[width=0.3\linewidth]{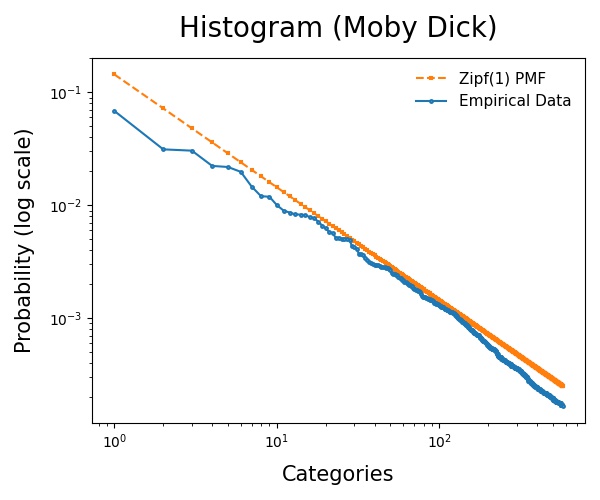}
    }
    \subfigure[Neural spike train]{
    \includegraphics[width=0.3\linewidth]{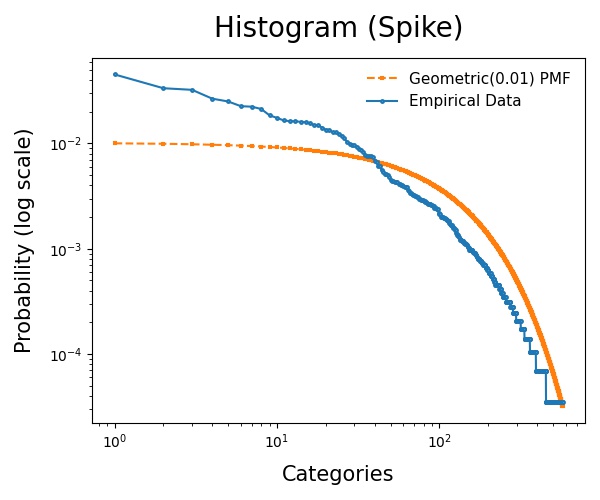}
    }
    \subfigure[Sonnets]{
    \includegraphics[width=0.3\linewidth]{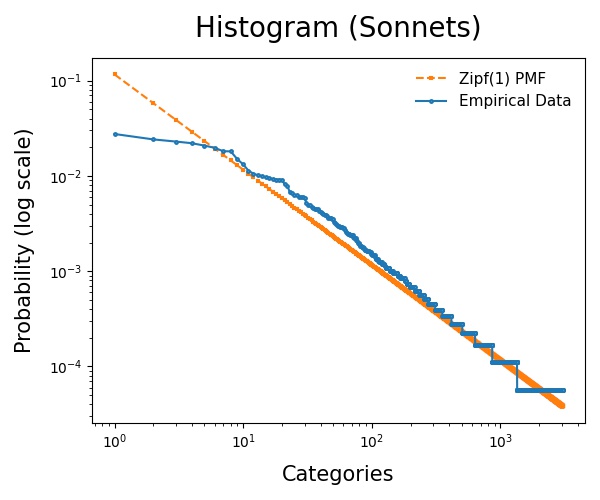}
    }
    \hspace{.0in}
    \subfigure[Moby Dick]{
    \includegraphics[width=0.3\linewidth]{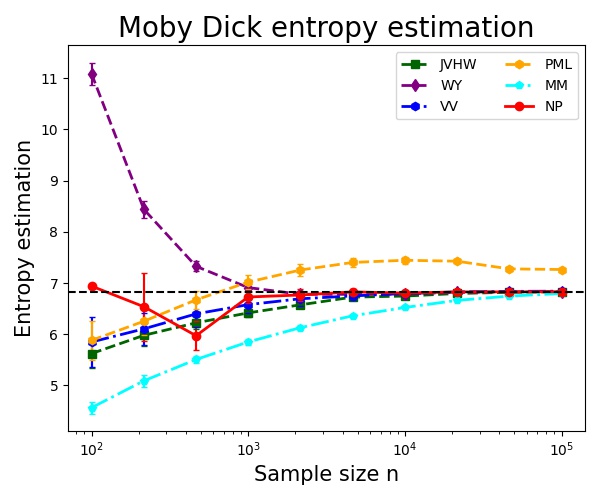}
    }
    \subfigure[Neural spike train]{
    \includegraphics[width=0.3\linewidth]{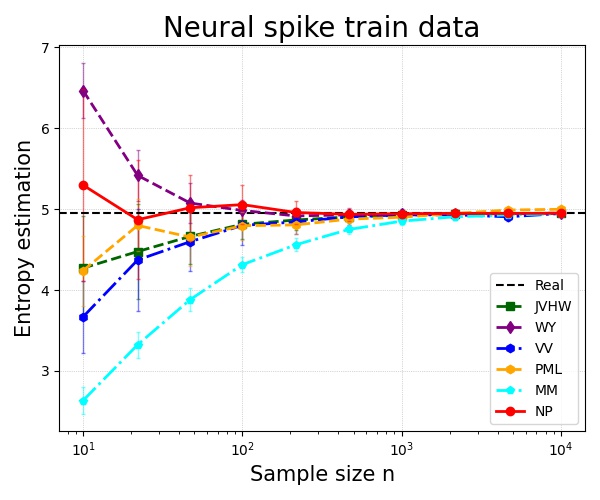}
    }
    \subfigure[Sonnets]{
    \includegraphics[width=0.3\linewidth]{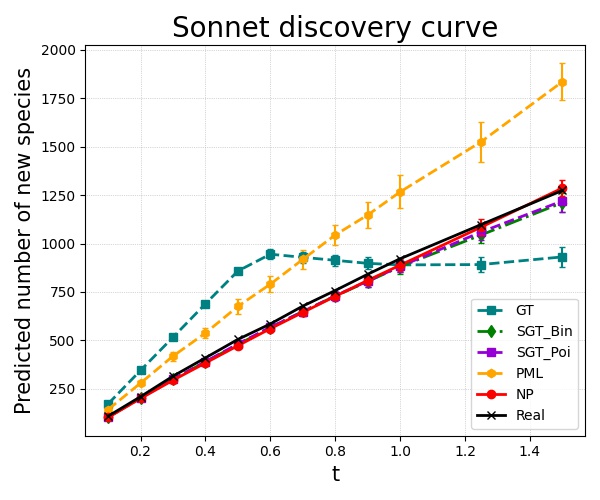}
    }
    \hspace{.0in}
    \caption{Experiment on real-world datasets.}
    \label{fig:fig-real-data}
\end{figure}

Figure~\ref{fig:fig-real-data} summarizes the results for the experiments. (a)–(c) show the histograms of the datasets  with a respective reference distribution: the linguistic datasets, Moby Dick and Sonnets, exhibit power-law tails, while the neural spike train data displays lighter tails resembling a geometric distribution.
(d) and (e) illustrate the convergence of the estimators as the sample size $n$ increases on the novel \textit{Moby Dick} and neural spike train data. For both datasets, we set $k$ to the number of distinct elements (words or firing patterns) observed in the entire dataset. The black dashed line marks the empirical entropy computed from the full dataset.
The NPMLE-based estimators achieve high accuracy especially when the sample size $n$ is relatively small ($10^1$–$10^3$) compared with the support size $k$ ranging from $10^4$ to $10^5$, while other estimators suffer from larger error in this regime.
Panel (f) plots the discovery curve of the predicted newly observed categories as $t$ varies, averaged over 50 trials.
The NPMLE-based estimator is implemented with support size $k=66534$, following the estimate of Shakespeare’s total vocabulary size in \cite{ET76}.
The actual number of newly discovered words is shown as a gray line. The results indicate that the NPMLE plug-in estimator aligns most closely with the true values, particularly in cases where  $t\geq 1$.

\subsection{Application on large language model evaluation} 
\label{sec:llm}

Large language models (LLMs) have demonstrated remarkable capabilities in recent years, making their evaluation increasingly essential for reliability, accuracy, and safe deployment in real-world applications. Nevertheless, their strong generalization ability results in an extremely large output space, posing substantial challenges for reliable assessment.
A simple yet effective approach is to characterize key properties of the output distribution through a certain functional, which enables the application of functional estimation based on a sample from repeated queries.
For instance, the model hallucination in terms of semantic consistency is characterized by  uncertainty measures such as entropy \cite{farquhar2024detecting,nikitin2024kernel}, and the number of unseen serves as quantifier the model’s capability unobserved by the outputs \cite{nasr2025scalable,li2025evaluatingunseen}. To this end, the NPMLE plug-in estimator serves as a competitive candidate due to its superior performance in large-alphabet settings, as typically encountered in LLM outputs.

\begin{figure}
    \centering
    \includegraphics[width=0.7\linewidth]{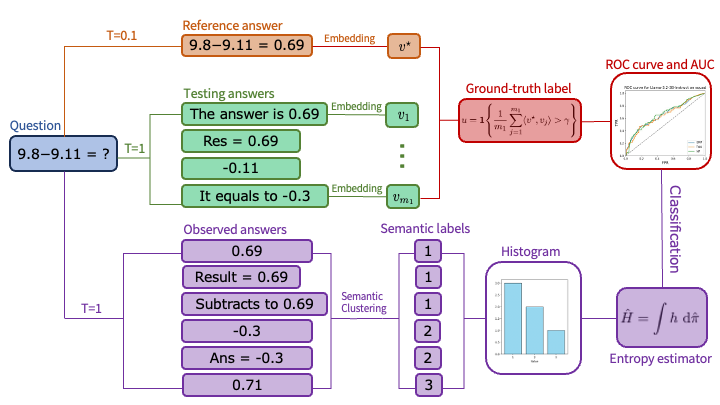}
    \caption{Diagram of the experiment on LLMs.}
    \label{fig:llm-diagram}
\end{figure}

We consider the detection problem of LLM hallucinations, defined as outputs that are nonsensical or unfaithful to the source. In particular, we focus on \textit{confabulations}, where models fluently generate unsubstantiated answers that are both incorrect and sensitive to randomness. 
Semantic entropy \cite{farquhar2024detecting} is an effective approach to capture hallucination by entropy estimate of  model outputs at the semantic level, giving improved performance to the naive lexical approaches. Following the framework of semantic entropy estimation, we evaluate the following LLMs: \texttt{Llama-3.2-3B-Instruct} \cite{grattafiori2024llama3}, \texttt{Mistral-7B-Instruct-v0.3} \cite{jiang2023mistral7b}, \texttt{Qwen3-4B-Instruct-2507} \cite{yang2025qwen3technicalreport}, and \texttt{DeepSeek-R1-Distill-Llama-8B} \cite{deepseekai2025deepseekr1}
across datasets from diverse domains, including general knowledge \texttt{SQuAD} \cite{rajpurkar2018squad}, biology and medicine \texttt{BioASQ} \cite{krithara2023bioasq}, and open-domain \texttt{NQ} \cite{kwiatkowski2019nq}. 
For each (model, dataset) pair, we randomly select a given number of  questions and generate multiple answers at different temperatures to form reference, testing, and observation sets. Binary ground-truth hallucination labels are then constructed via embedding the reference and testing answers and computing the cosine similarity.
Finally, semantic clustering and entropy estimation are applied to the observed answers, and model performance is evaluated using the receiver operating characteristic (ROC) curve and the area under the ROC curve (AUC) (see, \eg, \cite[Sec.~9.2.5]{hastie2009elements}) for the binary event that hallucination occurs across all questions. Figure~\ref{fig:llm-diagram} summarizes the overall procedure, and more experimental details are provided in Appendix~\ref{app:llm}.

\begin{table}[htbp]
\centering
\setlength{\tabcolsep}{3pt} 
\renewcommand{\arraystretch}{1.15} 
\scriptsize
\begin{tabularx}{0.95\textwidth}{l *{4}{ccc}}
\toprule
 & \multicolumn{3}{c}{\textbf{Llama-3.2}} 
 & \multicolumn{3}{c}{\textbf{Mistral-v0.3}} 
 & \multicolumn{3}{c}{\textbf{Qwen3}} 
 & \multicolumn{3}{c}{\textbf{DeepSeek-R1}} \\
\cmidrule(lr){2-4} \cmidrule(lr){5-7} \cmidrule(lr){8-10} \cmidrule(lr){11-13}
 & \texttt{SQuAD} & \texttt{BioASQ} & \texttt{NQ} & \texttt{SQuAD} & \texttt{BioASQ} & \texttt{NQ} & \texttt{SQuAD} & \texttt{BioASQ} & \texttt{NQ} & \texttt{SQuAD} & \texttt{BioASQ} & \texttt{NQ} \\
\midrule
EMP     & 0.6538 & 0.7522 & 0.8657 &  0.7664 & 0.7298  & 0.8493  & \textbf{0.8182} & \textbf{0.8389} & 0.8843  &  0.7082 & 0.5462 & 0.8086 \\
TOK & 0.6465 & 0.7536 & 0.8665 & 0.7494 & \textbf{0.7333} & 0.8488 & 0.8142 &0.8158 &0.8748   &  0.7123 & 0.5439 & 0.8069 \\
NP      & \textbf{0.6623} & \textbf{0.7551} & \textbf{0.8706} &  \textbf{0.7808}  & 0.7241 & \textbf{0.8502} & 0.8147 &0.8383 &\textbf{0.8899} & \textbf{0.7190}  & \textbf{0.5500} & \textbf{0.8154}  \\
\bottomrule
\end{tabularx}
\caption{AUC values across different models and datasets.}
\label{tab:llm_AUCs}
\end{table}
Table~\ref{tab:llm_AUCs} summarizes the results. The NPMLE plug-in estimator \eqref{eq:F_hat} (NP) is applied for entropy estimation given the semantic labels of the observed answers.
As baselines, we include the empirical estimator (EMP), also referred to as the discrete semantic entropy in \cite{farquhar2024detecting}, and the Shannon entropy computed from the normalized token log-probabilities from the model outputs (TOK). Compared with TOK, NP relies only on the model outputs themselves rather than token-level logit probabilities,  which may be inaccessible for black-box LLMs (e.g., GPT-4 and Claude). 
The NPMLE estimator achieves higher AUC values across most settings, implying more accurate and robust detection of hallucinations. 
While the improvement is moderate given the limited number of observations constrained by the cost of semantic clustering, the flexibility and strong performance of the NPMLE-based estimator suggest promising potential for scaling to larger models and broader evaluation tasks.

\section{Discussion}
\label{sec:discussion}

\subsection{Modeling with binomial mixtures}
\label{sec:binomial-npmle}
So far, we have focused on Poisson mixtures with $q_n(x, r) = \poi(x, nr)$ in \eqref{eq:N_mixture}.  A natural question arises: can alternative mixture models also effectively capture frequency count behavior? One such example is the binomial mixture with $q_n(x, r) = \bin(x, n, r)$,  which is directly motivated by the \iid\ sampling scheme $N \sim \operatorname{Multi}(n, P)$
with marginals $\pbb[N_i = j] = \bin(j, n, p_i)$.
In this case, the histogram distribution can be estimated via the binomial NPMLE:
$$
\hat{\pi} = \argmax_{\pi \in \mathcal{P}([0,1])} \sum_{i=1}^k \int \bin(N_i, n, r)\, \mathrm{d}\pi(r).
$$
Given the close connection between the Poisson and multinomial models (see Section~\ref{sec:poisson_npmle}), it is reasonable to expect comparable performance under the both settings.
Figure~\ref{fig:fig-models-NP} compares the entropy plug-in estimators under the binomial and Poisson settings, where the frequency counts are generated either from the Poisson or multinomial model and fitted using the Poisson or Binomial NPMLE. The results show that the two sampling schemes exhibit similar behavior, and both mixture models achieve nearly identical performance given the same input. 

\begin{figure}[htbp]
    \centering
    \subfigure[Uniform]{
    \includegraphics[width=0.3\linewidth]{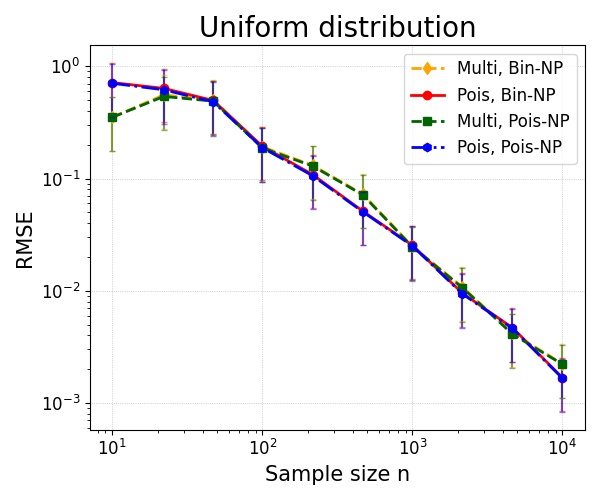}
    }
    \subfigure[Spike-and-uniform]{
    \includegraphics[width=0.3\linewidth]{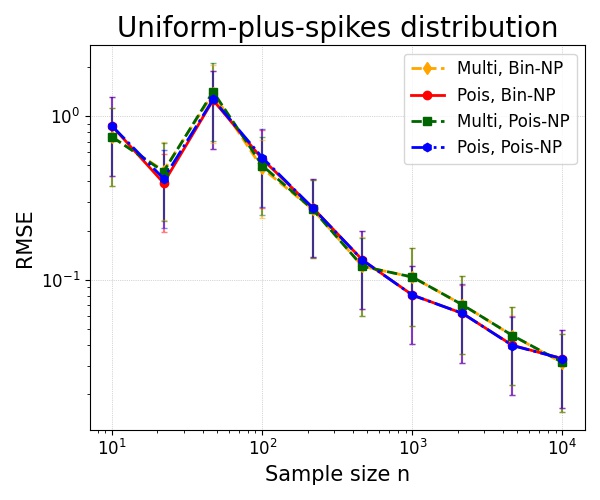}
    }
    \subfigure[Zipf(1)]{
    \includegraphics[width=0.3\linewidth]{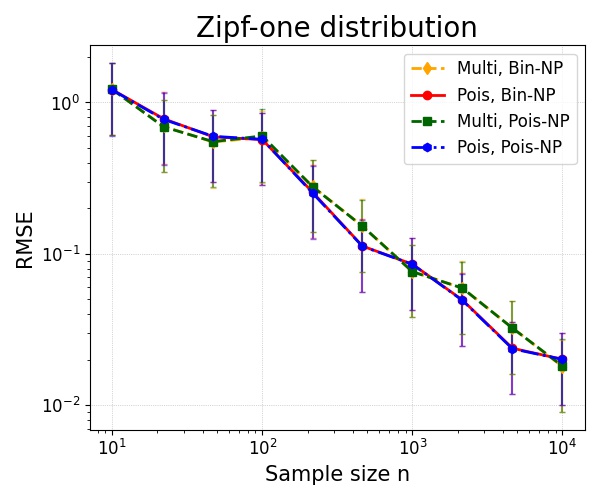}
    }  
    \hspace{.0in}
    \caption{Comparison between Binomial and Poisson settings for Shannon entropy estimation. ``Multi'' and ``Pois'' denote data generated from multinomial and Poisson sampling with $k = 1000$, while ``Bin-NP'' and ``Pois-NP'' correspond to the Binomial and Poisson NPMLE fits, respectively.}
    \label{fig:fig-models-NP}
\end{figure}

\subsection{Extension to continuous observations}
The framework of mixture modeling with NPMLE fitting can also be extended to continuous observations.
For example, consider the  Gaussian sequence model
\begin{align*}
    Y_i \inddistr N(\theta_i,1),
\end{align*}
where $\theta = (\theta_1,\dots,\theta_n)$ are unknown parameters.
In spirit of the proposed framework, we apply the Gaussian NPMLE 
\begin{equation*}
\hat{\pi}=\argmax_{\pi}\sum_{i=1}^n\log \int \frac{1}{\sqrt{2\pi}} \exp\pth{-\frac{(Y_i-\theta)^2}{2}} \diff \pi(\theta),
\end{equation*}
which serves as an estimate of the empirical mixing distribution $\pi_\theta = \frac{1}{n} \sum_{i=1}^n \delta_{\theta_i}$.
For the density estimation problem, \cite{Zhang08} establishes bounds on the Hellinger risk $H(f_{\pi_\theta}, f_{\hat{\pi}})$ under bounded support or tail conditions on $\pi_\theta$.
Then, following the analysis of Section~\ref{sec:theory-main},
the IPM between $\pi_\theta$ and $\hat \pi$ can be controlled via a similar deconvolution argument.
Consequently, for the downstream task of functional estimation, the NPMLE plug-in estimator $\hat G = n \cdot \Expect_{\hat \pi} g$ can be employed to estimate the target functional $G(\theta)=\sum_{i=1}^n g(\theta_i)$ (\eg, the power of the $L_q$-norm studied in \cite{CL11,CollierAOS15}). 

Moreover, a particular interest in the Gaussian sequence model is the sparse regime, \ie, $\|\theta\|_0 \triangleq \sum_{i=1}^n \indc{\theta_i\neq 0}\leq s$ for some $s\in[n]$.
When $s$ is known, a natural approach is to impose a sparsity constraint by solving
\begin{equation*}
\max_{\pi: \pi(0) \geq 1 - \frac{s}{n}}\sum_{i=1}^n\log \int \frac{1}{\sqrt{2\pi}} \exp\pth{-\frac{(Y_i-\theta)^2}{2}} \diff \pi(\theta).
\end{equation*}
which explicitly enforces the desired sparsity structure in the estimated mixing measure.
\cite[Section~7.2.4]{L95} provides guarantees for the convexity of the program and the existence of solutions.
A rigorous theoretical analysis of this extension is left for future work.

\bibliography{sample}
\bibliographystyle{alpha}
\appendix
\newpage
\section{Preliminaries}
\label{app:preliminaries}

\subsection{Polynomial and Poisson approximations}
We introduce some basic notations and results from approximation theory that will be used to establish the main results.
Let \( \Poly_D \) denote the set of all polynomials of degree at most \( D \). 
For a function \( f \) defined on a set \( I \), the best uniform approximation error by \( \Poly_D\) is defined as  
\[
E_D(f, I) \triangleq \inf_{p \in \Poly_D} \sup_{x \in I} |f(x) - p(x)|.
\]
Denote the maximum deviation (diameter) of $f$ over $I$ as 
\begin{align}
    \label{eq:M_g}
    M(f, I) \triangleq \sup_{x, y \in I} |f(x) - f(y)|.
\end{align}
We have $E_D(f, I) \leq M(f, I)$ by approximating $f$ with a constant function.

We provide further details on characterizing the approximation error in terms of the following modulus of smoothness \cite{DT87}. 
Define the first-order difference operator as 
$$\Delta_{h}^1(f, x)\triangleq\left\{\begin{array}{ll}
f(x+h / 2)-f(x-h / 2), & x \pm h / 2 \in[0,1], \\
0, & \text {otherwise,}
\end{array}\right.$$
and let $\Delta_{h}^{r}\triangleq\Delta\pth{ \Delta_{h}^{r-1}}$ for $r \in\naturals$.
Let $\varphi(x) = \sqrt{x(1 - x)}.$
For  $t\in[0,1]$, the $r\Th$ Ditzian-Totik moduli of smoothness of order $r$ is defined as 
$$\omega^{r}_{\varphi}(f, t)\triangleq\sup _{h \in[0, t]}\left\|\Delta_{h \varphi(\cdot)}^{r}(f, \cdot)\right\|_{\infty}.$$
The following lemmas relate $\omega^{r}_{\varphi}$ to the best approximation error.

\begin{lemma}[{\cite[Theorem 7.2.1]{DT87}}]
\label{lem:poly-DT}
    For any $r\in \naturals$ and $f \in L_{\infty}[0,1]$, there exists a constant $C=C(r)$  independent of  $D>r$  and $f$ such that
\begin{align*}
    E_D(f,[0,1]) \leq C \omega_{\varphi}^{r}\pth{ g, D^{-1}}, \quad D>r.
\end{align*}
\end{lemma}

\begin{lemma}[{\cite[Theorem 2.1.1]{DT87}}]
\label{lem:omega-K-DT}
Suppose that $f \in L_{\infty}(0,1)$ is $r$-times continuously differentiable for some $r\in\naturals$. 
Then, we have for some constants  $M>0$  and  $t_{0}>0$,
$$\omega_{\varphi}^{r}(f, t) \leq M t^{r}\left\|\varphi^{r} f^{(r)}\right\|_{\infty,(0,1)} , \quad 0<t \leq t_{0}.$$
\end{lemma}

For the special case of $1$-Lipschitz functions, the following simplified bound holds:
\begin{lemma}[Jackson's theorem]
\label{lem:approx-Jackson}
Let $D\in\naturals$. Given any $f\in\calL_1$ on a bounded interval $[a, b] \subseteq \mathbb{R}$, there exists a polynomial $p\in \Poly_D$ such that for some universal constant $C > 0$,
\[
|f(x) - p(x)| \leq \frac{C \sqrt{(b - a)(x - a)}}{D} \leq \frac{C(b - a)}{D}, \quad \forall x \in [a, b].
\]
\end{lemma}

The following lemma establishes upper bounds on the coefficients based on the Chebyshev's celebrated equioscillation theorem. We present a simplified version from  \cite[Lemma 11]{HS21}, which is a corollary of \cite[Sec.~2.9.12]{Tim63}.

\begin{lemma}
\label{lem:poly-coeff-bound}
Let  $p_{n}(x)=\sum_{\nu=0}^{n} a_{\nu} x^{\nu} \in \Poly_n$ such that  $\abs{ p_{n}(x)} \leq A$  for  $x \in[a, b]$. Then
\begin{enumerate}
\item[(a)] If  $a+b \neq 0$ , then
$$\abs{ a_{\nu}} \leq 2^{7 n / 2} A\abs{ \frac{a+b}{2}}^{-\nu}\pth{ \abs{ \frac{b+a}{b-a}}^{n}+1}, \quad \nu=0, \cdots, n.$$
\item[(b)] If  $a+b=0$, then
$$\abs{ a_{\nu}} \leq A b^{-\nu}(\sqrt{2}+1)^{n}, \quad \nu=0, \cdots, n.$$
\end{enumerate}
\end{lemma}

Next, we turn to the Poisson approximation, which aims to approximate a given function using Poisson mass functions. 
The following two lemmas control the Poisson approximation error and upper bound the corresponding coefficients.

\begin{lemma}
\label{lem:charlier-bound}
A polynomial $p(x)=\sum_{d=0}^D a_d (x-x_0)^d \in \Poly_D$ admits the representation $p(x) = a_0 + \sum_{j=0}^{\infty} b_j \poi(j, n x)$ with coefficients $\{b_j\}$ satisfying
\begin{align*}
    \abs{ b_{j}} &  \leq  \sum_{d=1}^{D} |a_d| \pth{2\max\sth{\abs{\frac{j}{n}-x_0}, \sqrt{\frac{4jD}{n^2}} }}^d.
\end{align*}
\end{lemma}
\begin{proof}
Note that
\[
    \sum_{j=d }^{\infty} \frac{j!}{(j-d )!n^{d }} \cdot \poi(j, n x)
    = \sum_{j=d }^{\infty}  \frac{(nx)^{j-d }}{(j-d )!} x^d  e^{-nx} =x^{d }.
\]
Then, we have 
\begin{align*}
p(x) 
& = a_0 + \sum_{d=1}^{D} a_{d} \sum_{d^{\prime}=0}^{d}\binom{d}{d^{\prime}}\pth{-x_0}^{d-d^{\prime}} x^{d^{\prime}} \\
& = a_0 + \sum_{d=1}^{D} a_{d} \underbrace{\sum_{d^{\prime}=0}^{d}\binom{d}{d^{\prime}}\pth{-x_0}^{d-d^{\prime}} \sum_{j=0}^{\infty} \frac{j!}{\pth{ j-d^{\prime}}!n^{d^{\prime}}}}_{\triangleq g_d} \cdot \poi(j, n x) 
\end{align*}
Applying \cite[Lemma 30]{LMM} yields that $|g_d| \leq  (2 (|\frac{j}{n}-x_0| \vee \sqrt{4jD/n^2} ) )^d$.
Applying the triangle inequality, the desired result follows.
\end{proof}

\begin{lemma}
\label{lem:pois-poly-approx}
Let $p\in\Poly_D$, $S\subseteq [0,1]$ an interval, and $r$ a $t$-large function (defined in~\eqref{eq:r_condition}).
There exist universal constants $C,c_0,c_1$ such that, if $t\geq \frac{CD}{n}$, one can construct a function of the form $g (x) = a + \sum_{j=0}^{\infty} b_{j} \poi(j,nx)$ satisfying 
\begin{align}
\label{eq:pois-poly-approx-1}
    \|p- g\|_{\infty,S_r} 
    \leq  2 M(p,S_r) \cdot n\exp\pth{-c_1 nt},
\end{align}
with $b_{j}=0$ for $j/n \notin S_{2r} $ and $\max _{j}\abs{ b_{j}} \leq c_0^D M(p,S_r)$.
\end{lemma}
\begin{proof}
Without loss of generality, let $S_r=[x_0-L,x_0+L]$ and $p(x)=\sum_{d=0}^D a_d (x-x_0)^d$.
Let $M_0\triangleq M(p,S_r)$.
Applying Lemma~\ref{lem:poly-coeff-bound}(b) yields $|a_d|\le M_0 c_2^D L^{-d}$ for all $d\in [D]$, where $c_2=\sqrt{2}+1$.
It follows from Lemma~\ref{lem:charlier-bound} that $p(x)=a_0 + \sum_{j=0}^{\infty} b_j' \poi(j, n x)$, where 
\begin{align}
\label{eq:prf-pois-poly-approx-2}
    \abs{ b_{j}'} 
    \leq M_0 c_2^D \sum_{d=1}^{D} \pth{2 L^{-1}\max\sth{\abs{\frac{j}{n}-x_0}, \sqrt{\frac{4jD}{n^2}} }}^d.
\end{align}
We construction the approximation function $g$ as
\[
g(x) = a_0 + \sum_{j/n \in S_{2r}} b_{j}' \poi(j,nx).
\]

We first upper bound the coefficients $|b_j'|$ for $j/n\in S_{2r}$. 
By definition, there exists $x'\in S$ such that $|j/n-x'|\leq 2r(x')$.
Note that $r(x')\le L$. 
It follows that
\[
\abs{\frac{j}{n}-x_0}\leq \abs{\frac{j}{n}-x'} + |x'-x_0| \leq 2r(x') + L \le 3L.
\]
Since $t \ge CD/n$, we have
\[
\sqrt{\frac{D}{n}\frac{j}{n}}
\leq \sqrt{\frac{t}{C}\frac{j}{n}} 
\leq \sqrt{\frac{t}{C}(x'+2r(x'))}
\le \sqrt{\frac{t}{C}3(x'\vee r(x'))}
\stepa{\le} \sqrt{\frac{3}{C}} r(x')
\leq \sqrt{\frac{3}{C}} L,
\]
where (a) applies the $t$-large condition~\eqref{eq:r_condition}.
Hence, \eqref{eq:prf-pois-poly-approx-2} implies $|b_{j}'| \leq M_0 c_0^D$ for $j/n\in S_{2r}$.

Next we upper bound the approximation error $|p(x)- g (x)| = |\sum_{j\notin S_{2r}} b_j' \poi(j,nx)|$ for $x\in S_r$.
Consider the coefficients $|b_j'|$ for $j/n\not\in S_{2r}$. 
By definition, there exists $x'\in S$ such that $|x-x'|\leq r(x')$.
If $j/n\le 2x$, then
\[
\sqrt{\frac{D}{n}\frac{j}{n}}
\le \sqrt{\frac{t}{C}2(x'+r(x'))}
\le \sqrt{\frac{4}{C}} r(x')
\leq \sqrt{\frac{4}{C}} L.
\]
Since $|j/n - x'|\ge 2r(x')$ for $j/n \notin S_{2r}$, by triangle inequality, $|j/n - x| \ge r(x')$.
If $j/n \ge 2x$, then $j/n \ge r(x') \ge t \ge C D/n$, and $2(j/n - x)\ge j/n$.
We get
\[
\sqrt{\frac{D}{n}\frac{j}{n}}
\le \sqrt{\frac{1}{C}}\frac{j}{n}
\le 2\sqrt{\frac{1}{C}}\abs{\frac{j}{n} - x}.
\]
Since $|j/n - x_0|\le |j/n - x| + L$ by triangle inequality, we obtain 
$|j/n-x_0|\vee \sqrt{4jD/n^2}\lesssim |j/n-x|+L$.
Then, it follows from \eqref{eq:prf-pois-poly-approx-2} that, for some constant $c_4>0$,
\begin{align*}
\abs{ b_{j}'} 
\leq M_0D c_4^D  \pth{ 1 + \frac{\abs{ j/n- x}}{ L} }^{D}\leq  M_0D \pth{ 2c_4\frac{\abs{ j/n- x}}{ r(x')} }^{D},
\quad j/n \notin S_{2r},
\end{align*}
where the last inequality holds by $1\vee \frac{\abs{ j/n- x}}{ L} \leq \frac{\abs{ j/n- x}}{ r(x')}$.

Furthermore, for $j/n \notin S_{2r}$,
\begin{align*}
\poi(j,nx)  
\stepa{\leq} &\exp \qth{-\frac{1}{3} \pth{\frac{(j-nx)^{2}}{nx} \wedge |j-nx| }}  \\
\stepb{\leq} & \exp \qth{ -\frac{1}{3} \frac{\abs{j-nx}}{r(x')} \pth{\frac{r^2(x')}{x'+r(x')} \wedge r(x') } } \\
\stepc{\leq} & \exp \qth{ -\frac{t}{6} \frac{\abs{j-nx}}{r(x')} },
\end{align*}
where (a) follows from Lemma~\ref{lem:pois-ratio-point} and the fact that $\frac{\delta^2}{2}\geq \frac{\delta^2 \wedge \delta}{3}$;
(b) uses $|x-x'|\leq r(x') \leq \abs{j/n-x}$;
and (c) applies the $t$-large condition \eqref{eq:r_condition}.
Denote $y_j=\abs{j-nx}\geq nr(x')$. 
It follows that
\begin{align*}
    \sum_{j\notin S_{2r}} \abs{ b_j' \poi(j,nx) } &\leq  M_0D \sum_{j\notin S_{2r}} \exp \pth{ -\frac{ty_j }{6r(x')} + D\log \frac{2c_4y_j }{nr(x')} } \\
    &\stepa{\leq}  M_0D \sum_{j\notin S_{2r}} \exp \qth{ - nt \pth{\frac{y_j}{6nr(x')}-\frac{1}{C}\log \frac{2c_4y_j }{nr(x')}} }\\
    &\leq  2M_0D \int_{nr(x')-1}^\infty \exp \qth{ - nt \pth{\frac{y}{6nr(x')}-\frac{1}{C}\log \frac{2c_4y }{nr(x')}} } \diff y \\
    & \stepb{\leq} 2M_0 n \exp \pth{ - c_1 nt},
\end{align*}
where (a) applies $D\le nt /C$, and (b) holds 
since $nr(x')\geq nt\geq CD$ with a large universal constants $C>0$. 
\end{proof}

\subsection{Integral probability metric}
\label{sec:ipm}

Let $\calF$ be a class of real-valued measurable functions. 
The \textit{integral probability metric} (IPM) \cite{muller1997integral} between two probability measures \( P, P' \) with respect to \( \calF \) is defined as 
\begin{align*}
    d_{\calF}(P, P')=\sup _{g \in \calF}\abs{ \Expect_{P} [g]-\Expect_{P'} [g] }.
\end{align*}
The class \( \calF \) can be chosen to represent various commonly-used discrepancies between probability measures. As a typical example, for the class of 1-Lipschitz functions $\calL_1$, $d_{\calL_1}$ is the 1-Wasserstein distance as in \eqref{eq:W1_dual} according to the Kantorovich–Rubinstein theorem \cite[Theorem 1.14]{villani03}. 
Other examples include the total variation distance when $\calF=\{g: \|g\|_{\infty}\leq \frac{1}{2}\}$, and the maximum mean discrepancy when 
$\calF$ is the unit ball of a reproducing kernel Hilbert space.

Regarding our problem of interest, the integral probability metric provides a unified criterion for evaluating the performance of plug-in functional estimators.
Let $G(P)$ and $\hat G$ be the symmetric additive functional and the plug-in estimator
based on the histogram estimate $\hat\pi$ as defined in \eqref{eq:def_func_P} and \eqref{eq:F_hat}, respectively.
By definition, we have
\begin{align*}
    |\hat G- G(P)| \leq k d_{\calF}(\hat \pi, \pi_P), \quad g\in\calF.
\end{align*}

Particularly, for $s\in\naturals,\gamma > 0$, $\boldsymbol{\eta}= (\eta_0,\ldots,\eta_s ) \in \reals_{\geq 0}^{s+1}$, and $C_s=s!$,
consider the following class of continuous functions on $[0,1]$: 
\begin{align}
\label{eq:calF_s}
    \calF_{s,\gamma,\boldsymbol{\eta}} \triangleq \left\{ f : f(x) = x \ell(x),\; \abs{ x^r \ell^{(r)}(x)} \leq C_s x^{\gamma-1}\log^{\eta_r}\pth{1+\frac{1}{x}},\ r = 0,1,\ldots,s \right\},
\end{align}
This family broadly encompasses functions whose derivatives are non-smooth in a neighborhood of zero, including the target functions discussed in Section~\ref{sec:truncated_npmle}.
For instance, $h(x) = -x \log x\in\calF_{2,1,(1,0,0)}$, while $f_\alpha(x) = x^\alpha$, $\alpha \geq 0$ lies in $\calF_{s,\alpha,\boldsymbol{0}}$ for any $s \in \mathbb{N}$. 
Moreover, the integral probability metric associated with \( \calF_{1,1,\boldsymbol{0}} \) is studied as the \textit{relative earthmover distance} in \cite{vv17}.
Specifically, the following lemma holds:
\begin{lemma}
\label{lem:calF_s}
For $f\in \calF_{s,\gamma,\boldsymbol{\eta}}$, the following statements hold: 
\begin{enumerate}
    \item[(i)] $\abs{x^{r}f^{(r)}(x)} \lesssim x^{\gamma}\log^{\eta_r\vee \eta_{r-1}}\pth{1+\frac{1}{x}}$ for $x\in[0,1]$ and $r=0,1,\ldots,s$, where $\eta_{-1}\triangleq 0$.
    \item[(ii)] 
    $M(f,[0,\beta])\lesssim  \|x^\gamma \log^{\eta_0} \pth{1+\frac{1}{x}}\|_{\infty,[0,\beta]}$ for $\beta\in(0,1]$.
    \item[(iii)]  $E_D(f,[0,\beta]) \lesssim
   D^{-r}\beta^{\frac{r}{2}} \|x^{\gamma-\frac{r}{2}} \log^{\eta_r\vee \eta_{r-1}} \pth{1+\frac{1}{x}}\|_{\infty,[0,\beta]}$ for any $D>r$ and $\beta\in (0,1]$.    
\end{enumerate}
\end{lemma}
\begin{proof}
(i) holds by the Leibniz rule $f^{(r)}(x)=x \ell^{(r)}(x) + r \ell^{(r-1)}(x)$ and the definition in \eqref{eq:calF_s}.
(ii) follows from $M(f,[0,\beta])\leq 2\|f\|_{\infty,[0,\beta]}$ and  applying (ii) with $r=0$.
To prove (iii), denote $f_\beta(x)=f(\beta x)$.  Applying Lemma~\ref{lem:poly-DT}, we have
\begin{align}
\label{eq:best-poly-approx-F2-1}
    E_D(f,[0,\beta]) = E_D(f_\beta,[0,1])\lesssim {\omega^{r}_{\varphi}(f_\beta, D^{-1})}.
\end{align}
Also, note that
\begin{align*}
         \omega^{r}_{\varphi}(f_\beta, D^{-1}) &\stepa{\leq} D^{-r} \|\varphi^r f^{(r)}_\beta \|_{\infty,[0,1]} \\
         &\stepb{\leq} D^{-r}\beta^{r/2} \sup_{x\in[0,1]} \abs{(\beta x)^{r/2} f^{(r)}(\beta x)} \\ 
         &\stepc{\lesssim} D^{-r}\beta^{r/2}\sup_{y\in[0,\beta]} \abs{y^{r/2-(r-\gamma)} \log^{\eta_r\vee \eta_{r-1}} \pth{1+\frac{1}{y}} },
\end{align*}
where (a) follows from Lemma~\ref{lem:omega-K-DT}, (b) holds by the definition of $\varphi$, and (c) applies (ii). Finally, the desired result holds. 
\end{proof}

In Appendix~\ref{app:prf_truncated_np}, we upper bound the IPM for  $\calF_{s,\gamma,\boldsymbol{\eta}}$ and establish convergence guarantees for the corresponding functional estimation problems.

\subsection{Tail of Poisson distributions}
\begin{lemma}[{\cite[Theorem 5.4]{MU2017}}]
\label{lem:pois-ratio-tail}
Let $X \sim \Poi(\lambda)$. 
For $\delta > 0$, 
    $$\pbb(X \geq(1+\delta) \lambda) \leq\pth{ \frac{e^{\delta}}{(1+\delta)^{1+\delta}}}^{\lambda} \leq \exp \pth{ -\frac{\pth{ \delta^{2} \wedge \delta} \lambda}{3}};
    $$
For $0<\delta<1$, 
$$
    \pbb(X \leq(1-\delta) \lambda) \leq\pth{ \frac{e^{-\delta}}{(1-\delta)^{1-\delta}}}^{\lambda} \leq \exp \pth{ -\frac{\delta^{2} \lambda}{2}}.
$$
\end{lemma}

\begin{lemma}
\label{lem:pois-ratio-point}
    For $\delta > 0$, 
    $$\sup_{x \geq(1+\delta) \lambda} \sqrt{2\pi x}  \cdot \poi(x,\lambda) \leq \pth{ \frac{e^{\delta}}{(1+\delta)^{1+\delta}}}^{\lambda} \leq  \exp \pth{ -\frac{\pth{ \delta^{2} \wedge \delta} \lambda}{3}};
    $$
    For $0<\delta<1$, 
    $$
    \sup_{0\leq x  \leq(1-\delta) \lambda} (\sqrt{2\pi x} \vee 1)\cdot\poi(x,\lambda) 
    \leq  \pth{ \frac{e^{-\delta}}{(1-\delta)^{1-\delta}}}^{\lambda} \leq  \exp \pth{ -\frac{\delta^{2} \lambda}{2}}.
    $$
\end{lemma}
\begin{proof}
Define $g(t) \triangleq t - (1+t)\log (1+t)$ for $t>-1$. 
The function $g $ is increasing in $(-1,0]$ and decreasing in $[0,\infty)$. 
For any $x \geq(1+\delta) \lambda$, let $x = (1+\delta') \lambda$ with $\delta'\geq \delta$.
Then, applying Stirling's formula \cite{robbins55stirling} $\sqrt{2 \pi n}\left(\frac{n}{e}\right)^{n} <n!$ yields that
\[
\poi(x ,\lambda)
= e^{-\lambda}\frac{\lambda^{x}}{x!} 
\leq  \frac{\pth{\lambda/x}^{x}}{\sqrt{2\pi  x}} e^{x-\lambda} 
= \frac{\exp\pth{\lambda g(\delta')}}{\sqrt{2\pi  x}}
\leq \frac{\exp\pth{\lambda g(\delta)}}{\sqrt{2\pi  x}} 
\leq \frac{1}{\sqrt{2\pi x}} \exp \pth{ -\frac{(\delta^{2} \wedge \delta)\lambda}{3}},
\]
where the last inequality uses Lemma~\ref{lem:pois-ratio-tail}.

Likewise, for any $0< x  \leq(1-\delta) \lambda$ with $\delta\in(0,1)$, there exists $\delta\leq \delta' <1$ such that $x = (1-\delta') \lambda$. Then, we have
\begin{align*}
    \poi(x ,\lambda) & \leq\frac{\exp\pth{\lambda g(-\delta')}}{\sqrt{2\pi  x}}\leq \frac{\exp\pth{\lambda g(-\delta)}}{\sqrt{2\pi  x}} \leq \frac{1}{\sqrt{2\pi x}} \exp \pth{ -\frac{\delta^{2} \lambda}{2}}.
\end{align*}
Finally, if $x=0$, we have $\poi(0,\lambda) = \lim_{\delta'\to 1}\exp\pth{\lambda g(-\delta')} \leq  \exp\pth{\lambda g(-\delta)} \leq \exp ( -\frac{\delta^{2} \lambda}{2}).$
Combining the upper bounds, the desired result follows.
\end{proof}

\begin{lemma}
    \label{lem:pois_rn}
    Let $N\sim \Poi(np)$, $\hat p=N/n$, $p\in[0,1]$, and $r:[0,1] \mapsto [0,\infty)$.
    For any interval $S\subseteq [0,1]$ and $c_2>c_1\geq 0$, there exists a constant $c>0$ depending on $c_1,c_2$ such that
    \begin{align*}
       \sup_{ p \in S_{c_1r}} \pbb\qth{ \hat p \notin S_{c_2r}   } \vee \sup_{ p \notin  S_{c_2r}} \pbb\qth{ \hat p \in S_{c_1r}  } \leq 2\exp\pth{-cn\inf_{x\in [0,1]} \frac{r^2(x)}{x} \wedge r(x)}.
    \end{align*}
\end{lemma}
\begin{proof}
Firstly, fix any $p \in S_{c_1 r}$ and let $\delta = \inf_{y \notin S_{c_2r}} |\frac{y}{p}-1|$. For any $y \notin S_{c_2r}$, there exists $x\in S$ satisfying $|x-p|\leq c_1r(x)$ and $|x-y|> c_2r(x)$, which implies $|p-y|> (c_2-c_1) r(x)$. Letting 
$t \triangleq \inf_{x\in [0,1]} \frac{r^2(x)}{x} \wedge r(x)$, it follows that 
\begin{align*}
\pth{ \delta^{2} \wedge \delta} p \geq \inf_{x\in S} \frac{(c_2-c_1)^2r^2(x)}{x+c_1r(x)} \wedge (c_2-c_1) r(x) \geq \pth{\frac{(c_2-c_1)^2}{2(1+c_1)} \wedge (c_2-c_1)}t.
\end{align*}
Applying Lemma~\ref{lem:pois-ratio-tail} with the fact $\frac{\delta^2}{2}\geq \frac{\delta^2 \wedge \delta}{3}$  yields that for some $c'>0$,
\begin{align*}
 \sup_{ p \in S_{c_1r}} \pbb\qth{ \hat p \notin S_{c_2r}} \leq \pbb\qth{ |\hat p -p| \geq \delta p } \leq 2\exp\pth{-c'nt}.
\end{align*}

Next, fix any $p \notin  S_{c_2r}$.
Let $\delta = \inf_{y \in S_{c_1r}} |\frac{y}{p}-1|$ . 
For any $y \in S_{c_1r}$, there exists $x\in S$ satisfying $|y-x|\leq c_1r(x)$ and $|p-x|> c_2 r(x)$.
If $p\geq \sup_{x'\in S_{c_2r} }x'$, then $y\leq x+c_1r(x)\leq x+c_2r(x)\leq p$, implying that $\delta^2 p \geq \inf_{x\in S}\frac{(c_2-c_1)^2r^2(x)}{x+c_2r(x)}\gtrsim t$.
Otherwise, we have $p\leq \inf_{x'\in S_{c_2r} }x'$ and  $y\geq x-c_1r(x)\geq x-c_2r(x)\geq p$. 
Since $y-p \geq (c_2-c_1)r(x)$ for $y> 2p$ and $\frac{(y-p)^2}{p}\geq \frac{(c_2-c_1)^2r^2(x)}{x}$ for $p\leq y\leq 2p$,
we have  $(\delta^2\wedge\delta)p\gtrsim t$.
Applying Lemma~\ref{lem:pois-ratio-tail}, the upper bound for $\sup_{ p \notin  S_{c_2r}} \pbb\qth{ \hat p \in S_{c_1r}}$ is likewise obtained.
\end{proof}

\subsection{Approximation by finite Poisson mixtures}
\label{app:finite-approx}
Consider the Poisson mixture $f_{P}(\cdot) \triangleq \int \poi(\cdot, \theta) \diff P(\theta)$. 
Let $d(f,g)$ be a function that measures the approximation error of $g$ by $f$, and $\calP_m$ the set of distributions supported on at most $m$ atoms.
Define
\begin{equation*}
  \comp (\epsilon,P,d)\triangleq \min\{m\in\naturals: \exists P_m\in \calP_m, d(f_{P_m},f_P) \leq  \epsilon\},
\end{equation*}
\ie, the smallest order of a finite mixture that approximates a given mixture $f_P$ within a prescribed accuracy $\epsilon$. 
For uniform approximation over a distribution family $\calP$, define
\begin{equation*}
  \comp (\epsilon,\calP,d) \triangleq \sup_{P\in\calP} \comp (\epsilon,P,d),
\end{equation*}

\begin{lemma}
\label{lem:pois-finite-approx}
For $\epsilon \in (0,1/2)$ and $b>a>0$, 
\begin{align*}
    \comp (\epsilon,\calP([a,b]), L_{\infty}) \lesssim 
    (\sqrt{b} - \sqrt{a}) \log^{3/2} \frac{1}{\epsilon} + \log^{2} \frac{1}{\epsilon}. 
\end{align*}
\end{lemma}
\begin{proof}
We construct an approximation of $f_G$ for $G\in \calP([a,b])$. 
Let $\gamma \triangleq C\log(1/\epsilon)$ with a constant $C>0$ to be chosen. 
Define $i_a \triangleq \Floor{\sqrt{a/\gamma}}$ and $i_b \triangleq \Floor{\sqrt{b/ \gamma}}$.
Consider the following partition of $[a,b]$: 
\[
I_i \triangleq \big[a \vee i^{2} \gamma, (i+1)^{2}\gamma \big), \quad i_a \leq i <i_b, 
\]
and $I_{i_b} \triangleq [a \vee i_b^{2} \gamma, b]$.
Let $G_{i}$ be the conditional distribution of $G$ on $I_{i}$.
By Carathéodory theorem, there exists a discrete distribution $G'_i$ supported on $L_{i}$ atoms in $I_{i}$ such that
\begin{equation}
\label{eq:matching-Gi}
\int u^{k} G_{i}(du) = \int u^{k} G'_i(du),
\quad \forall k=1,\ldots,L_{i},
\end{equation}
where $L_i$ is a sequence to be specified.
Define $w_i \triangleq G(I_i)$ and $G' \triangleq \sum_{i=0}^{N} w_{i} G'_i$ that is supported on $m = \sum_{i=1}^N L_i$ atoms. 
Then,
\begin{align*}
    \|f_G-f_{G'}\|_\infty 
    \leq \sum_{i=i_a}^{i_b} w_i \|f_{G_{i}} - f_{G'_i} \|_\infty 
    \le \max_{i\in [i_a,i_b]}\sup_{j} |f_{G_{i}}(j) - f_{G'_i}(j)|.
\end{align*}

Define $r(x)= \frac{1}{2}(\sqrt{\gamma x} + \gamma)$. 
For each $i$, define $\tilde I_{i} \triangleq I_{i-1} \cup I_{i} \cup I_{i+1}$, where $I_{i_a-1}\triangleq[a-r(a),a]$ and $I_{i_b+1}\triangleq [b,b+r(b)]$.
By definition, $(I_i)_r\subseteq \tilde I_{i}$.
Applying Lemma~\ref{lem:pois_rn}, for $j \notin \tilde I_{i}$,
\begin{align*}
    |f_{G_{i}}(j) - f_{G'_i}(j)| 
    \leq \sup_{j \notin (I_i)_r} f_{G_{i}}(j) + f_{G'_i}(j) \leq 4\exp(-c'\gamma)\le \epsilon.
\end{align*}

For $j\in \tilde I_{i}$, let $\poi_j(x)\triangleq x^j e^{-x}/j!$ and $P_{L_i,j} \in \Poly_{L_i}$ be the best polynomial such that $\Norm{\poi_j-P_{L_i,j}}_{\infty,I_{i}}=E_{L_i}(\poi_j,I_i)$. 
By~\eqref{eq:matching-Gi}, 
$\Expect_{G_i} [P_{L_i,j}] = \Expect_{G_i'}[P_{L_i,j}]$. 
Therefore,
\begin{align}
\label{eq:pois-pmf-E-0}
|f_{G_{i}}(j) - f_{G'_i}(j)|
&=\big| \Expect_{G_i}  [\poi_j] - \Expect_{G_i'}[\poi_j] \big| \nonumber\\
  &\leq \big|\Expect_{G_i} [\poi_j - P_{L_i,j}]\big| + \big|\Expect_{G_i'}[\poi_j - P_{L_i,j}]\big|
  \leq  2E_{L_i}(\poi_j,I_i). 
\end{align}
Next, we derive upper bounds on $E_{L_i}(\poi_j,I_i)$ with $j \in \tilde I_{i}$. 

\paragraph{Case 1: $i\leq  \sqrt{\gamma} $.} 
Using the Chebyshev interpolation polynomial (see \cite[Eq.\  (4.7.28)]{Atkinson1989}), we obtain
\begin{align*}
    E_{L_i}(\poi_j,I_i) 
    \leq \frac{\sup_{x \in I_{i}} |\poi_j^{(L_i+1)}(x)| }{2^{L_i}(L_i+1)!}\left(\frac{|I_{i}|}{2}\right)^{L_i+1}.
\end{align*}
When $j \leq L_i+1$, \cite[Eq.\ (3.23)]{WY20Book} shows that $|\poi_j^{(L_i+1)}(x)| \leq e^{-x/2}\binom{L_i+1}{j}$. 
Then, 
\[
E_{L_i}(\poi_j,I_i) \leq 
\frac{\binom{L_i+1}{j}}{2^{L_i}(L_i+1)!}\left(\frac{(2i+1)\gamma}{2}\right)^{L_i+1} \leq \pth{\frac{C_1(i+1)\gamma}{L_i}}^{L_i+1}, 
\]
with some constant $C_1>0$.
For $j \in \tilde I_{i}$, we have $j\le (i+2)^2 \gamma$.
Choosing $L_i= C_2(i+1)^2\gamma \ge j$ for some large constant $C_2$, we obtain  that $E_{L_i}(\poi_j,I_i)\le \epsilon/2$ with $L_i\lesssim \gamma^2$.

\paragraph{Case 2: $i > \sqrt{\gamma}$.}
Note that $I_i \subseteq [j-K\sqrt{j}, j+K\sqrt{j}]$ for $j \in \tilde I_{i}$ with $K \triangleq 3 \sqrt{\gamma}$.
Denote $g_j(x) =  \poi_j(x+j) = c_j \exp(\tilde g_j(x))$ with $\tilde g_j(x) \triangleq j \log (1+ x / j)-x$ and $c_j \triangleq (j / e)^j / j!\leq 1$ by Stirling approximation. 
It follows that
\begin{align}
\label{eq:pois-pmf-E-1}
    E_{L_i}(\poi_j,I_i) \leq E_{L}(\poi_j,[j-K \sqrt{j}, j+K \sqrt{j}]) = E_{L_i}(g_j,[-K \sqrt{j}, K \sqrt{j}]).
\end{align}
We upper bound \eqref{eq:pois-pmf-E-1} by constructing an explicit polynomial approximation. 
Let $k \triangleq \ceil{c'\gamma}$ with $c'>0$ to be chosen and $D=\Floor{L_i/k}$.
Let $G(x) \triangleq \sum_{\ell=0}^D (-x)^\ell/\ell!$ be a degree-$D$ polynomial and $H_k(x) \triangleq \sum_{\ell=1}^k (-1)^{\ell+1} x^\ell/\ell$ be a degree-$k$ polynomial. 
Define 
\[
p(x) \triangleq c_j G(- \tilde g_{j,k}(x)),
\qquad
\tilde g_{j,k}(x) \triangleq j H_k(x / j)-x.
\]
Then $p$ is a polynomial of degree no more than $L_i$.
For the approximation error over $|x|\le K \sqrt{j}$, we have
\begin{align}
|g_j(x) - p(x)|
& \le |g_j(x) - g_{j,k}(x)| + | g_{j,k}(x) - p(x)| \nonumber\\
& \le |e^{\tilde g_j(x)}-e^{\tilde g_{j,k}(x)}|   
+ |e^{-(-\tilde g_{j,k}(x))} - G(-\tilde g_{j,k}(x))|, \label{eq:poi-approx-error}
\end{align}
where $g_{j,k}(x) \triangleq c_j \exp(\tilde g_{j,k}(x))$.

For the first term on the right-side of \eqref{eq:poi-approx-error}, it follows from Taylor's theorem that $|\log (1+x)-H_k(x)| \leq|x|^{k+1}$ for $|x|\le 1/2$, which implies $|\tilde g_{j}(x)-\tilde g_{j,k}(x)| \leq j|x/j|^{k+1}$ for $|x|\le j/2$.
Since $j \geq (i-1)^2 \gamma \geq 36 \gamma = 4K^2$, we have $ |x|\le K \sqrt{j} \le j/2$.
Then, $|\tilde g_{j}(x)-\tilde g_{j,k}(x)| \le j (K\sqrt{j}/j)^{k+1} = K^2 (K/\sqrt{j})^{k-1}\leq K^2 2^{-k+1} \lesssim \epsilon$.
Consequently,
\[
\abs{e^{\tilde g_j(x)}-e^{\tilde g_{j,k}(x)} }
\stepa{\leq} e |\tilde g_{j}(x)-\tilde g_{j,k}(x)|  
\leq \epsilon/2,
\]
where (a) uses $|e^{x}-e^{y}|\leq e^{x\vee y}|x-y|$ and $\tilde g_j(x)\leq 0$.

For the second term on the right-side of \eqref{eq:poi-approx-error}, note that $H_k(y)\leq y$ and $|H_k(y)-y| \leq |H_k(y)-\log(1+y)|+ |y-\log(1+y)| \leq C_3 y^2$ for $|y|\leq \tfrac{1}{2}$ for a constant $C_3$. 
Then, for $|x|\le j/2$, we have $\tilde g_{j,k}(x)\leq 0$ and $\left|\tilde g_{j,k}(x) \right| \leq C_3x^2/j \leq C_3K^2$.
Then,
\[
    \abs{e^{-(-\tilde g_{j,k}(x))} - G(-\tilde g_{j,k}(x))}
    \le \max_{x\in[0,C_3K^2]} \abs{\sum_{j=0}^D \frac{(-x)^j}{j!}-e^{-x}}
    \stepa{\leq} \frac{(C_3K^2)^{D+1}}{(D+1)!}
    \stepb{\le} \epsilon/2,
\]
where (a) uses Taylor's theorem; 
(b) follows by setting $L_i =C_4\gamma^2$ with a large constant $C_4$.
Combining \eqref{eq:pois-pmf-E-1} and \eqref{eq:poi-approx-error} yields $E_{L_i}(\poi_j,I_i)  \leq \epsilon.$

Consequently, under both cases, $G'\in\calP([a,b])$ assigns at most $L_i\leq O(\gamma^2)$ atoms in each subinterval
$I_i$, with the overall $L_\infty$ approximation error at most $\epsilon$.
Hence, we have 
$$\comp (\epsilon,\calP([a,b]), L_{\infty}) \lesssim (i_b-i_a+1) \gamma^2 \lesssim (\sqrt{b}-\sqrt{a}) \gamma^{3/2} + \gamma^{2},$$
which completes the proof.
\end{proof}

For $\epsilon>0$, an $\epsilon$-net of a set $\calF$ with respect to a metric $d$ is a set $\calN$ such that for all $f\in\calF$, there exists $g\in\calN$ such that $d(g,f)\le \epsilon$. 
The minimum cardinality of $\epsilon$-nets is denoted by $N(\epsilon,\calF,d)$.
Define $\calF([a,b]) \triangleq \{f_P:P([a,b])=1\}$ for $a\leq b$.

\begin{lemma}
    \label{lem:npmle-metric-entropy}
   There exists a universal constant $C>0$ such that for $\epsilon \in (0,1/2)$ and $0\leq a\leq b$,
    \[
        \log N(\epsilon,\calF([a,b]),L_{\infty}) \le  C\comp(\epsilon,\calP([a,b]),L_\infty)  \log \frac{b-a+1}{\epsilon^2}.
    \]
\end{lemma}
\begin{proof}
Let $m=\comp(\epsilon,\calP([a,b]),L_\infty)$.
Let $\calN_m\subseteq \Delta_{m-1}$ be an $\epsilon$-net of $\Delta_{m-1}$ under the $L_1$-distance with cardinality $|\calN_m| \leq 2m\left(1+\frac{1}{\epsilon}\right)^{m-1}$ (see, e.g., \cite[Corollary 27.4]{PW-it}).
    Define $\calL\triangleq\{\ceil{\frac{a}{\epsilon}}\epsilon,(\ceil{\frac{a}{\epsilon}}+1)\epsilon,\dots, \floor{\frac{b}{\epsilon}}\epsilon\}$.
    Define the following set of finite mixture densities 
    \begin{align*}
    \calC \triangleq \Bigg\{ \sum_{j=1}^m w_j \Poi(\theta_j) :  (w_1,\dots,w_m) \in \calN_m, \theta_1 \leq \dots \leq \theta_m, \{\theta_j\}_{j=1}^m \subseteq \calL 
    \Bigg\}.
\end{align*}
By applying $\binom{n}{m}\leq (\frac{en}{m})^m$, the cardinality of $\calC$ is upper bounded by 
\begin{align*} 
    |\calC| 
    \leq \binom{m+|\calL|-1}{m}|\calN_m|
    \le \exp\pth{ C m \log \pth{\frac{b-a+1}{\epsilon^2}}}.
\end{align*}

    Next, we prove $\calC$ is an $\epsilon$-net. 
    By definition of $m^*$, for any $P\in\calP([a,b])$, there exists $P_m=\sum_{j=1}^m w_j\delta_{\theta_j}$ with $a\leq \theta_1 \le \dots \le \theta_m\leq b$ such that $\|f_{P_m}-f_P\|_{\infty}\leq \epsilon$.
    Let $\theta_j^\prime \triangleq \theta_j \frac{\floor{|\theta_j|/\epsilon}}{|\theta_j|/\epsilon} \in \calL$ and choose $w^\prime\in\calN_m$ so such $\Norm{w-w^\prime}_1 \le \epsilon$.
    Define $P_m' \triangleq \sum_{j=1}^m w_j\delta_{\theta_j'}$ and $P_m''=\sum_{j=1}^m w^\prime_j\delta_{\theta_j^\prime}\in\calC$.
    Then,
    \[
    \|f_P - f_{P_m^{\prime\prime}}\|_{\infty}
    \le \|f_P - f_{P_m}\|_{\infty}
    +\|f_{P_m}-f_{P_m^\prime}\|_{\infty}
    +\|f_{P_m^{\prime}}-f_{P_m^{\prime\prime}}\|_{\infty}.
    \]
    Note that $\poi_i(\cdot)$ is 1-Lipschitz by $ |\poi'_i(x)| = |\poi_{i-1}(x)- \poi_i(x)| \leq 1$.
    Applying triangle inequality, we obtain $\|f_{P_m}-f_{P_m^\prime}\|_{\infty}\le \sup_{j}|\theta_j-\theta'_j| \leq \epsilon.$
    By triangle inequality, $\|f_{P_m^{\prime}}-f_{P_m^{\prime\prime}}\|_{\infty}\leq \Norm{w-w^\prime}_1 \leq \epsilon.$
    Hence, $\calC$ is a $3\epsilon$-net of $\calF([a,b])$ under the $L_\infty$ distance.
    Replacing $3\epsilon$ with $\epsilon$ yields the desired result.
\end{proof}

\section{Proofs in Section~\ref{sec:npmle-basic} }
\label{app:prf_basic_props}

\subsection{Proof of Proposition~\ref{prop:weight_remainder_asymp}}
\label{app:npmle-weight}
Define the following event
\begin{align}
    \label{eq:A'}
    A= \{  | \hat p_i - p_i | \leq r(p_i)/2,\ \forall i\in [k]\}.
\end{align}
Applying Lemma~\ref{lem:pois_rn}, there exists a universal $c_0>0$ such that $P[A^c]\leq 2k\exp(-c_0nt)$.
In the following, we prove that (a)--(c) hold that under the condition that $A$ occurs.

First, we prove (a).  
Let $\varepsilon \triangleq \exp(-cnt)$ for some $c$ to be specified.
It suffices to show that, under the event $A$, any distribution in $\Pi \triangleq \{ \pi\in \calP([0,1]) : \pi_P(S) > \hat \pi(S_r) (1+\varepsilon) \}$ is suboptimal. 
In particular, we show that~\eqref{eq:1st_opt} cannot simultaneously hold for all $Q\in\calP([0,1])$.
The condition~\eqref{eq:1st_opt} with $Q=\delta_{\hat p_i}$ for $i\in I_S \triangleq \{i\in [k]: p_i\in S\}$ yields
\begin{align*}
    k \geq \sum_{j=1}^k \frac{f_Q(N_j)}{f_\pi(N_j)} \geq \frac{f_Q(N_i)}{f_\pi(N_i)} = \frac{\poi(N_i,N_i)}{f_\pi(N_i)}.
\end{align*}
If $N_i=0$, then $\poi(N_i,N_i)=1$; If $N_i\ge 1$, it follows from the Stirling's formula \cite{robbins55stirling} that $\poi(N_i,N_i)\ge c'/\sqrt{N_i} $ for some constant $c'$.
Then, 
\begin{align}
\label{eq:npmle-remainder-2}
    f_{\pi}(N_i) \geq \frac{1}{k}\poi(N_i,N_i) \geq \frac{c'}{k(\sqrt{N_i} \vee 1)}.
\end{align}

Let $\hat \mu$ denote the NPMLE \eqref{eq:pois-npmle-def} given a subset of frequencies $\{N_i: i\in I_S\}$.
Next, we show that \eqref{eq:1st_opt} fails to hold for $Q=\hat \mu$.
Define $w\triangleq\pi (S_r)$ and $w^\star\triangleq\pi_P(S)$.
Denote $\pi|_S$ as the conditional distribution of $\pi$ on a given measurable set $S$.  
By definition, $f_{\pi} = w f_{\pi|_{S_r } } + (1-w) f_{\pi|_{(S_r)^c}}.$
Then, 
\begin{align}
\label{eq:npmle-decomp}
    \sum_{i=1}^k \frac{f_{\hat \mu}(N_i)}{ f_{\pi}(N_i)} \geq \sum_{i\in I_S} \frac{  f_{\hat \mu}(N_i)}{ f_{\pi}(N_i)} = \sum_{i\in I_S}\frac{ f_{\hat \mu}(N_i)}{ wf_{\pi|_{S_r }}(N_i)} - \sum_{i\in I_S}
    \frac{f_{\hat \mu}(N_i) }{f_{\pi}(N_i)}
    \pth{\frac{f_{\pi}(N_i)}{ wf_{\pi|_{S_r }}(N_i)}-1 }.
\end{align}
By the optimality of $\hat\mu$, we obtain from~\eqref{eq:1st_opt_2} that
$$
\sum_{i\in I_S} \frac{ f_{\hat \mu}(N_i)}{ wf_{\pi|_{S_r }} (N_i)} \geq \frac{|I_S|}{w}
= \frac{kw^\star}{w}
> k(1+\varepsilon).
$$

Next we upper bound the second term on the right-hand side of~\eqref{eq:npmle-decomp}.
Note that
\begin{equation}
\label{eq:outside-bound}
\frac{f_{\pi}(N_i)}{ wf_{\pi|_{S_r }}(N_i)}-1
=\frac{(1-w)f_{\pi|_{(S_r)^c}}(N_i)}{f_{\pi}(N_i) - (1-w)f_{\pi|_{(S_r)^c}}(N_i) }
\le \frac{\sup_{\theta\in (S_r)^c } \poi(N_i,n\theta)}{f_{\pi}(N_i) - \sup_{\theta\in (S_r)^c } \poi(N_i,n\theta) }.
\end{equation}
Note that $\hat p_i \in S_{r/2}$ for $i\in I_S$ under the event $A$.
For $\theta\in (S_r)^c$, define $\delta\triangleq |\frac{\hat p_i}{\theta}-1|$.
By definition, there exists $x\in S$ satisfying $|\hat p_i-x|\leq r(x)/2$ and $|\theta-x|> r(x)$.
If $x \leq \theta$, then $\hat p_i\leq x+r(x)/2\leq x+r(x)\leq \theta$, implying that $\delta^2 \theta = (1-\frac{\hat p_i}{\theta})^2 \theta \geq \frac{r^2(x)}{4(x+r(x))}\geq \frac{t}{8}$.
If $x > \theta$, we have $\hat p_i\geq x-r(x)/2\geq x-r(x)\geq \theta$. 
Then, for $\hat p_i> 2\theta$, $(\delta^2\wedge\delta)\theta=\hat p_i-\theta \geq \frac{r(x)}{2}\geq \frac{t}{2}$;
For $\theta\leq \hat p_i\leq 2\theta$,
$(\delta^2\wedge\delta)\theta=\frac{(\hat p_i-\theta)^2}{\theta}\geq \frac{(r(x)/2)^2}{x} \geq \frac{t}{4}$.
Applying Lemma~\ref{lem:pois-ratio-point} yields that, for some universal constant $c_1>0$, 
\begin{align}
\label{eq:npmle-remainder-1}
    \sup_{\theta\in (S_r)^c } \poi(N_i,n\theta)
    \leq \frac{1}{\sqrt{2\pi N_i} \vee 1}\exp(-c_1nt).
\end{align}
Therefore, combining \eqref{eq:outside-bound} and \eqref{eq:npmle-remainder-1}, for $i\in I_S$, 
\begin{align}
\label{eq:npmle-remainder-3}
    \frac{f_{\pi}(N_i)}{ wf_{\pi|_{S_r }}(N_i)}-1
    \stepa{\leq} \frac{ \frac{1}{\sqrt{2\pi N_i} \vee 1} 2\exp(-c_1nt)  }{\frac{c'}{k(\sqrt{N_i} \vee 1)}-\frac{1}{\sqrt{2\pi N_i} \vee 1}2\exp(-c_1nt) } 
    \stepb{\leq} \exp\pth{-c_2nt},
\end{align}
where (a) follows from \eqref{eq:npmle-remainder-2} and \eqref{eq:npmle-remainder-1}, and (b) holds for some constant $c_2$ for $t>C\frac{\log k}{n}$ with a large constant $C>0$. 
Letting $c=c_2$, by \eqref{eq:npmle-decomp},
\begin{align*}
    \sum_{i=1}^k \frac{f_{\hat \mu}(N_i)}{ f_{\pi}(N_i)} \geq  \sum_{i\in I_S} \frac{  f_{\hat \mu}(N_i)}{ f_{\pi}(N_i)} 
    \geq \frac{1}{1+\epsilon}
    \sum_{i\in I_S}\frac{ f_{\hat \mu}(N_i)}{ wf_{\pi|_{S_r }}(N_i)} >  \frac{k(1+\varepsilon)}{1+\varepsilon} =  k.
\end{align*}
Consequently, (a) follows.

Then, we prove (b). Let $S'=(S^{c,r})^c$ satisfy $S_r \cap S'_r =\emptyset$. Denote 
 $v=\pi(S'_r)$, $v^\star=\pi_P(S')$, and $u^\star=\pi_P((S\cup S')^c)$.
 By definition, $w+v\leq 1$ and $u^\star+v^\star+w^\star=1$.
 Applying (a) to $S'$ yields that
 $ v^\star - v \leq \frac{\epsilon}{1+\epsilon} v^\star  \leq \epsilon v^\star \leq \epsilon$. 
Then, we have 
\begin{align*}
 w \leq 1- v = (1- v^\star) + (v^\star - v) \leq w^\star + u^\star
  + \epsilon = \pi_P(S'^c) + \epsilon ,
\end{align*}
which gives the result.

Finally, we prove (c).
We show that under the event $A$, \eqref{eq:1st_opt} and~\eqref{eq:1st_opt_2} cannot simultaneously hold for any distribution in $\Pi' \triangleq \{ \pi\in \calP([0,1]) \mid  \pi (S_r ) < 1 \}$.
Suppose that $A$ occurs, and \eqref{eq:1st_opt} holds for some $\pi \in \Pi'$.  
Applying~\eqref{eq:npmle-remainder-2} and~\eqref{eq:npmle-remainder-1}, for each $i\in [k]$, 
\begin{align*}
\frac{f_{\pi|_{(S_r)^c}}(N_i)}{f_{\pi|_{S_r }}(N_i)} & \leq \frac{ f_{\pi|_{(S_r)^c}}(N_i) }{f_{\pi}(N_i)-f_{\pi|_{(S_r)^c}}(N_i)} \leq \frac{ 2\exp(-c_1nt)  }{c'k^{-1}-2\exp(-c_1nt) } < 1.
\end{align*}
where the last inequality holds since $t>C\frac{\log k}{n}$ with a large constant $C>0$.
Since $w<1$, we have 
\begin{align*}
\sum_{i=1}^k \frac{f_{\pi}(N_i)}{f_{\pi|_{S_r }}(N_i)} 
= kw + (1-w)  \sum_{i=1}^k \frac{f_{\pi|_{(S_r)^c}}(N_i)}{f_{\pi|_{S_r }}(N_i)}  < k,
\end{align*}
which violates the optimality condition \eqref{eq:1st_opt_2} with $Q=\pi|_{S_r}$.
Consequently, (c) holds.

\subsection{Proof of Proposition~\ref{prop:pi_p_minimax_lb}}
By Le Cam’s two-point method (see, \eg, \cite[Sec.~2.4.2]{Tsybakov2009IntroductionTN}), for $P,Q \in \Delta_{k-1}$,
\begin{align}
\label{eq:2-point-lb}
    \inf_{\hat f}\sup_{P\in \Delta_{k-1}} \Expect H^2 (\hat f, f_{\pi_P}) \geq \frac{H^2(f_{\pi_P},f_{\pi_Q})}{4} \exp (- \KL( \otimes_{i=1}^k \Poi(np_i) \| \otimes_{i=1}^k \Poi(nq_i)) ).
\end{align}
Set $P=(p_1,p_2,\ldots,p_k)=(\frac{1-\epsilon}{3},\frac{2+\epsilon}{3},0,\ldots,0)$ and $Q=(q_1,q_2,\ldots,q_k)=(\frac{1}{3},\frac{2}{3},0,\ldots,0)$, where $\epsilon = c_0n^{-1/2}$ for some $c_0$ to be chosen. We have 
\begin{align*}
    \KL( \otimes_{i=1}^k \Poi(np_i) \| \otimes_{i=1}^k \Poi(nq_i)) 
    \stepa{=} \sum_{i=1}^2 n\pth{p_i\log\frac{p_i}{q_i} - p_i + q_i} 
    = \frac{n}{4} (\epsilon^2+O(\epsilon^3)) \asymp 1.
\end{align*}
where (a) uses the identity $\KL\left(\Poi(\lambda_1) \| \Poi(\lambda_2)\right)=\lambda_1 \log \frac{\lambda_1}{\lambda_2}-\lambda_1+\lambda_2$.
Moreover, by letting $w_1=w_2=\frac{1}{k}$ and $w_3=\frac{k-2}{k}$, we get
\begin{align*}
    \frac{1}{2}H^2(f_{\pi_P},f_{\pi_Q}) 
    &= 1-\sum_{j=0}^\infty 
    \sqrt{ 
    \pth{\sum_{i=1}^3 w_i\poi(j,np_i) } \pth{\sum_{i'=1}^3 w_{i'} \poi(j,nq_{i'}) } 
    } \\
    &\geq 1 - \qth{ \sum_{i,i'=1}^3 \sum_{j=0}^\infty  \sqrt{ w_i w_{i'} \poi(j,np_i) \poi(j,nq_{i'})} }    \\
    &= \frac{1}{k} \sum_{i=1}^2 \frac{1}{2}H^2(\Poi(np_i),\Poi(nq_i))-\sum_{i\neq i'} \sum_{j=0}^\infty \sqrt{  w_i w_{i'} \poi(j,np_i) \poi(j,nq_{i'})}.
\end{align*}
Applying the identity $\frac{1}{2}H^2 (\operatorname{Poi}(\lambda_1), \operatorname{Poi}(\lambda_2))=1-\exp (- \frac{(\sqrt{\lambda_1}-\sqrt{\lambda_2})^2}{2})$ yields
\begin{align*}
    \frac{1}{2}H^2(\Poi(np_i),\Poi(nq_i))]
    &= 1-\exp \left(- \frac{(n\epsilon/3)^2}{2 (\sqrt{np_i}+\sqrt{nq_i})^2}\right) \geq 1-\exp\pth{-c_1n\epsilon^2}, \quad i=1,2; \\
    \sum_{j=0}^\infty \sqrt{  \poi(j,np_i) \poi(j,nq_{i'})} &= \exp \left(- \frac{(\sqrt{np_i}-\sqrt{nq_{i'}})^2}{2}\right) \leq \exp\pth{-c_2n}, \quad i,i'\in [3],\ i\neq i',
\end{align*}
for some universal constants $c_1,c_2$. 
Hence, there exist $c_0>0$ such that
$H^2(f_{\pi_P},f_{\pi_Q}) \gtrsim 1/k$ when $n\gtrsim \log k$.
Applying \eqref{eq:2-point-lb}, the desired result follows.

\section{Proofs in Section~\ref{sec:theory}}
\label{app:prf_theory}

\subsection{Proofs in Section~\ref{sec:theory-asymp}}
\label{app:prf_thm_npmle-asymp}

\begin{proof}[Proof of Theorem~\ref{thm:npmle-asymp}]
Let $\hat F$ and $F^{\star}$ denote the cumulative distribution function (CDF) of $\hat\pi$ and $\pi_P$, respectively. 
The quantile coupling formula \cite[Eq.~(2.52)]{villani03} yields that
\begin{align}
    W_{1}(\hat\pi,\pi_P)= \int_{0}^{1}\left|\hat F^{-1} (u)-{F^{\star}}^{-1}(u)\right| \diff u,
\end{align}
where $F^{-1}(u) \triangleq \inf\{t : F(t) \geq u\}$ for $ u \in (0, 1)$ is the quantile function of a CDF $F$.

Let $0\leq q_1<\ldots< q_L\leq 1$ be all distinct values in $(p_1,\ldots,p_k)$.
For any $\delta\in(0,1)$, let
\[
\epsilon_1\triangleq \frac{C}{n} \log \frac{2k}{\delta},
\qquad 
\epsilon_2 \triangleq \pth{\frac{1}{4} \min_{i\neq j} |q_i-q_j|}^2.
\]
Define $r_j(x)\triangleq \sqrt{x\epsilon_j} + \epsilon_j$ that satisfies $\inf_{x\in [0,1]} \frac{r_j^2(x)}{x} \wedge r_j(x) \geq \epsilon_j$, and let $I_{\ell,j}= [q_\ell-r_j(q_\ell),q_\ell+r_j(q_\ell)] \triangleq [q_{\ell,j}^{\mathrm{L}}, q_{\ell,j}^{\mathrm{U}}]$. 
Applying Proposition~\ref{prop:weight_remainder_asymp}(c) with $r_1$ yields that, with probability $1-\delta$,
\begin{align}
\label{eq:prf-npmle-asymp-0}
    \hat \pi(\cup_{\ell=1}^L I_{\ell,1}) = 1.
\end{align}
For any $q_i>q_j$, we have $q_i-r_2(q_i)-r_2(q_j) > q_i-4\sqrt{\epsilon_2} \geq q_j$, which implies that the intervals $I_{\ell,2}$ are disjoint.
Then $q_j \in I_{\ell,2}^{c,r_2}$ for all $j\neq \ell$.
Applying Proposition~\ref{prop:weight_remainder_asymp}(b) with $r_2$ yields that, with probability $1-2k\exp(-c_1n)$,
\begin{align}
\label{eq:prf-npmle-asymp-1}
    \hat \pi(I_{\ell,2}) \leq \pi_P((I_{\ell,2}^{c,r_2})^c) + \delta' = \pi_P(q_\ell) + \delta', \quad \forall \ell\in [L],
\end{align} 
where $\delta'\triangleq \exp(-c_2n)$.

Next, we upper bound the difference $|\hat F^{-1} (u)-{F^{\star}}^{-1}(u)|$ under the events~\eqref{eq:prf-npmle-asymp-0} and \eqref{eq:prf-npmle-asymp-1} that occur with probability $1-\delta-2k\exp(-c_1n)$.
There exists $N=N_\delta$ such that $\epsilon_1<\epsilon_2$ for all $n\ge N$. 
Then, $\hat \pi(I_{\ell,1}) \leq \hat \pi(I_{\ell,2})\leq \pi_P(q_\ell) + \delta'$.  
With the notation $u_0^\star = 0$ and $u_\ell^\star = F^{\star}(q_{\ell})$, $\ell\in[L]$, we have 
\begin{align*}
    \hat F(q_{\ell,1}^{\mathrm{L}}) &= 
    \hat \pi \pth{ \cup_{j=1}^{\ell-1} I_{\ell,1}}
    = \sum_{j=1}^{\ell-1} \hat \pi \pth{I_{\ell,1}}
    \leq \sum_{j=1}^{\ell-1} (\pi^{\star}(q_\ell) + \delta') \leq u_{\ell-1}^\star + k \delta', \\
    \hat F(q_{\ell,1}^{\mathrm{U}}) &= 
    1 - \hat \pi \pth{ \cup_{j=\ell+1}^{L} I_{\ell,1} }   \geq 1 - \sum_{j=\ell+1}^{L} (\pi^{\star}(q_\ell) + \delta')  \geq u_\ell^\star - k \delta'.
\end{align*}
Then, for $u\in (u_{\ell-1}^\star+k\delta',u_\ell^\star-k\delta')$, we have $\hat F^{-1}(u) \in [q_{\ell,1}^{\mathrm{L}}, q_{\ell,1}^{\mathrm{U}}]$ and ${F^{\star}}^{-1}(u)=q_\ell$.
Hence, 
\begin{align*}
W_{1}(\hat\pi,\pi_P) 
&=\sum_{\ell = 1}^{L} \int_{  (u_{\ell-1}^\star,u_\ell^\star] }\left|\hat F^{-1} (u)-{F^{\star}}^{-1}(u)\right| \diff u  \\
&\le \sum_{\ell = 1}^{L} \pth{ \int_{  (u_{\ell-1}^\star+k\delta',u_\ell^\star-k\delta') }\left|\hat F^{-1} (u)-{F^{\star}}^{-1}(u)\right| \diff u +2k\delta'}  \\
&\leq \sum_{\ell = 1}^{L} (u_\ell^\star-u_{\ell-1}^\star)\cdot r_1(q_\ell)+2k^2\delta'  \\
&\leq \pth{\sqrt{\frac{C}{n}\log \frac{2k}{\delta}} + \frac{C}{n}\log \frac{2k}{\delta}} + 2k^2\delta',
\end{align*}
which completes the proof.
\end{proof}

\subsection{Proofs in Section~\ref{sec:theory-main}}
\label{app:prf_thm_npmle-global}

\begin{lemma}[Hellinger rate for constrained approximate NPMLE]
\label{lem:pois-hellinger-rate-c}
Suppose that $X_i \inddistr \Poi(\theta_i)$ for $i\in [n]$, and  $\{\theta_i\}_{i=1}^n \subseteq [a,b]$.
Let $\pi^\star=\frac{1}{n}\sum_{i=1}^{n}\delta_{\theta_i}$ and 
\[
    \epsilon_n^2
    = \frac{ ( \sqrt{b}-\sqrt{a}+\sqrt{\log (n(b+1))} ) \log^{\frac{5}{2}} (n(b+1))}{n} \vee 1.
\]
There exist constants $s^\star,c^\prime>0$ such that for any $s \geq s^\star$,\footnote{Here we adopt a slight abuse of notation by letting  $f_{P}(\cdot) \triangleq \int \poi(\cdot, \theta) \diff P(\theta)$ in the statement and proof of Lemma~\ref{lem:pois-hellinger-rate-c}, in contrast to the definition $f_\pi = \int \poi(\cdot, nr) \diff \pi(r)$ as in~\eqref{eq:pois-npmle-def} throughout the paper.} 
\[
\sth{ \pi\in \calP([a,b]): \frac{1}{n}\sum_{i=1}^n \log \frac{f_{\pi}}{f_{\pi^\star}}(X_i) \geq  - c_0 \epsilon_n^2 } 
\subseteq \sth{ \pi \in \calP([a,b]): H(f_\pi, f_{\pi^\star})< s\epsilon_n },
\]
under an event that occurs with probability $1-n^{-c^\prime s^2}$. 
\end{lemma}
\begin{proof}
It suffices to consider the case $(\sqrt{b}-\sqrt{a}+\sqrt{\log n(b+1)} ) \frac{\log^{\frac{5}{2}} n}{n} \lesssim 1$.
Define 
$$
\calF \triangleq 
\{f_\pi: \pi \in \calP([a,b]), H(f_\pi, f_{\pi^\star})\ge s\epsilon_n \}.
$$
Let $\epsilon = n^{-2} (b+1)^{-1}$. 
By Lemmas \ref{lem:pois-finite-approx} and \ref{lem:npmle-metric-entropy}, there exists an $\epsilon$-net $\calN_\epsilon$ of $\calF$ under the $L_\infty$-norm of cardinality $H_\epsilon\triangleq \log |\calN_\epsilon|\lesssim \comp\log (n(b+1)) \lesssim n \epsilon_n^2$, where $m^\star\triangleq m^\star(\epsilon,\calP([a,b]),L_\infty)$.
Consider the following event
\[
E\triangleq \sth{\max_{g\in\calN_\epsilon}\frac{1}{n}\sum_{i=1}^n \log \frac{g+\epsilon }{f_{\pi^\star}}(X_i) < - c_0 \epsilon_n^2  }.
\]
For any $\pi \in \calP([a,b])$ such that $f_\pi\in\calF$, there exists $g\in\calN_\epsilon$ such that $f_{\pi}(x)\le g(x)+\epsilon$ for all $x\in\reals$.
However, under the event $E$, we have
\[
\frac{1}{n}\sum_{i=1}^n \log \frac{f_{\pi}(X_i)}{f_{\pi^\star}(X_i)}   
\le \max_{g\in\calN_\epsilon}\frac{1}{n}\sum_{i=1}^n \log \frac{g+\epsilon }{f_{\pi^\star}}(X_i)
< - c_0 \epsilon_n^2.
\]

It remains to upper bound $\Prob[E^c]$. 
For a fixed function $g\in\calN_\epsilon$, applying the Chernoff bound yields that 
\[
\prob{\frac{1}{n}\sum_{i=1}^n \log \frac{g+\epsilon }{f_{\pi^\star}}(X_i)\ge - c_0 \epsilon_n^2}
\le \exp\pth{ \frac{c_0 n \epsilon_n^2}{2} + \sum_{i=1}^n \log \Expect\sqrt{\frac{g+\epsilon }{f_{\pi^\star}} (X_i)} }.
\]
Note that
\begin{align*}
\frac{1}{n}\sum_{i=1}^n \log \Expect\sqrt{\frac{g + \epsilon }{f_{\pi^\star}} (X_i)} &\leq 
 \frac{1}{n}\sum_{i=1}^n \Expect\sqrt{\frac{g + \epsilon }{f_{\pi^\star}} (X_i)} - 1 = \Expect_{X\sim f_{\pi^\star}}\sqrt{\frac{g+\epsilon }{f_{\pi^\star}} (X)} - 1
 \\
 &\leq \Expect_{X\sim f_{\pi^\star}}\sqrt{\frac{g}{f_{\pi^\star}} (X)} - 1 + \Expect_{X\sim f_{\pi^\star}}\sqrt{\frac{\epsilon}{f_{\pi^\star}}(X) } \\
 &= -\frac{H^2(g,f_{\pi^\star})}{2} + \sum_{j=0}^\infty \sqrt{\epsilon f_{\pi^\star}(j)}. 
\end{align*}
Since $g\in\calN_\epsilon$, we have  $H^2(g,f_{\pi^\star})\ge (s\epsilon_n)^2$.
For the second term, applying Cauchy-Schwarz inequality yields $\sum_{j\in[0,b]}\sqrt{f_{\pi^\star}(j)} \leq \sqrt{b+1}$. Moreover, 
\begin{align*}
    \sum_{j>b} \sqrt{f_{\pi^\star}(j)} &\stepa{\leq} \sum_{j>b} \sqrt{\poi(j,b)} 
    \stepb{\leq} 2 \sum_{j>b/2}\sqrt{\poi(2j,b)} =  2 \sum_{j>b/2}\frac{ b^{j} }{\sqrt{(2j)!}} e^{-b/2}\\
    & \stepc{\lesssim}  \sum_{j>b/2} j^{1/4}\frac{ (b/2)^{j} }{j!} e^{-b/2} 
    \leq \Expect_{X\sim\Poi(\frac{b}{2})} [X^{\frac{1}{4}}]
    \leq \pth{\Expect X}^{\frac{1}{4}}
    \leq b^{^{\frac{1}{4}}},
\end{align*}
where (a) follows from $f_{\pi^\star}(j) \leq \sup_{\theta\in[0,b]} \poi(j,\theta)\leq \poi(j,b)$ for $\pi^\star \in \calP([a,b])$ and $j\geq b$; (b) uses $\poi(2j+1,b)\leq\poi(2j,b)$ for $j\geq b$;
(c) holds by $\frac{(2j)!}{(j!)^2}=\binom{2j}{j}\ge \frac{2^{2j}}{\sqrt{4j}}$.
Hence, we get $\sum_{j=0}^\infty \sqrt{f_{\pi^\star}(j)} \leq c''(\sqrt{b+1})$ for some universal constant $c''>0$.
Then,
\[
\prob{\frac{1}{n}\sum_{i=1}^n \log \frac{g+\epsilon }{f_{\pi^\star}}(X_i)\ge - c_0 \epsilon_n^2}
\le \exp\pth{ n \pth{\frac{c_0 \epsilon_n^2}{2} - \frac{s^2\epsilon_n^2}{2} + c''\sqrt{\epsilon(b+1)} } }.
\]
Moreover, with $\epsilon = n^{-2} (b+1)^{-1}$, we have that $\epsilon_n^2\geq \frac{1}{n} = \sqrt{\epsilon (b+1)}.$ Applying the union bound, there exist absolute constants $c_1,s^\star>0$ such that for any $s>s^\star$,
\begin{align*}
\pbb[E^c]&=\prob{\max_{g\in\calN_\epsilon}\frac{1}{n}\sum_{i=1}^n \log \frac{g+\epsilon }{f_{\pi^\star}}(X_i) \ge - c_0 \epsilon_n^2 } \\
    &\le  \exp\pth{ -n \pth{\frac{s^2\epsilon_n^2}{2} - c''\sqrt{\epsilon(b+1)} - \frac{c_0 \epsilon_n^2}{2}}  + H_\epsilon} \\
    &\le  \exp\pth{- c_1 s^2 m^\star \log  (n(b+1))}.\qedhere
\end{align*}
\end{proof}

In the next two lemmas, suppose $P\in\Delta_{k-1}$ satisfies Assumption \ref{ass:separation}.
Denote $\hat\pi_{\ell}$ and $\pi_{P,\ell}$ as the conditional distribution of $\hat\pi$ and $\pi_P$ on $I_{\ell}\triangleq \{q_\ell\}_{r^\star_t} = [q_\ell-r^\star_t(q_\ell), q_\ell+r^\star_t(q_\ell)]$, respectively.
Let $k_\ell\triangleq\sum_{i=1}^k\indc{p_i\in I_\ell}$. 

\begin{lemma}
\label{lem:approximate-np}
There exist universal constants $C, c, c_0 > 0$ such that if $t\geq\tfrac{C\log k}{n}$ and $s\leq \tfrac{t}{2}$, then, with probability $1-2k\exp(-c_0nt)$,
\begin{align*}
   \sum_{i: p_i \in I_\ell} \log\frac{ f_{\hat \pi_{\ell}} (N_i)}{f_{\pi_{P,\ell}}(N_i)} 
   \geq -k\exp(-cnt), \quad \forall \ell\in[L].
\end{align*}
\end{lemma}
\begin{proof}
Define the constrained NPMLE over the interval $I_\ell$ given input $\{N_i: p_i\in I_\ell\}$ as 
\[
\hat \mu_\ell \triangleq \argmax_{\pi\in\calP(I_\ell)} \sum_{i:p_i\in I_\ell} \log f_\pi (N_i).
\]
Let $w_\ell=\hat\pi(I_\ell)$ and $\hat \pi_\ell'$ denote the conditional distribution of $\hat\pi$ on $(I_\ell)^c$.
Then, $\hat\pi = w_\ell\hat \pi_\ell+(1-w_\ell) \hat \pi_\ell'$. 
Letting $\nu_\ell \triangleq w_\ell\hat \mu_\ell + (1-w_\ell) \hat \pi_\ell'$. By the optimality condition \eqref{eq:0th_opt} of $\hat \pi$, we have 
\begin{align*}
   0\leq \sum_{i=1}^k \log \frac{f_{\hat \pi}(N_i)}{f_{\nu_\ell}(N_i)} &=  \sum_{i: p_i \in I_\ell} \log \frac{f_{\hat \pi}(N_i)}{f_{\nu_\ell}(N_i)}+ \sum_{i: p_i \notin I_\ell} \log \frac{f_{\hat \pi}(N_i)}{f_{\nu_\ell}(N_i)}\\
   &\leq  \sum_{i: p_i \in I_\ell} \log \frac{f_{\hat \pi}(N_i)}{w_\ell f_{\hat \mu_{\ell}}(N_i)}
   + \sum_{i: p_i \notin I_\ell} \log \frac{f_{\hat \pi}(N_i)}{ (1-w_\ell) f_{\hat \pi_\ell'}(N_i) }\\
   &= \sum_{i: p_i \in I_\ell} \log\frac{f_{\hat \pi}(N_i)}{w_\ell f_{\hat \pi_{\ell}}(N_i)} 
    + \sum_{i: p_i \in I_\ell} \log\frac{ f_{\hat \pi_{\ell}} (N_i)}{f_{\hat \mu_{\ell}}(N_i)} 
    + \sum_{i: p_i \notin I_\ell} \log \frac{f_{\hat \pi}(N_i)}{ (1-w_\ell) f_{\hat \pi_\ell'}(N_i)}.
\end{align*}
Let $A=\{|\hat p_i - p_i | \leq r^\star_t(p_i)/2,\ \forall i\in [k]\} $ be defined in \eqref{eq:A'} with $t\ge \frac{C\log k}{n}$ such that  $P[A^c]\leq 2k\exp(-c_0nt)$.
Following the derivation in \eqref{eq:npmle-remainder-3}, under the event $A$, we have 
\begin{align}
\label{eq:prf-approx-np-1}
  \sup_{p_i \in I_\ell} \frac{f_{\hat \pi}(N_i)}{w_\ell f_{\hat \pi_{\ell}}(N_i)} \vee \sup_{ p_i \notin I_\ell} \frac{f_{\hat \pi}(N_i)}{ (1-w_\ell) f_{\hat \pi_\ell'}(N_i)} \leq 1+ \exp(-cnt)
\end{align}
with a universal constant $c>0$, which implies that
\begin{align*}
    \sum_{i: p_i \in I_\ell} \log\frac{f_{\hat \pi}(N_i)}{w_\ell f_{\hat \pi_{\ell}}(N_i)} 
    + \sum_{i: p_i \notin I_\ell} \log \frac{f_{\hat \pi}(N_i)}{ (1-w_\ell) f_{\hat \pi_\ell'}(N_i)} \leq k\log (1+ \exp(-cnt))
    \leq k\exp(-cnt).
\end{align*}
Then, by the optimality condition \eqref{eq:0th_opt} of $\hat \mu$, we have 
\[
   \sum_{i: p_i \in I_\ell} \log\frac{ f_{\hat \pi_{\ell}} (N_i)}{f_{\pi_{P,\ell}}(N_i)} \geq \sum_{i: p_i \in I_\ell} \log\frac{ f_{\hat \pi_{\ell}} (N_i)}{f_{\hat \mu_{\ell}}(N_i)} \geq -k\exp(-cnt). \qedhere
\]
\end{proof}

\begin{lemma}
\label{lem:pip_sum_weights}
\[
\sum_{\ell=1}^L |I_\ell|\sqrt{ k_\ell} \lesssim  t\sqrt{k} + \sqrt{Lt}\wedge t^{\frac{1}{3}},
\qquad
\sum_{\ell=1}^L |I_\ell| k_\ell \lesssim tk +  \sqrt{tk}.
\]
\end{lemma}
\begin{proof}
Without loss of generality, let $q_1 < q_2 < \ldots < q_L$. 
By the $r^\star_t$-separation condition under Assumption \ref{ass:separation}, $q_{\ell+1} - (\sqrt{q_{\ell+1} t} + t) \ge q_{\ell} + (\sqrt{q_{\ell} t} + t)$ for all $\ell\ge 1$, implying $\sqrt{q_{\ell+1}} - \sqrt{q_{\ell}}  = \frac{q_{\ell+1}-q_{\ell}}{\sqrt{q_{\ell+1}} + \sqrt{q_{\ell}}} \ge \sqrt{t}$.
It follows that 
\[
q_{\ell}\geq (\ell-1)^2 t,
\qquad
|I_\ell| =  2(\sqrt{q_\ell t} + t) \asymp \sqrt{q_\ell t},
\qquad \forall \ell\ge 2.
\]

If $q_\ell \leq \kappa t$ with $\kappa=100$, then $t\le q_{\ell-1} + r^\star_t(q_{\ell-1}) \leq q_\ell - r^\star_t(q_\ell) \leq q_\ell \leq \kappa t$;
otherwise, if $q_\ell > \kappa t$, then $r^\star_t(q_\ell)=2(\sqrt{q_\ell t}+ t) = 2q_\ell (\sqrt{\tfrac{t}{q_\ell}}+\tfrac{t}{q_\ell}) \leq \tfrac{q_\ell}{2}$.
Therefore, ${q_{\ell}-r^\star_t(q_{\ell})}\asymp{q_\ell}$ for $\ell \ge 2$.
We obtain that 
\[
\sum_{\ell=2}^{L}  q_{\ell} k_\ell
\lesssim \sum_{\ell=2}^{L}  k_\ell (q_{\ell}-r^\star_t{q_{\ell}})
\le \sum_{i=1}^{k} p_{i} \le 1.
\]
Define $\calJ = \{ \ell \in [L] : k_{\ell} \neq 0 \}$.
Applying Cauchy-Schwarz inequality yields $\sum_{\ell=2}^{L} \sqrt{q_\ell k_{\ell}} \lesssim \sqrt{|\calJ|}$ and $\sum_{\ell=2}^{L} \sqrt{q_\ell}k_{\ell}  \lesssim \sqrt{ \sum_{\ell=2}^{L} k_{\ell}}\le \sqrt{k}$.
If $q_1\leq \kappa t$, we have  $|I_1| \asymp \sqrt{q_\ell t} + t \asymp t$. It follows that 
\begin{equation}    
    \label{eq:ub-sum-weights}
    \sum_{\ell=1}^L |I_\ell|\sqrt{ k_\ell} 
    \lesssim t\sqrt{k}+\sqrt{t|\calJ|}, 
    \qquad
    \sum_{\ell=1}^L |I_\ell| k_\ell 
    \lesssim tk+\sqrt{tk}.
\end{equation}
If $q_1> \kappa t$, then ${q_1-r^\star_t(q_1)}\asymp{q_1}$ and $|I_1|\asymp \sqrt{q_1 t}$.
Similarly, we have $ \sum_{\ell=1}^{L} \sqrt{q_{\ell} k_\ell}\lesssim \sqrt{|\calJ|}$ and $ \sum_{\ell=1}^{L} \sqrt{q_{\ell} }k_\ell \lesssim \sqrt{k}$, and thus the upper bounds \eqref{eq:ub-sum-weights} continue to hold.

It remains to upper bound $|\calJ|$. 
Note that 
\[
1 =\sum_{i=1}^k p_i \geq \sum_{\ell \in \calJ \setminus \{1\}} (q_{\ell}-r^\star_t{q_{\ell}}) \asymp  \sum_{\ell \in \calJ \setminus \{1\}} q_{\ell} \gtrsim \sum_{\ell \in \calJ \setminus \{1\}} t(\ell-1)^2 
\gtrsim t(|\calJ|-1)^{3}.
\]
Combining with $|\calJ| \leq L$, we obtain $|\calJ| \lesssim t^{-\frac{1}{3}} \wedge L$ and complete the proof.
\end{proof}

\begin{proof}[Proof of Theorem~\ref{thm:npmle-global}]
Abbreviate $r=r^\star_t$. Denote $I_{\ell}\triangleq \{q_\ell\}_{r}$ and $I \triangleq  \cup_{\ell=1}^L  I_\ell$, where the $I_{\ell}$'s are disjoint under Assumption \ref{ass:separation}.
Let $w_\ell^\star \triangleq \pi_P (I_{\ell})$ and $w_\ell \triangleq \hat \pi (I_{\ell})$. 
Without loss of generality, suppose that all $w_\ell^\star>0$.
The assumption $n\geq \Omega(\frac{k}{\log k})$ implies that $\log n \gtrsim \log k$.
By Proposition~\ref{prop:weight_remainder_asymp}, given any $c_0>0$ and $s=\tfrac{c_0\log n}{n}$, there exists a constant $C$ such that, for $t=C\frac{\log n}{n}$, the following event occurs with probability $1 - \exp(-c'nt)$ for a constant $c'>0$:
\begin{align}
\label{eq:prf-npmle-global-family-event}
A \triangleq \sth{\hat \pi (\cup_{\ell=1}^L I_\ell) = 1, \quad \max_{\ell\in[L]} \abs{w_\ell-w_\ell^\star}\leq \exp(-c'nt) }.
\end{align}
Since $\hat\pi$ and $\pi_{P}$ are supported on $[0,1]$ and thus $W_1(\hat\pi,\pi_{P})\leq 1$, we have 
\begin{align}
\label{eq:global-err-decomp-A}
    \Expect W_{1}(\hat\pi,\pi_{P}) 
    \leq \Expect W_{1}(\hat\pi,\pi_{P})\Indc_A + \exp(-c'nt).
\end{align}
For $\Expect W_{1}(\hat\pi,\pi_{P})\Indc_A$, by the dual representation \eqref{eq:W1_dual}, it suffices to uniformly upper bound $\Expect_{\hat\pi} g- \Expect_{\pi_{P}} g$ for $g\in \calL_1$ under $A$.
Without loss of generality, let $g(0)=0$.

Let $g\in \calL_1$ and $D\ge 1$. 
For each $\ell\in[L]$, by Lemma~\ref{lem:approx-Jackson}, there exists $p_\ell\in\Poly_D$ such that
\begin{align}
\label{eq:prf-npmle-global-family-1}
    |g(x) - p_\ell(x)| \leq c_0 \frac{\sqrt{|I_{\ell}|x} \wedge |I_{\ell}| }{D}, 
    \quad \forall x \in I_{\ell}.
\end{align}
By triangle inequality, $M_\ell \triangleq M(p_\ell,I_{\ell}) \leq M(g,I_{\ell})+2\|p_\ell-g\|_{\infty,I_{\ell}}  \lesssim |I_{\ell}|$.
Applying Lemma~\ref{lem:pois-poly-approx}, there exists $\hat g_\ell (x) = a_\ell + \sum_{j} b_{j,\ell} \poi(j, n x)$ with $\max _{j}\abs{ b_{j,\ell}} \lesssim c_2^D|I_{\ell}|$ such that
\begin{align}
\label{eq:prf-npmle-global-family-2}
    \sup_{x\in I_{\ell}}|p_\ell(x)- \hat g_\ell(x) | 
    \lesssim |I_{\ell}| n\exp\pth{-c_1 nt}.
\end{align}
Define $p(x) \triangleq \sum_{\ell=1}^L p_\ell (x)\indc{x\in I_\ell}$ and $\hat g(x) \triangleq \sum_{\ell=1}^L \hat g_\ell (x) \indc{x\in I_\ell}$.
Then,
\begin{equation}
\label{eq:prf-npmle-global-family-0}
    \Expect_{\pi_{P}} g - \Expect_{\hat\pi} g 
    = \underbrace{\int (p-\hat g)(\diff\pi_{P}-\diff {\hat\pi})}_{\triangleq \calE_1}
    + \underbrace{\int \hat g (\diff\pi_{P}-\diff {\hat\pi})}_{\triangleq \calE_2}
    + \underbrace{\int (g-p) (\diff\pi_{P}-\diff {\hat\pi})}_{\triangleq \calE_3},
\end{equation}
where $\calE_1,\calE_2,\calE_3$ depend on $g$ and $\hat\pi$.
Next we derive upper bounds of $\Expect[\sup_{g\in\calL_1} \calE_i \Indc_A]$.

\paragraph{Bounding $\calE_1$.} 
When $A$ occurs, both $\pi_{P}$ and $\hat\pi$ are supported on $\cup_{\ell=1}^L I_\ell$. 
Hence, 
\begin{equation}
\label{eq:prf-npmle-global-family-3}    
\sup_{g\in\calL_1}\calE_1 \Indc_A
\le 2 \max_{\ell\in[L]}\|p- \hat g\|_{\infty,I_\ell}
=2\max_{\ell\in[L]} \|p_\ell-\hat g_\ell \|_{\infty,I_\ell} 
\lesssim 
n\exp\pth{-c_1 nt}
\triangleq \bar{\calE_1}.
\end{equation}

\paragraph{Bounding $\calE_2$.} 
Denote $\hat\pi_{\ell}$ and $\pi_{P,\ell}$ as the conditional distribution of $\hat\pi$ and $\pi_P$ on $I_\ell$, respectively.
Denote $k_\ell=\sum_{i=1}^k\indc{p_i\in I_\ell}=kw_\ell^\star$.
Under the event $A$, we have 
\begin{align}
\label{eq:global-err-decomp-1}
    \calE_2 
    = \sum_{\ell=1}^L  w_\ell^\star \int \hat g \diff \pi_{P,\ell} - w_\ell \int \hat g \diff \hat\pi_{\ell} 
    \le \sum_{\ell=1}^L w_\ell^\star \abs{\int \hat g_\ell (\diff \pi_{P,\ell}-\diff {\hat\pi_{\ell}})} +  \abs{w_\ell^\star-w_\ell}  \|\hat g\|_{\infty,I_\ell}.
\end{align}
Combining \eqref{eq:prf-npmle-global-family-event}, \eqref{eq:prf-npmle-global-family-1}, and \eqref{eq:prf-npmle-global-family-2} yields $\abs{w_\ell^\star-w_\ell}  \|\hat g\|_{\infty,I_\ell}\lesssim \exp(-c'nt)$.
Additionally,
\begin{align*}
    \abs{\int \hat g_\ell (\diff \pi_{P,\ell}-\diff {\hat\pi_{\ell}})} 
    & \leq \sum_{j=0}^{\infty} \abs{ b_{j,\ell}  (f_{\pi_{P,\ell}}(j)-f_{\hat\pi_{\ell}}(j)) } 
    \le \max_{j}\abs{ b_{j,\ell}} \cdot \| f_{\pi_{P,\ell}}-f_{\hat\pi_{\ell}}\|_1\\
    & \lesssim c_2^D |I_{\ell}| \cdot H(f_{\pi_{P,\ell}},f_{\hat\pi_{\ell}}).
\end{align*}
Denote $I_\ell = [a_\ell,b_\ell]$ and $\epsilon_{\ell}^2 \triangleq \pth{\sqrt{n b_\ell}-\sqrt{n a_\ell} +\sqrt{\log (k_\ell(nb_\ell+1))} }\frac{\log^{5/2} n}{k_\ell} \wedge 1$. 
Since $b_\ell = q_\ell +\sqrt{q_\ell t} + t \asymp (\sqrt{q_\ell} + \sqrt{t})^2$ and $b_\ell-a_\ell = 2r(q_\ell) = 2\sqrt{t} (\sqrt{q_\ell} + \sqrt{t})$, we have $\sqrt{n b_\ell}-\sqrt{n a_\ell} \le \sqrt{n} \frac{b_\ell-a_\ell}{\sqrt{b_\ell}} \lesssim \sqrt{\log n}$.
Furthermore, by $\log k_\ell \le \log k \lesssim \log n $ and $b_\ell \le 1$, we obtain $\epsilon_{\ell}^2\lesssim \frac{\log^3 n}{k_\ell}$.
For $\ell\in [L]$, define
\begin{equation*}
    E_0^{(\ell)} \triangleq \sth{ \frac{1}{k_\ell}\sum_{i: p_i\in I_\ell} \log f_{\hat\pi_{\ell}}(N_i) 
    \geq \frac{1}{k_\ell}\sum_{i: p_i\in I_\ell} \log f_{\pi_{P,\ell}}(N_i) -  \epsilon_{\ell}^2 }, \quad  E_0=\cap_{\ell=1}^L E_0^{(\ell)}.
\end{equation*}
Applying Lemma~\ref{lem:approximate-np} yields $P[E_0]\geq 1- k \exp(-cnt)$.
Then, by Lemma~\ref{lem:pois-hellinger-rate-c}, 
\begin{align*}
    \Expect H(f_{\pi_{P,\ell}},f_{\hat\pi_{\ell}}) 
    = \Expect H(f_{\pi_{P,\ell}},f_{\hat\pi_{\ell}})\Indc_{E_0} + \Expect H(f_{\pi_{P,\ell}},f_{\hat\pi_{\ell}})\Indc_{E_0^c}
    \lesssim \sqrt{\frac{\log^{3} n}{k_\ell}}.
\end{align*}
Consequently, we obtain from \eqref{eq:global-err-decomp-1} that
\begin{align}
\Expect\qth{\sup_{g\in\calL_1} \calE_2 \Indc_A} 
& \lesssim L \exp(-c'nt) + \sum_{\ell=1}^L \frac{k_\ell^\star}{k} c_2^D |I_{\ell}| \cdot \Expect[H(f_{\pi_{P,\ell}},f_{\hat\pi_{\ell}})]  \nonumber \\
& \lesssim L \exp(-c'nt) + \frac{c_2^D  \log^{3/2} n}{k}  \sum_{\ell=1}^L |I_\ell| \sqrt{k_\ell} \nonumber \\
& \stepa{\lesssim} L \exp(-c'nt) + \frac{c_2^D  \log^{3/2} n}{k} \pth{ t\sqrt{k} + \sqrt{Lt}\wedge t^{\frac{1}{3}} }
\triangleq \bar{\calE_2}, \label{eq:prf-npmle-global-family-4}
\end{align}
where (a) applies Lemma~\ref{lem:pip_sum_weights}.

\paragraph{Bounding $\calE_3$.} 
Since $\pi_{P}$ and $\hat\pi$ are supported on $\cup_{\ell=1}^L I_\ell$ under the event $A$, we have
\begin{align*}
\calE_3={\int (g-p_\ell) (\diff\pi_{P}-\diff {\hat\pi})} & \leq  \int |g-p_\ell| (\diff\pi_{P}+\diff {\hat\pi})
= \sum_{\ell=1}^L \int_{I_\ell} |g-p_\ell| (\diff\pi_{P}+\diff {\hat\pi}). 
\end{align*} 
For each $\ell$, if $q_\ell \leq C' t$, then $|I_{\ell}| \lesssim t$. 
Applying \eqref{eq:prf-npmle-global-family-1} yields $|g(x)-p_\ell(x)|\lesssim \frac{\sqrt{|I_{\ell}|x}}{D} \lesssim \frac{\sqrt{t x}}{D}$ for $x\in I_\ell$.
Otherwise, if $q_\ell \geq C' t$, then $|I_{\ell}| \asymp \sqrt{q_\ell t}  \lesssim q_\ell$ and thus $x\gtrsim q_\ell$ for $x\in I_\ell$, which implies that 
$|g(x)-p_\ell(x)|\lesssim \frac{|I_{\ell}|}{D}\asymp \frac{\sqrt{q_\ell t}}{D}\lesssim \frac{\sqrt{t x}}{D}$ by \eqref{eq:prf-npmle-global-family-1}.
Combining both cases yields 
\begin{align*}
\calE_3 \lesssim \frac{\sqrt{t}}{D} \sum_{\ell=1}^L \int_{I_\ell} \sqrt{x} (\diff\pi_{P}+\diff {\hat\pi}) 
= \frac{\sqrt{t}}{D}(\Expect_{\pi_P} \sqrt{X}+\Expect_{\hat\pi} \sqrt{X})
\le \frac{\sqrt{t}}{D}(\sqrt{\Expect_{\pi_P} X}+ \sqrt{\Expect_{\hat\pi} X}).
\end{align*} 
By definition, $\Expect_{\pi_P} X = \frac{1}{k}\sum_i p_i = \frac{1}{k}$.
Note that $f(x)=x$ is a linear function and 1-Lipschitz. 
Then the similar analysis of the error terms $\calE_1$ and $\calE_2$ in \eqref{eq:prf-npmle-global-family-0} continues to hold for $f$, while $\calE_3=0$.
Therefore, $|\Expect_{\hat\pi} X - \Expect_{\pi_{P}} X| \lesssim \bar{\calE_1}+\bar{\calE_2}$.
It follows that 
\begin{align}
\label{eq:prf-npmle-global-family-5}
\sup_{g\in\calL_1}\calE_3\Indc_A 
\lesssim \frac{\sqrt{t}}{D}\pth{\sqrt{\bar{\calE_1}+\bar{\calE_2}} + \sqrt{\frac{1}{k}}}.
\end{align}

\paragraph{Combining the upper bounds.}
Incorporating \eqref{eq:prf-npmle-global-family-0} with \eqref{eq:prf-npmle-global-family-3}, \eqref{eq:prf-npmle-global-family-4}, and \eqref{eq:prf-npmle-global-family-5}, we have for any $D\ge 1$,
\begin{align*}
    \Expect W_{1}(\hat\pi,\pi_{P})\Indc_A \lesssim   \bar{\calE_1}+\bar{\calE_2} + \frac{\sqrt{t}}{D} \pth{\sqrt{\bar{\calE_1}+\bar{\calE_2}} +  \sqrt{\frac{1}{k}}}.
\end{align*}
For $t=C\frac{\log n}{n}$ and $D\lesssim \log n$ such that $c_2^D\lesssim n^{0.1}$,
by the assumption $n\geq \Omega(\frac{k}{\log k})$,
we have $\bar{\calE_1}+ \bar{\calE_2} \lesssim \frac{1}{k}$.
Then we obtain from \eqref{eq:global-err-decomp-A} that
\begin{align}
\label{eq:prf-npmle-global-family-D}
    \Expect W_{1}(\hat\pi,\pi_{P}) \lesssim \frac{c_2^D \log^{3/2} n}{k}\pth{ t\sqrt{k} + \sqrt{Lt}\wedge t^{\frac{1}{3}} } 
    + \frac{\sqrt{t/k}}{D}. 
\end{align}

We are now ready to complete the proof.
Let $C_1>0$ be a universal constant to be chosen. 
We discuss the following cases:

\textit{Case 1: $k\geq C_1(L\wedge n^{1/3}) \log^3n$.} 
Set $D= c_3 \log \frac{k/\log^3n}{L\wedge n^{1/3}}$, 
Since $n\geq \Omega(\frac{k}{\log k})$, we have $D\lesssim \log n$, and with sufficiently small $c_3>0$,
\begin{align*}
    Dc_2^D \lesssim  \sqrt{\frac{1}{t \log^{3}n}} \wedge \sqrt{\frac{k}{(L\wedge t^{1/3})\log^{3}n}} \asymp \frac{\sqrt{kt}}{\log^{3/2} n(\sqrt{k}t + \sqrt{Lt}\wedge t^{\frac{1}{3}})}.
\end{align*}
Then, for sufficiently large $C_1$,
\begin{align*}
\Expect W_{1}(\hat\pi,\pi_{P}) & \lesssim \frac{c_2^D \log^{3/2} n}{k}\pth{ t\sqrt{k} + \sqrt{Lt}\wedge t^{\frac{1}{3}} } + \frac{\sqrt{t}}{D\sqrt{k}} \asymp \frac{\sqrt{t}}{D\sqrt{k}}\asymp \sqrt{\frac{\log n}{kn}}\frac{1}{\log_+(\frac{k/\log^{3}n}{L\wedge n^{1/3}} )}.
\end{align*}

\textit{Case 2: $k< C_1(L\wedge n^{1/3}) \log^3n$.} 
In this case, we apply a simplified argument without using Poisson deconvolution. 
Suppose that $A$ occurs.
Similar to \eqref{eq:global-err-decomp-1}, for any $g\in\calL_1$ with $g(0)=0$,
\[
\int g (\diff \pi_P-\diff {\hat\pi})
\leq \sum_{\ell=1}^L w_\ell^\star \abs{\int g (\diff \pi_{P,\ell}-\diff {\hat\pi_{\ell}})} + \sum_{\ell=1}^L \abs{w_\ell^\star-w_\ell}  \|g\|_{\infty,I_\ell}.
\]
By \eqref{eq:prf-npmle-global-family-event} and $\| g\|_{\infty,I_\ell} \leq 1$, we have $\sum_{\ell=1}^L \abs{w_\ell^\star-w_\ell}  \|g\|_{\infty,I_\ell} \lesssim L \exp(-c'nt)$.
Applying Lemma~\ref{lem:pip_sum_weights} yields 
\begin{align*}
    \sum_{\ell=1}^L w_\ell^\star \abs{\int g (\diff \pi_{P,\ell}-\diff {\hat\pi_{\ell}})} \lesssim \sum_{\ell=1}^L \frac{k_\ell}{k} |I_{\ell}| \lesssim \frac{1}{k}(tk + \sqrt{tk}) = t + \sqrt{\frac{t}{k}}.
\end{align*}
It follows that, under event $A$,
\begin{align*}
    W_{1}(\hat\pi,\pi_{P})  =\sup_{g\in \calL_1} \int g (\diff \pi_P-\diff {\hat\pi}) \lesssim \sqrt{\frac{\log n}{kn}} + \frac{\log n}{n} \asymp \sqrt{\frac{\log n}{kn}}.
\end{align*}
Applying \eqref{eq:global-err-decomp-A}, we have $\Expect W_{1}(\hat\pi,\pi_{P})\lesssim \sqrt{\frac{\log n}{kn}}$. 

Finally, combining the two cases yields the desired result.
\end{proof}

\subsection{Proofs in Section~\ref{sec:truncated_npmle}}
\label{app:prf_truncated_np}

\begin{theorem}
\label{thm:npmle-truncated}
Suppose $\log n \geq \Omega(\log k)$ and $P\in\Delta_{k-1}$ with $\pi_P \in \calP([0,\tfrac{c\log n}{n}])$ for a constant $c>0$.
Let $\hat\pi$ be the NPMLE in \eqref{eq:pois-npmle-def}.
There exist constants $C,c',c_0,C'$ such that, for $t=\tfrac{C\log n}{n}$, with probability $1-\exp(-c'nt)$, 
\begin{align}
    \label{eq:rate_global_ipm_0}
    d_\calF(\pi_{P},\hat\pi) \leq C' \inf_{D\in\naturals}\sup_{g\in\calF}\pth{M(g,[0,t]) \cdot \frac{c_0^D\log^3 n}{\sqrt{k}} + E_D(g,[0,t])}.
\end{align}
Particularly, for any $\epsilon\in (0,1)$ and $\mathcal{F} = \mathcal{F}_{s,\gamma,\boldsymbol{\eta}}$ where either $s < 2\gamma$, or $s = 2\gamma$ with $\eta_{s} = \eta_{s-1} = 0$, there exists $C_\epsilon'>0$ depending on $\epsilon$ such that with probability $1-\exp(-c'nt)$, 
\begin{align}
     \label{eq:rate_global_ipm_1}
     d_\calF(\pi_{P},\hat\pi) \leq C_\epsilon' \pth{\frac{\log n}{n}}^\gamma \pth{(\log n)^{-s+\eta_s\vee \eta_{s-1}}+\frac{n^\epsilon}{\sqrt{k}}}.
\end{align}
\end{theorem}
\begin{proof}
Denote $I= [0,t]$ with $t= C\frac{\log n}{n}$.
Define the events 
\begin{align*}
    A \triangleq \sth{\hat \pi (I)=1, \quad H^2(f_{\pi_{P}},f_{\hat\pi}) \leq C_1\frac{\log^3 n}{k} }.
\end{align*}
Applying Proposition~\ref{prop:weight_remainder_asymp} and Lemma~\ref{lem:pois-hellinger-rate-c} yields $\pbb[A^c]\leq  \exp(-c'nt)$.

It remains to uniformly upper bound $\Expect_{\hat\pi} g- \Expect_{\pi_{P}} g$ for $g\in \calF$ under $A$.
Suppose that $A$ occurs.
Let $p \in \Poly_D$ achieve the best uniform approximation error $E_{D}(g, I)$, and denote $M \triangleq M(p,I)$.  
Applying Lemma~\ref{lem:pois-poly-approx}, there exists $\hat g (x) = a + \sum_{j} b_{j} \poi(j, n x)$ satisfying $\|p- \hat g\|_{\infty,I} \lesssim  M n\exp\pth{-c_1 nt}$ and $\max _{j}\abs{ b_{j}}\leq c_2^DM$. 
Then,
\begin{align*}
\Expect_{\pi_{P}} g - \Expect_{\hat\pi} g 
&= \int (p-\hat g)(\diff\pi_{P}-\diff {\hat\pi})
+ \int \hat g (\diff\pi_{P}-\diff {\hat\pi})
+ \int (g-p) (\diff\pi_{P}-\diff {\hat\pi}) \\
&\leq 2\|p- \hat g\|_{\infty,I} +  \sum_{j=0}^{\infty}  b_{j}  (f_{\pi_{P}}(j)-f_{\hat\pi}(j)) + 2 E_D(g,I) \\
&\lesssim  Mn\exp\pth{-c_1 nt} +  \max _{j}\abs{ b_{j}} \cdot \| f_{\pi_{P}}-f_{\hat\pi}\|_1 + E_D(g,I) \\
& \lesssim M(n\exp\pth{-c_1 nt} + c_2^D  H(f_{\pi_{P}},f_{\hat\pi}) )+ E_D(g,I).
 \end{align*}
By triangle inequality, $M \leq M(g,I)+2E_D(g,I)\leq 3M(g,I)$.
When $A$ occurs, we have $H^2(f_{\pi_{P}},f_{\hat\pi}) \lesssim  \frac{\log^3 n}{k}$.
Taking infimum over $D\in\naturals$ and supremum over $g\in\calF$, we obtain \eqref{eq:rate_global_ipm_0}.
Particularly, for $\calF=\calF_{s,\gamma,\boldsymbol{\eta}}$ such that $s < 2\gamma$ or $s= 2\gamma$ with $\eta_{s}=\eta_{s-1}=0$, applying Lemma~\ref{lem:calF_s} yields that for any $g\in\calF$,
\begin{align*}
 M(g,I) &\lesssim  \|x^\gamma \log^{\eta_0} \pth{1+\frac{1}{x}}\|_{\infty,[0,t]} \lesssim t^\gamma \log^{\eta_0} \pth{1+\frac{1}{t}},\\
 E_D(g,I) &\lesssim
   D^{-s}t^{\frac{s}{2}} \|x^{\gamma-\frac{s}{2}} \log^{\eta_s\vee \eta_{s-1}} \pth{1+\frac{1}{x}}\|_{\infty,[0,t]} \lesssim D^{-s}t^{\gamma}\log^{\eta_s\vee \eta_{s-1}} \pth{1+\frac{1}{t}}.
\end{align*}
Set $D=c_0 \log n$  such that $D>s$ and $c_0^D\leq n^{\frac{\epsilon}{2}}$. 
Substituting into \eqref{eq:rate_global_ipm_0}, \eqref{eq:rate_global_ipm_1} then follows.
\end{proof}
Then, we consider the problem of estimating a symmetric functional $G(P)$, including the Shannon entropy $H(P) = \sum_{i=1}^k p_{i} \log \frac{1}{p_{i}}$, power-sum $F_\alpha(P) =  \sum_{i=1}^k p_{i}^\alpha$, $\alpha \in (0,1)$, and the support size $S(P) =  |\{i \in[k] \mid p_i > 0\}|$,  
with the function $g$ as $h(x)=-x\log x$, $f_\alpha(x)=x^\alpha$, and $s(x)=\indc{x>0}$, respectively. 
Let $I=\{0\}_{r_t^\star}$ with $t>0$ to be specified, and
$\tilde H$, $\tilde F_\alpha$, and $\tilde S$ denote the estimators $\tilde G$ defined in \eqref{eq:G-combined} with $g = h$, $f_\alpha$, and $s$.
After truncated by the corresponding upper and lower bounds of each functional, the proposed localized NPMLE estimators are 
\begin{align*}
    \hat{H} &=
     (\tilde H \wedge \log k)  \vee 0, \\
     \hat{F}_\alpha &=
     (\tilde F_\alpha \wedge k^{1-\alpha})  \vee 0, \\
     \hat{S} &=
     (\tilde S \wedge k)  \vee 0. 
\end{align*}
Denote $\calD_k$ as the family of probability distributions whose minimum non-zero mass is at least $\frac{1}{k}$. By definition, $\calD_k\in \Delta_{k-1}$.
The following proposition establishes the convergence rate of the localized NPMLE estimator, which implies Theorem~\ref{thm:np_trunc_H_estimate} and also provides corresponding results for the functionals $F_\alpha$ and $S$.
\begin{proposition}
\label{prop:np_trunc_HFS_estimate}
Suppose that $\log n \gtrsim \log k$, and $P\in\Delta_{k-1}$.
There exist constants $C,C'$ such that with $t=C\frac{\log n}{n}$,
\begin{align}
\Expect|\hat{H}-H(P)| &\leq C'\pth{ \frac{k}{n\log n} + \frac{\log n}{\sqrt{n}}}, \label{eq:np_truncate_rate_H} \\
\Expect|\hat F_{\alpha}-F_{\alpha}(P)| &\leq \begin{cases}
        C'\frac{k}{(n \log n)^{\alpha}}  ,& \alpha \in (0,1/2],\ \log n \asymp \log k, \\
        C'\pth{\frac{k}{(n \log n)^{\alpha}} + \frac{k^{1-\alpha}}{\sqrt{n}}},& \alpha \in (1/2,1),    
        \end{cases}
\label{eq:np_truncate_rate_Falpha}
\end{align}
and for any $P \in \calD_k$,\footnote{Particularly, given $P\in\calD_k$, we instead optimize the NPMLE program \eqref{eq:npmle-trunc-def} under the additional support constraint $\pi \in \calP([0,1]\setminus (0,\frac{1}{k}))$ for support size estimation.}
\begin{align}
 \Expect\left|\hat{S}-S(P)\right| & \leq C' k\exp \left(-\Theta\left(\sqrt{\frac{n \log k}{k}}\right)\right), \quad n\lesssim k\log k. 
\label{eq:np_truncate_rate_S}
\end{align} 
\end{proposition}

 Compared with the existing minimax rates from \cite{JVHW15,WY15unseen} summarized in Table~\ref{tab:hfs_rates}, the localized NPMLE estimator achieves the optimal sample complexity and (near-)optimal convergence rates for all considered functionals.
 
\begin{table}[htb]
\centering
\renewcommand{\arraystretch}{1.3} 
\setlength{\tabcolsep}{6pt} 
\small
\begin{tabular}{
    >{$}c<{$}  
    >{\centering\arraybackslash}m{4.2cm}  
    >{\centering\arraybackslash}m{4.2cm}
    >{\centering\arraybackslash}m{4.2cm}  
}
\toprule
\textbf{$G$} & \textbf{Minimax rate} & \textbf{Localized NPMLE} & \textbf{Regime} \\
\midrule
H &
\makecell{$\frac{k}{n \log n} + \frac{\log k}{\sqrt{n}}$} &
\makecell{$\frac{k}{n \log n} + \frac{\log n}{\sqrt{n}}$} &
\makecell{$n \gtrsim \dfrac{k}{\log k}$} \\

\midrule
F_\alpha &
\makecell{$\frac{k}{(n \log n)^{\alpha}} 
+ \frac{k^{1-\alpha}\indc{\alpha \in (\tfrac{1}{2},1)}}{\sqrt{n}}$} &
\makecell{$\frac{k}{(n \log n)^{\alpha}} 
+ \frac{k^{1-\alpha}\indc{\alpha \in (\tfrac{1}{2},1)}}{\sqrt{n}}$} &
\makecell{$n \gtrsim k^{1/\alpha}/\log k$ \\ 
$\log n \asymp \log k$ (if $\alpha\in (0,\tfrac{1}{2}]$)} \\
\midrule
S &
\makecell{$k\exp\!\Big(-\Theta\!\big(\sqrt{\tfrac{n \log k}{k}}\big)\Big)$} &
\makecell{$k\exp\!\Big(-\Theta\!\big(\sqrt{\tfrac{n \log k}{k}}\big)\Big)
$} &
\makecell{$\dfrac{k}{\log k}\lesssim n\lesssim k\log k$} \\
\bottomrule
\end{tabular}
\caption{Performance of the localized NPMLE compared to minimax rates.}
\label{tab:hfs_rates}
\end{table}

\begin{remark}
     For support size estimation, we impose a lower bound on the nonzero probabilities (\ie, $P\in\calD_k$) to exclude small probability masses that may be indistinguishable from zero; otherwise, consistent estimation would be impossible.
Moreover, when $n\geq \Omega(k\log k)$, the minimax optimal rate is simply achieved by the empirical distribution \cite{WY15unseen}.
\end{remark}

\begin{proof}[Proof of Proposition~\ref{prop:np_trunc_HFS_estimate}]
Denote $I_{\kappa} \triangleq \{0\}_{r^\star_{\kappa t}}$ for $\kappa>0$,
and let $I=I_1$. Define the following  events: 
\begin{align*}
    A_{1} &= \cap_{i=1}^{k}\left\{ \hat p_i' \in I_1 \Rightarrow p_{i} \in I_{2} \right\},\\ 
    A_{2} &= \cap_{i=1}^{k}\left\{  \hat p_i' \notin I_1  \Rightarrow  p_{i} \notin I_{1/2} \right\}, \\
    A_3 &= \cap_{i=1}^{k}\left\{  p_i \in I_{2} \Rightarrow \hat p_i \in I_{3} \right\}.
\end{align*}
Let $A=\cap_{i=1}^3 A_i$.
Recall that $\pi_{P,I}\triangleq \frac{1}{|\calJ|}\sum_{i\in \calJ}\delta_{p_i}$.
Applying Lemma~\ref{lem:pois_rn}  and the union bound, there exists constants $C,c'$ such that $P[A^c]\leq k\exp(-c'nt)$ with $t=C\frac{\log n}{n}$.
For each $G=H, F_{\alpha},S$, define
\begin{align*}
     \calE_1(G) =  |\calJ| (\Expect_{\hat\pi_I} g- \Expect_{\pi_{P,I}} g), \quad  \calE_2(G) =\sum_{i\in[k]\setminus\calJ} \pth{\tilde g(\hat p_i) - g(p_i)}.
\end{align*}
By definition, $|\hat G - G(P)| \leq |\tilde G - G(P)| =|\calE_1(G) + \calE_2(G)|$.
Since $G,\hat G \in [\underline{G},\overline{G}]$, we have 
\begin{align}
\label{eq:trunc-estim-bias-decomp}
    \Expect|\hat G - G(P)| &= \Expect|\hat G - G(P)|\Indc_{A}+ \Expect|\hat G - G(P)|\Indc_{A^c} \nonumber \\
     &\leq  \Expect |\calE_1(G)|\Indc_{A}+ \Expect |\calE_2(G)|\Indc_{A} + (\overline{G}-\underline{G}) \pbb[A^c].
\end{align}
\textbf{Entropy $G=H$.}
Note that $h = -x\log x\in \calF=\calF_{2,1,(1,0,0)}$.
We have $|\calE_1(H)|\leq |\calJ|  d_\calF(\hat\pi_I,\pi_{P,I})$.
Applying Theorem~\ref{thm:npmle-truncated} yields that, conditioning on the event $A_1$, 
\begin{align*}
|\calJ| d_\calF(\hat\pi_I,\pi_{P,I}) \leq C'|\calJ| \frac{\log n}{n}\pth{ \frac{1}{\log^2 n} +\frac{n^{\epsilon}}{\sqrt{ |\calJ|  }}} \leq C' \frac{\log n}{n}\pth{ \frac{k}{\log^2 n} +\sqrt{k}n^{\epsilon} }
\end{align*}
holds with probability  $1- \exp(-c_1nt)$ for some constants  $C',c_1>0$, where $\epsilon=0.1$.
Moreover, by Lemma~\ref{lem:calF_s}, 
$d_\calF(\hat\pi_I,\pi_{P,I})\leq \sup_{g\in \calF} M(g,[0,1]) \lesssim 1$.
It follows that 
\begin{align*}
    \Expect |\calE_1(H)|\Indc_A 
    &\leq \Expect |\calJ| d_\calF(\hat\pi_I,\pi_{P,I})\Indc_{A_1} 
    \lesssim \frac{\log n}{n}\pth{ \frac{k}{\log^2 n} +\sqrt{k}n^{\epsilon} } + k \exp(-c_1nt).
\end{align*}
Mooreover, substituting $\log k$ by $\log n$ in \cite[Eq.~(61)]{WY16} yields that $\Expect \calE_2^2(H) \lesssim ( \frac{k}{n\log n})^2 + \frac{\log^2 n}{n}$. 
Also note that $H\in[0,\log k]$.
Applying \eqref{eq:trunc-estim-bias-decomp}, we have with sufficiently large $C>0$,
\begin{align*}
\Expect|\hat{H}-H(P)| 
&\lesssim \frac{k}{n\log n} + \frac{\sqrt{k}\log n}{n^{1-\epsilon}} + \frac{\log n}{\sqrt{n}} 
\asymp \frac{k}{n\log n}  + \frac{\log n}{\sqrt{n}},
\end{align*}
where the last inequality holds since $(\tfrac{\sqrt{k}\log n}{n^{1-\epsilon}})^2\lesssim \frac{k}{n\log n}\cdot\frac{\log n}{\sqrt{n}}\lesssim (\frac{k}{n\log n}  + \frac{\log n}{\sqrt{n}})^2$. 

\textbf{Power sum $G=F_\alpha$.}
Fix any $\alpha\in(0,1)$.
For $b=\Theta(\frac{\log n}{n})$, we have $M(f_\alpha, [0, b])\leq O((\frac{\log n}{n})^\alpha)$, and by \cite[Lemma 19]{JVHW15}, $E_{D}(f_\alpha, [0,b])\leq O((n\log n)^{-\alpha})$.
Fix any $\epsilon>0$, and choose $D\asymp \log n$ such that  $c_0^D\leq n^{\frac{\epsilon}{2}}$.
in Theorem~\ref{thm:npmle-truncated}. Then, conditioning on $A$, the inequality
\begin{align*}
|\calJ|\abs{\Expect_{\hat\pi_I}f_\alpha - \Expect_{\pi_{P,I}}f_\alpha} &\leq C'|\calJ|\pth{ \frac{1}{(n\log n)^\alpha} + \pth{\frac{\log n}{n}}^\alpha \frac{n^{\epsilon}}{\sqrt{|\calJ|}} }\\
&\leq C'\pth{ \frac{k}{(n\log n)^\alpha} + \pth{\frac{\log n}{n}}^\alpha \sqrt{k}n^{\epsilon} }
\end{align*}
holds with probability  $1- \exp(-c_1nt)$, where $C'$ depends on $\epsilon$.
It follows that 
\begin{align*}
    \Expect |\calE_1(F_\alpha)|\Indc_A & \leq  C' \pth{ \frac{k}{(n\log n)^\alpha} + \pth{\frac{\log n}{n}}^\alpha \sqrt{k} n^{\epsilon} } + k \exp(-c_1nt).
\end{align*}
Let $\tilde f_\alpha$ be defined in \eqref{eq:g_debias} with $g=f_\alpha$.
\cite[Lemma 2]{JVHW15} implies that with sufficiently large $C>0$,\footnote{Compared with $\tilde f_\alpha$, the bias-corrected estimator used in \cite[Lemma 2]{JVHW15} additionally introduces a smooth cutoff function over the interval $(0, c\log n)$. Nevertheless, with sufficiently large $C>0$, it exactly equals $\tilde f_\alpha$ under the event $A$.}
\begin{align*}
     \Expect [\calE_2^2(F_\alpha)\Indc_A]
     &\lesssim \frac{k^2}{n^{2\alpha}(\log n)^{4-2\alpha}}+  \frac{k}{n^{2\alpha}(\log n)^{1-2\alpha}}, \quad \alpha\in (0,\frac{1}{2}],\\
    \Expect [\calE_2^2(F_\alpha)\Indc_A]
    &\lesssim \frac{k^2}{n^{2\alpha}(\log n)^{4-2\alpha}} + \frac{k}{n^{2\alpha}(\log n)^{2-2\alpha}} + \sum_{i=1}^k\frac{p_i^{2\alpha-1}}{n}, \quad \alpha\in (\frac{1}{2},1].
\end{align*} 
When $\alpha\in(0,\frac{1}{2}]$ and $\log n \asymp \log k$, choose $\epsilon>0$ such that $n^\epsilon\leq k^{\frac{1}{4}}$. Applying \eqref{eq:trunc-estim-bias-decomp} 
with sufficiently large $C>0$ yields that  
\begin{align*}
\Expect|\hat{F}_\alpha-F_\alpha(P)| &\lesssim \frac{k}{(n\log n)^\alpha} + \pth{\frac{\log n}{n}}^\alpha \sqrt{k}{n^{\epsilon}} \asymp \frac{k}{(n\log n)^\alpha}.
\end{align*}
When $\alpha\in(\frac{1}{2},1]$, we have $\sum_{i=1}^k p_i^{2\alpha-1}\leq k (\tfrac{1}{k})^{2\alpha-1}=k^{2-2\alpha}$ by Jensen's inequality. 
Setting $\epsilon=\frac{\alpha-1/2}{4}$ yields
\begin{align*}
    \Expect|\hat{F}_\alpha-F_\alpha(P)| &\lesssim \frac{k}{(n\log n)^\alpha} + \pth{\frac{\log n}{n}}^\alpha \sqrt{k}{n^{\epsilon}} + \frac{k^{1-\alpha}}{\sqrt{n}}\asymp \frac{k}{(n\log n)^\alpha}  + \frac{k^{1-\alpha}}{\sqrt{n}},
\end{align*}
where the last inequality holds since $((\frac{\log n}{n})^\alpha \sqrt{k}{n^{\epsilon}})^2 \lesssim \frac{k}{(n\log n)^\alpha}  \cdot \frac{k^{1-\alpha}}{\sqrt{n}}$ for  $\epsilon = \frac{\alpha-1/2}{4}$.

\textbf{Support size $G=S$.}
Suppose that $n\lesssim k\log k$.
\cite[Eq.~(40)]{WY15unseen} implies that there exists $D\asymp \log k$ and $p\in\Poly_D$ such that $p(0)=0$, and for some universal constant $c_2>0$, 
$$ \sup_{x \in I_3\setminus (0,\frac{1}{k})}
\abs{p(x)-s(x)} \leq \exp\pth{-c_2\sqrt{\frac{n\log k}{k}}} .$$

Note that $E_D(p,I_3)=0$ and $M(p,I_3)\lesssim 1$.
Since $n\lesssim k\log k$, there exist $c_3,\epsilon>0$ such that $n^\epsilon \lesssim k^{\frac{1}{4}} \lesssim\sqrt{k}\exp(-c_3\sqrt{\tfrac{n\log k}{k}})$. 
Conditioning on $A_1$, by Theorem~\ref{thm:npmle-truncated}, the event $|\calJ|\abs{\Expect_{\hat\pi_I} p - \Expect_{\pi_{P,I}} p} \leq C'\sqrt{|\calJ|}n^{\epsilon} \leq C'\sqrt{k}n^{\epsilon}$ holds with probability $1- \exp(-c_1nt)$. 
Under $A_3$, similar to Proposition~\ref{prop:npmle-1}, 
$\hat\pi_I$ is supported on $I_3\setminus (0,\frac{1}{k})$. Then,
\begin{align*}
   \Expect |\calE_1(S)|\Indc_A &\leq  \Expect[ |\calJ| |\Expect_{\hat\pi_I} p- \Expect_{\pi_{P,I}} p|\Indc_{A}] +  \Expect[ |\calJ||\Expect_{\hat\pi_I} (s-p)- \Expect_{\pi_{P,I}} (s-p)| \Indc_{A}] \\ 
   &\lesssim  \sqrt{k}{n^\epsilon} +  k\exp(-c_1nt) + 2k  \sup_{x \in I_3\setminus (0,\frac{1}{k}) } \abs{s(x)-p(x)} \\
 &\leq k \exp\pth{- \Theta\pth{\sqrt{\frac{n\log k}{k}}}}.
\end{align*}
Moreover, under the event $A_2$, for $i\not\in \calJ$ and $p_i>0$, we have $\hat p_i>0$. 
Hence, with $g(x)=s(x)$, $\tilde g (\hat p_i)= g(p_i) =1$ for any $i\not\in \calJ$, and thus  $\Expect |\calE_2(S)|\Indc_A =0$.
Applying \eqref{eq:trunc-estim-bias-decomp} with sufficiently large $C>0$, \eqref{eq:np_truncate_rate_S} then follows.
\end{proof}

\subsection{Proofs in Section~\ref{sec:pen-npmle}}
\label{app:prf_pen_np}

\begin{proof}[Proof of Proposition~\ref{prop:pen-npmle}]
(i) Fix any $k_2>k_1\geq k$. Denote  $\hat \pi_1 = \hat \pi_{k_1} $ and 
$\pi_2 = \frac{k_1}{k_2}\hat\pi_1 + (1-\frac{k_1}{k_2}) \delta_0$.
By definition, $f_{\pi_2}(x) = \frac{k_1}{k_2} f_{\hat\pi_1}(x)$ for $x>0$, and $f_{\pi_2}(0)= \tfrac{k_1}{k_2}f_{\hat\pi_1}(0)+ \tfrac{k_2-k_1}{k_2}$. Then, 
\begin{align*}
    & L(\pi_2;N,k_2) - L(\hat \pi_1;N,k_1) \\ =& \sum_{i=1}^{k} \log \frac{ f_{\pi_2} (N_i) }{f_{\hat \pi_1} (N_i)} + (k_2 - k)\log f_{\pi_2}(0)- (k_1 - k)\log f_{\hat \pi_1}(0) + k_2H(\tfrac{k}{k_2})  - k_1 H(\tfrac{k}{k_1}) \\
    =& k \log \frac{k_1}{k_2} + (k_2-k)\log \frac{\tfrac{k_1}{k_2}f_{\hat\pi_1}(0)+ \tfrac{k_2-k_1}{k_2}}{k_2-k} - (k_1-k)\log \frac{\log f_{\hat \pi_1}(0)}{k_1-k} + k_2\log k_2- k_1\log k_1\\
    =& (k_2-k) \log \frac{k_1 f_{\hat \pi_1}(0) + k_2 - k_1}{k_2-k} -  (k_1-k) \log \frac{k_1 f_{\hat \pi_1}(0)}{k_1-k}.
\end{align*}
Let $f(y)=(k_2-k) \log \frac{k_1 y+ k_2 - k_1}{k_2-k} -  (k_1-k) \log \frac{k_1y}{k_1-k}$, $y\in [0,1]$. Taking derivative yields that $f$ attains its minimum $f(y')=0$ at $y'=\tfrac{k_1-k}{k_1}$.
Consequently, (i) follows from $L(\hat \pi_{k_2};N,k_2)\geq L(\pi_2;N,k_2)\geq L(\hat \pi_{1};N,k_1)$.

(ii) By (i), it suffices to prove the statement for all $k'\in \naturals$ such that $k'\geq \hat k>k$.
Denote  $\hat \pi' = \frac{\hat k}{k'}\hat\pi + (1-\frac{\hat k}{k'}) \delta_0$. We have for any $Q\in \calP([0,1])$,
\begin{align*}
\sum_{i=1}^k  \frac{ f_{Q}(N_i) }{ f_{\hat \pi'}(N_i) } + k' f_{Q}(0) \stepa{=} \frac{k'}{\hat k}\pth{\sum_{i=1}^k \frac{ f_{Q}(N_i) }{ f_{\hat \pi}(N_i) } + \hat k f_{Q}(0)}  \stepb{\leq} \frac{k'}{\hat k} \hat k = k',
\end{align*}
where (a) holds by $N_i>0$ for $i\in[k]$, and (b) follows from 
\eqref{eq:1st-opt-pen-pi}.
Then, given $k'\in\naturals$, the first-order optimality condition \eqref{eq:1st_opt} satisfies for $\hat \pi'$. Since Proposition~\ref{prop:npmle-1} implies the uniqueness of such  $\hat \pi'$, $\hat \pi'$ is the Poisson NPMLE given $k'$.
Moreover, applying the derivation in (i) with the fact $f_{\hat \pi}(0)= \frac{\hat k - k}{\hat k}$
yields that $L(\hat \pi';N,k') = L(\hat \pi;N,\hat k)$.
Finally, (ii) follows.
\end{proof}

\section{Experiment Details}
\subsection{Implementation details of the NPMLE} 
\label{app:exp_algo}
We construct a finite grid $\{r_j\}_{j=1}^m$ given the input $N$ as follows.
We set the grid size $m$ range from 500 to 2000, which increases as the sample size $n$ grows.
Denote $\bar N = \max_{i=1}^k N_i$. If $\bar N \leq \frac{1.6\log n}{n}$, then $r_j= \frac{j-1}{m-1}\bar N $ is uniformly placed over $[0, \bar N]$. Otherwise, half of the grid points are uniformly placed over $[0, \frac{1.6\log n}{n}]$ and the remaining half are uniformly distributed over $(\frac{1.6\log n}{n}, \bar N]$.
We optimize the Poisson NPMLE \eqref{eq:pois-npmle-def}  over $\calP(\{r_i\}_{i=1}^m)$. Define $A=(A_{ij})\in\reals^{k \times m}$ with $A_{ij}=\poi(N_i,nr_j)$. Then, \eqref{eq:pois-npmle-def} is reduced to 
\begin{equation*}
    \hat{\pi} = \sum_{j=1}^m \hat{w}_j\delta_{r_j}, \quad  \hat w \in \argmax_{w\in \Delta_{m-1}} \frac{1}{k} \sum_{i=1}^{k} \log \left(\sum_{j=1}^{m} A_{i j} w_{j}\right),
\end{equation*}
which is a finite-dimensional convex  program. 
 Using a Lagrangian multiplier, we can also write the Lagrangian dual problem as 
\begin{equation}
    \label{eq:npmle-dual}
   \max_{v_1,...,v_k>0} \quad \sum_{i=1}^{k} \log v_{i} \quad \text { s.t. } \quad \frac{1}{k} A^{\top} v \leq \mathbf{1}_{m}.
\end{equation}
The optimal solution of the primal and the dual problems, $\{\hat w_j\}_{j=1}^m$ and $\{\hat v_i\}_{i=1}^k$, are related through the following equations (see also \cite[Theorem 2]{KM13}):
\begin{align}
    \label{eq:dual-wv}
    \sum_{j=1}^m  A_{ij} \hat w_j = 1/\hat v_i, \ i\in[k]; \quad 
    \hat w_j=0 \text{ if } \frac{1}{k} \sum_{i=1}^k \hat v_i A_{ij} <1.
\end{align}
The overall procedure is summarized in the following Algorithm~\ref{algo:npmle}. For estimating a specific functional $g$, one can then apply the plug-in formula \eqref{eq:F_hat} to the output of Algorithm~\ref{algo:npmle}.

\begin{algorithm}[H]
    \caption{Solving the NPMLE} 
    \label{algo:npmle}
    \begin{algorithmic}[1]
    \STATE \textbf{Input:} Frequency counts $N_1,\ldots,N_k$; grid size $m$; concentration parameter $n$.
    \STATE \textbf{Step 1:} Construct the grid $\{r_j\}_{j=1}^m$, and compute $A=(A_{ij})$ with $A_{ij}=\poi(N_i,nr_j)$.
    \STATE \textbf{Step 2:} Solve the NPMLE dual problem \eqref{eq:npmle-dual}. 
    \STATE \textbf{Step 3:} Obtain weights $\hat w_j$ via \eqref{eq:dual-wv}. 
    \STATE \textbf{Output:} $\hat\pi = \sum_{j=1}^m \hat w_j \delta_{r_j}$.
    \end{algorithmic}
\end{algorithm}

To compute the localized NPMLE \eqref{eq:npmle-trunc-def}, we set $I=[0, \kappa\tfrac{\log n}{n}]$ with a tuning parameter $\kappa>0$,
which is equivalent to the original formulation $I=\{0\}_{r_t^\star}$ with $t=\Theta(\tfrac{\log n}{n})$. 
Given one sequence of frequency counts $N=(N_1,\ldots,N_k)$, we optimize \eqref{eq:npmle-trunc-def} with $\calJ =\{ i \in [k] : \hat p_i = \frac{N_i}{n} \leq \kappa \cdot \frac{\log n}{n} \}$ (\ie, letting $N'=N$ in \eqref{eq:npmle-trunc-def}).
In our experiments, we set $\kappa = 3.6$.
The localized NPMLE is then combined with the bias-corrected estimator to yield the proposed estimator $\hat G$ in Section~\ref{sec:truncated_npmle} for a given symmetric functional $G$.
The procedure is summarized in Algorithm~\ref{algo:npmle-trunc}.
\begin{algorithm}[H]
\caption{Symmetric functional estimation via the localized NPMLE}
\label{algo:npmle-trunc}
\begin{algorithmic}[1]
\STATE \textbf{Input:} Frequency counts $N_1, \ldots, N_k$; concentration parameter $n$; grid size $m$; truncation threshold $\kappa$; target function $g$; upper and lower bounds $\Bar{G},\underline{G}$.
\STATE \textbf{Step 1:} Apply Algorithm~\ref{algo:npmle} to $\{N_i : i \in \mathcal{J} \}$ with grid size $m$ to obtain $\hat\pi_{\mathcal{J}}$.
\STATE \textbf{Step 2:} For $i \in [k] \setminus \mathcal{J}$, compute the bias-corrected estimate $\tilde{g}(\hat p_i)$ as defined in \eqref{eq:g_debias}.
\STATE \textbf{Step 3:} Combine both components to obtain the final estimator $\tilde G$ in \eqref{eq:G-combined}.
\STATE \textbf{Output:} Functional estimate  $\hat{G} = (\tilde G \wedge \Bar{G})  \vee \underline{G}.$
\end{algorithmic}
\end{algorithm}

For implementing the penalized NPMLE, we add a small regularization term $\tfrac{c_0}{k'^{c_1}}$ to the penalized likelihood \eqref{eq:Ni-pen-likelihood} to identify the smallest minimizer $\hat k$, which serves as an estimation of $k^\star$. In practice, we choose $c_0=10$ and $c_1=1$. The full computational procedure with grid discretization is summarized in Algorithm~\ref{algo:pen-npmle}.

\begin{algorithm}[H]
    \caption{Solving the penalized NPMLE} 
    \label{algo:pen-npmle}
    \begin{algorithmic}[1]
    \STATE \textbf{Input:} Positive frequency counts $N_1,\ldots,N_k$; parameters $m$, $n$, $c_0$, $c_1$. 
    \STATE \textbf{Step 1:} Construct the grid $\{r_j\}_{j=1}^m$, and compute $A=(A_{ij})$ with $A_{ij}=\poi(N_i,nr_j)$.
    \STATE \textbf{Step 2:} Optimize the penalized NPMLE program 
    \begin{align*}
        \max_{k'\geq k, w \in \Delta_m} \ \sum_{i=1}^{k} \log \left(\sum_{j=1}^{m} A_{i j} w_{j}\right) + (k' - k)\log \left(\sum_{j=1}^{m} e^{-nr_j} w_{j}\right) + k'H(\frac{k}{k'}) + \frac{c_0}{k'^{c_1}},
    \end{align*}
    and obtain the solution $\hat k, \hat w$.
    \STATE \textbf{Output:} $\hat k$, $\hat\pi = \sum_{j=1}^m \hat w_j \delta_{r_j}$.
    \end{algorithmic}
\end{algorithm}

\begin{remark}[Approximation error due to discretization]
The discretization procedure introduces numerical error that grows with the grid size. In practice, we increase the grid size $m$ with $n$ to prevent it from dominating the estimation error.
One may also resort to non-grid algorithms to eliminate this discretization error, such as gradient flow-based methods (e.g., \cite{yan2024flow} for Gaussian mixtures). We leave this for future work.
\end{remark}

\subsection{Additional simulation results}
\label{app:add_exp}
This subsection presents additional simulation results following the setup in Section~\ref{sec:simulation}.
 We consider the underlying distribution $P\in\Delta_{k-1}$
 as listed in Table~\ref{tab:distribution_generate}.
\begin{table}[H]
\centering
\renewcommand{\arraystretch}{1.25}
\setlength{\tabcolsep}{8pt}

{\small
\begin{tabular}{
    l
    >{$}l<{$}
}
\toprule
\textbf{Distribution} & \textbf{Definition of } P = (p_1, \ldots, p_k) \\
\midrule
Uniform & p_i = k^{-1},\ i \in [k] \\[2pt]

2-Mixed Uniform &
\makecell[l]{
p_i = \tfrac{2}{5k},\ i = 1,\ldots,\tfrac{k}{2}; \quad
p_i = \tfrac{8}{5k},\ i = \tfrac{k}{2}+1,\ldots,k
} \\[2pt]

Spike-and-uniform &
\makecell[l]{
p_i = \tfrac{1}{2(k-3)},\ i \in [k-3]; \quad
p_{k-2} = p_{k-1} = \tfrac{1}{8},\ 
p_k = \tfrac{1}{4}
} \\[2pt]

Geometric &
p_i \propto (1 - \theta)^i,\quad \theta = 1/k \\[2pt]

Log-series &
p_i \propto \tfrac{(1 - \theta)^i}{i},\quad \theta = 1/k \\[2pt]

Zipf(1) & p_i \propto i^{-1} \\
\bottomrule
\end{tabular}
}
\caption{Underlying distributions used in simulation.}
\label{tab:distribution_generate}
\end{table}

 Firstly, for Shannon entropy estimation, Figure~\ref{fig:fig-ent-2} presents additional results for the remaining distributions listed in Table~\ref{tab:distribution_generate} that complements those in Figure~\ref{fig:fig-ent-1}.

Next, we evaluate the performance of estimators on the support size functional. This experiment focuses on the large-alphabet regime, since in the large-sample regime, the error naturally vanishes as most categories are likely to be observed at least once.
To ensure the problem remains non-trivial, we select the support sizes of the underlying distributions in Table~\ref{tab:distribution_generate} such that the minimum non-zero probability mass is approximately $p_{\min} \approx 10^{-5}$. After generating the frequency counts via multinomial sampling, we pad zeros to the count vector to reach a total length of $k = 10^5$, and apply the NPMLE-based estimators on this extended vector. 
In the implementation, we discard any non-zero grid points smaller than $p_{\min}$ when constructing the grid $\{r_j\}$. 
Figure~\ref{fig:fig-supp-large-k} presents simulation results comparing the NPMLE estimators with several baseline methods, including the Empirical estimator, the Good–Turing estimator (GT)~\cite{GT53}, WY, VV, and PML. The performance is evaluated using the scaled RMSE, obtained by dividing the RMSE by the true support size. The NPMLE-based estimators perform among the best, and the localized NPMLE provides additional improvements in the more challenging heterogenous settings (e)–(f).
\begin{figure}[H]
    \centering
    \subfigure[2-Mixed Uniform]{
    \includegraphics[width=0.3\linewidth]{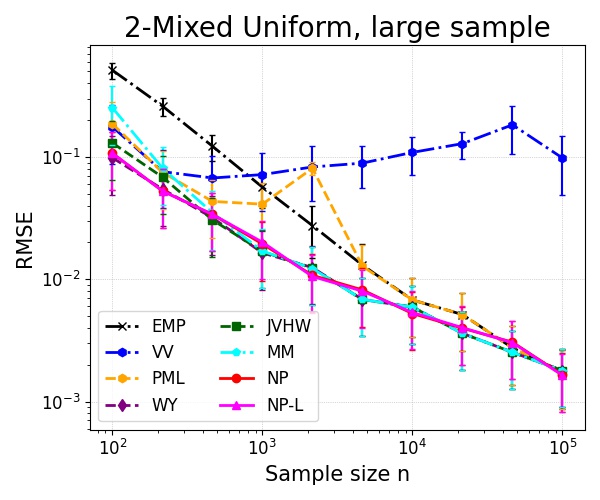}
    }
    \subfigure[Geometric]{
    \includegraphics[width=0.3\linewidth]{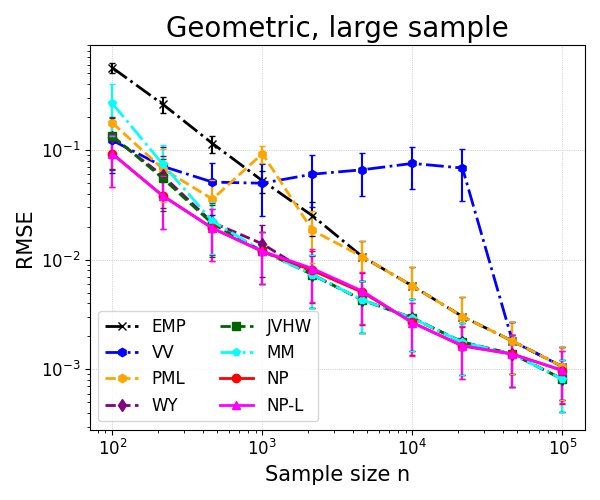}
    }
    \subfigure[Log-Series]{
    \includegraphics[width=0.3\linewidth]{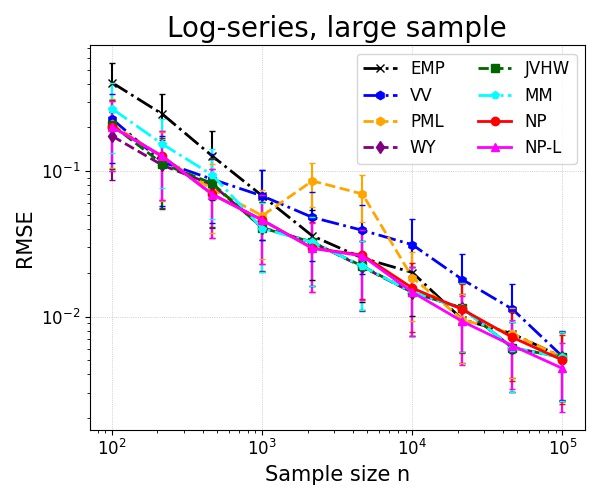}
    }

    \subfigure[2-Mixed Uniform]{
    \includegraphics[width=0.3\linewidth]{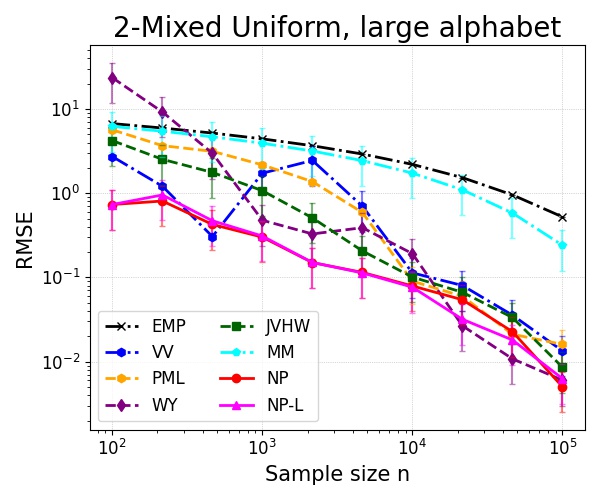}
    }
    \subfigure[Geometric]{
    \includegraphics[width=0.3\linewidth]{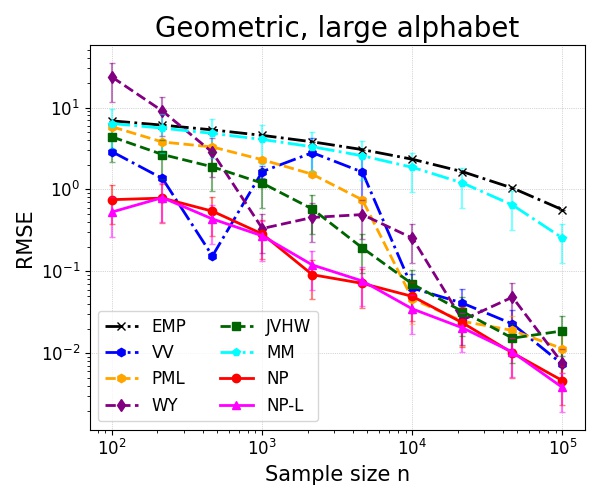}
    }
    \subfigure[Log-Series]{
    \includegraphics[width=0.3\linewidth]{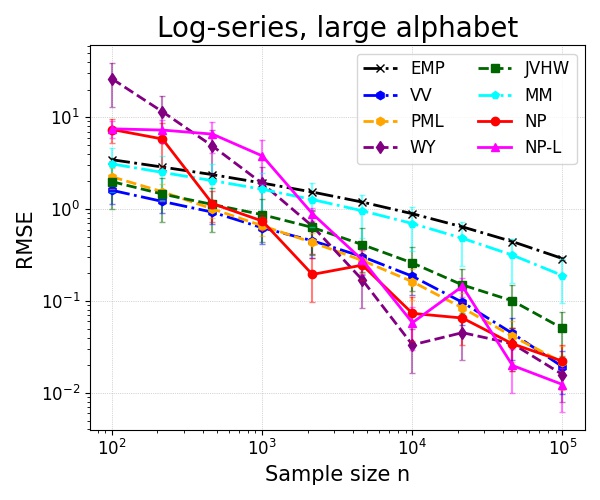}
    }
    \caption{Shannon entropy estimation (continue). (a)-(c) under the large-sample regime, and (d)-(f) under the large-alphabet regime.}
    \label{fig:fig-ent-2}
\end{figure}

\begin{figure}[H]
    \centering
    \subfigure[Uniform]{
    \includegraphics[width=0.3\linewidth]{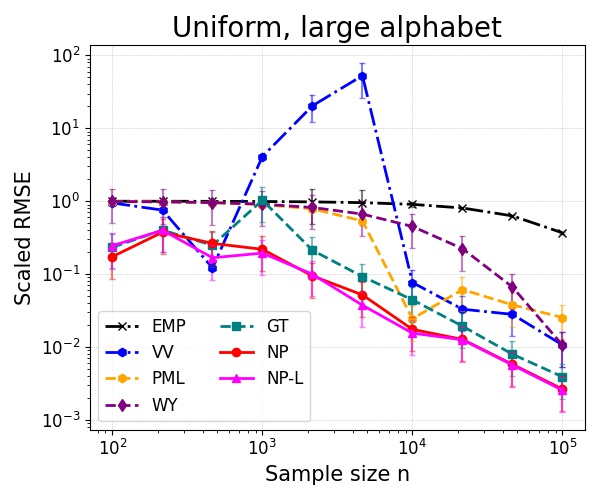}
    }
    \subfigure[2-Mixed Uniform]{
    \includegraphics[width=0.3\linewidth]{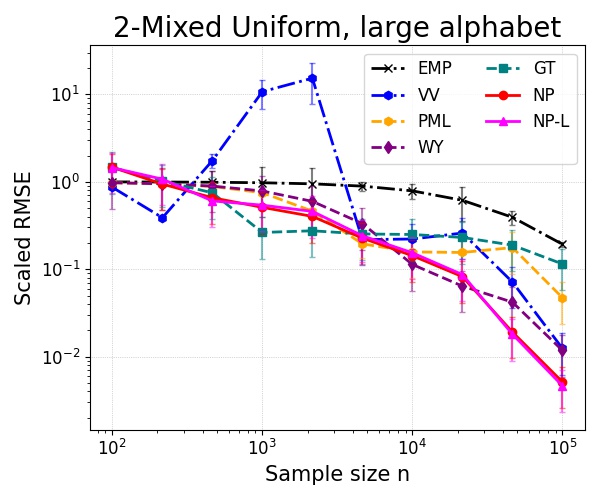}
    }
    \subfigure[Spike-and-uniform]{
    \includegraphics[width=0.3\linewidth]{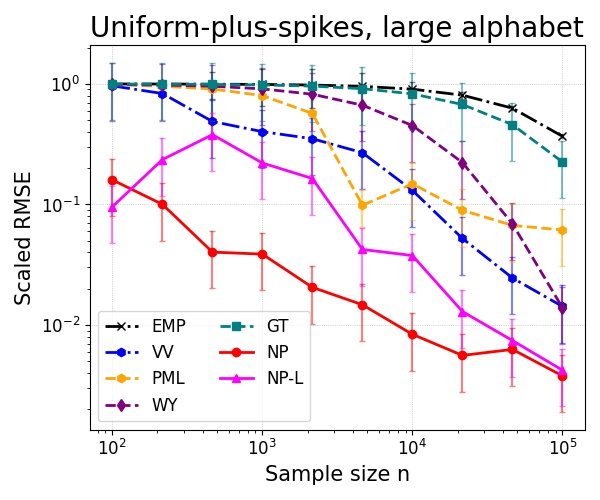}
    }
    \subfigure[Geometric]{
    \includegraphics[width=0.3\linewidth]{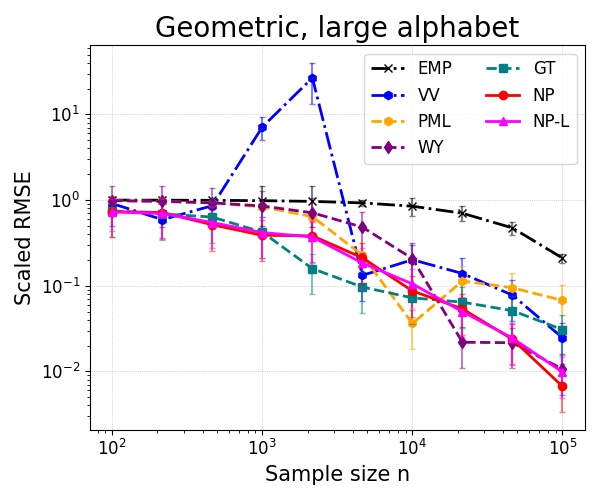}
    }
    \subfigure[Log-Series]{
    \includegraphics[width=0.3\linewidth]{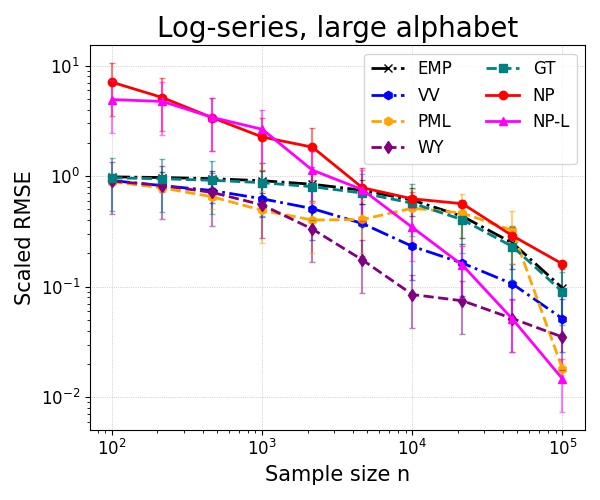}
    }
    \subfigure[Zipf(1)]{
    \includegraphics[width=0.3\linewidth]{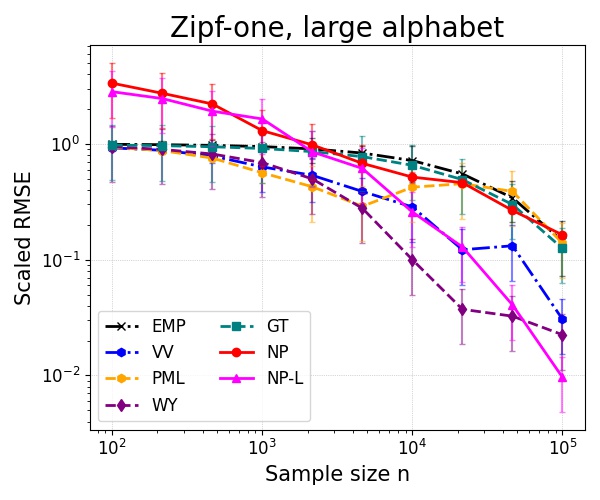}
    }
    \hspace{.0in}
    \caption{Support estimation under the large-alphabet regime.}
    \label{fig:fig-supp-large-k}
\end{figure}

We also provide experimental results for estimating the R\'enyi entropy, which is another important measure in information theory. For any $\alpha > 0$ with $\alpha \neq 1$ and a distribution $P \in \Delta_{k-1}$, the $\alpha$-R\'enyi entropy is defined as
$$ H_\alpha(P) = \frac{\log F_\alpha(P)}{1 - \alpha}, $$
where $F_\alpha(P) = \sum_{i=1}^k f_{\alpha}(p_i) = \sum_{i=1}^k p_i^\alpha$ is the $\alpha$-power sum.

We set $\alpha=0.5$ and estimate $H_\alpha$ using the plug-in estimator $\hat{H}_\alpha \triangleq \frac{\log \hat{F}_\alpha}{1 - \alpha}$, where $\hat{F}_\alpha$ can be obtained via the NPMLE plug-in estimator or the localized NPMLE estimator in Algorithms~\ref{algo:npmle} and~\ref{algo:npmle-trunc}.
The results for both the large-sample and large-alphabet regimes are presented in Figures~\ref{fig:fig-renyi-large-n} and~\ref{fig:fig-renyi-large-k}, respectively, where the NPMLE-based estimators again demonstrate significant advantages over the existing methods.

\begin{figure}[H]
    \centering
    \subfigure[Uniform]{
    \includegraphics[width=0.3\linewidth]{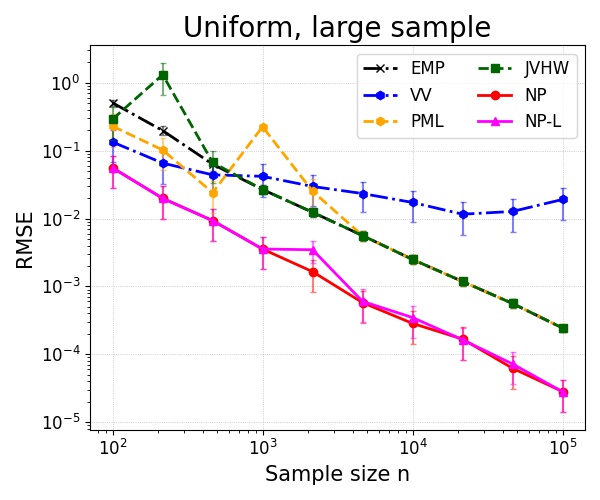}
    }
    \subfigure[2-Mixed Uniform]{
    \includegraphics[width=0.3\linewidth]{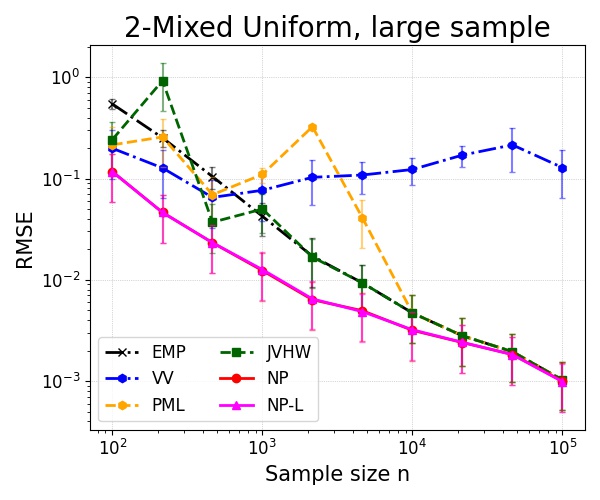}
    }
    \subfigure[Spike-and-uniform]{
    \includegraphics[width=0.3\linewidth]{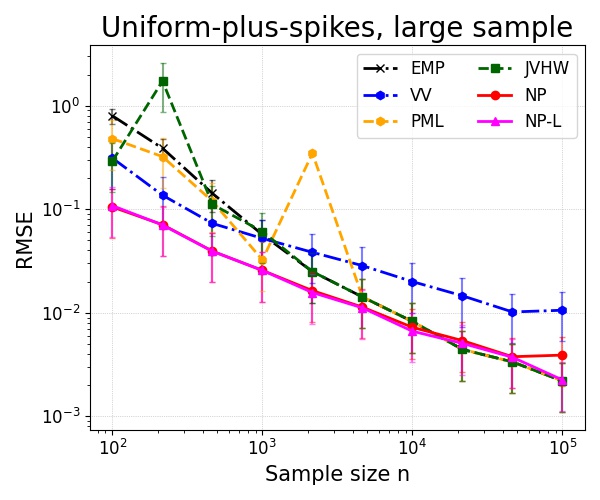}
    }
    \subfigure[Geometric]{
    \includegraphics[width=0.3\linewidth]{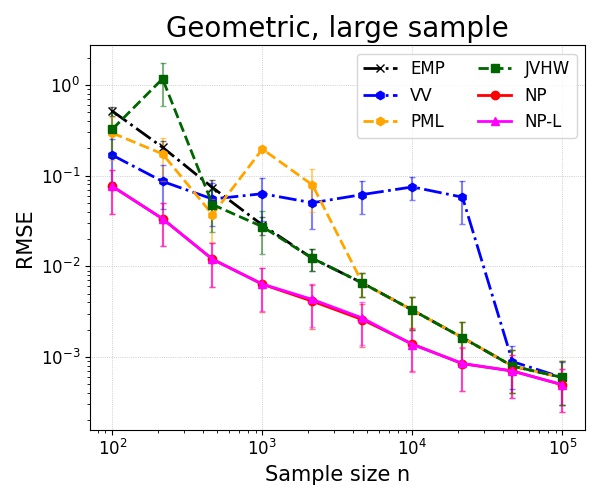}
    }
    \subfigure[Log-Series]{
    \includegraphics[width=0.3\linewidth]{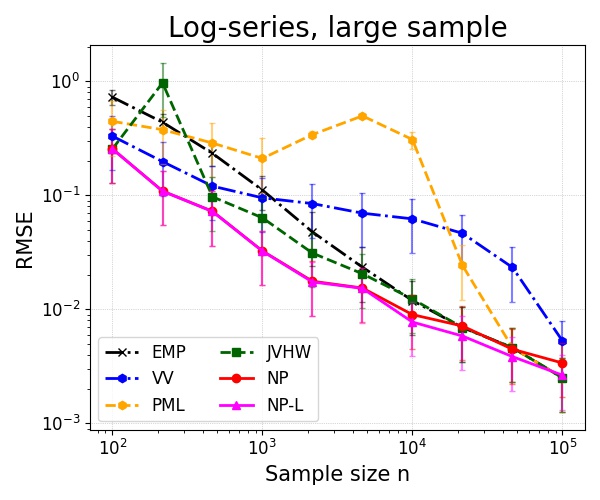}
    }
    \subfigure[Zipf(1)]{
    \includegraphics[width=0.3\linewidth]{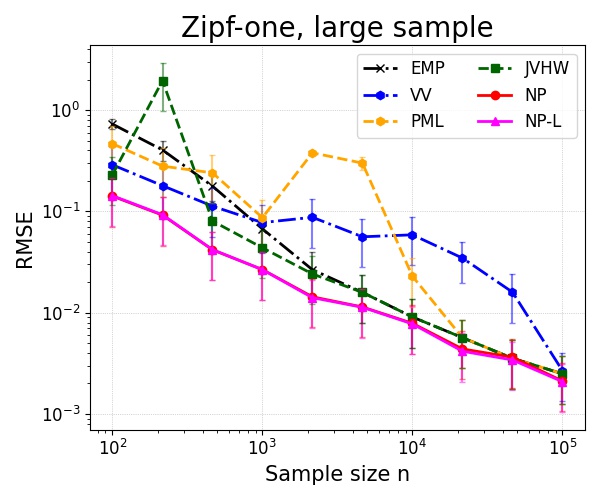}
    }
    \hspace{.0in}
    \caption{0.5-R\'enyi entropy estimation under the large-sample regime.}
    \label{fig:fig-renyi-large-n}
\end{figure}

\begin{figure}[H]
    \centering
    \subfigure[Uniform]{
    \includegraphics[width=0.3\linewidth]{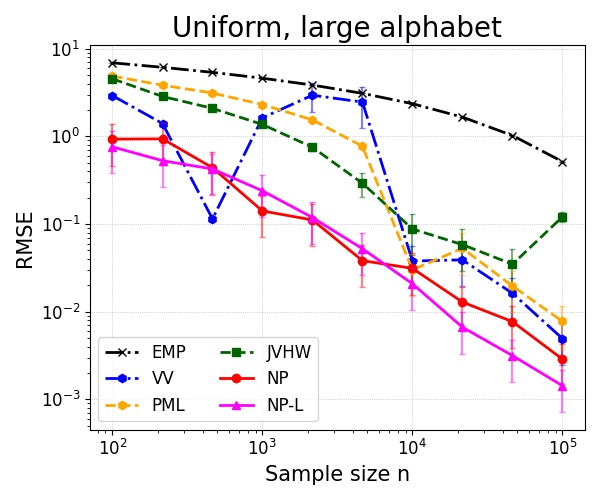}
    }
    \subfigure[2-Mixed Uniform]{
    \includegraphics[width=0.3\linewidth]{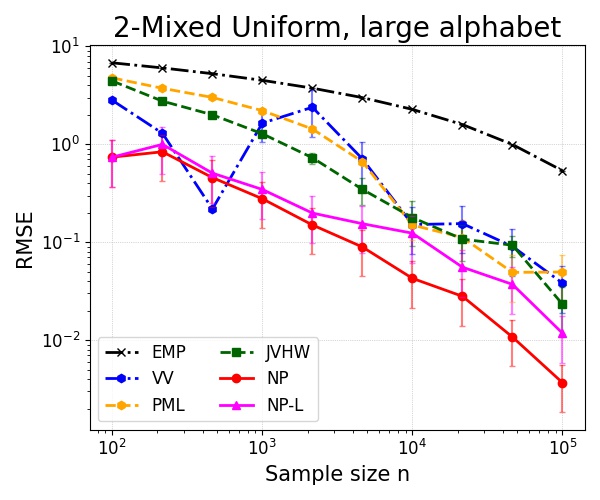}
    }
    \subfigure[Spike-and-uniform]{
    \includegraphics[width=0.3\linewidth]{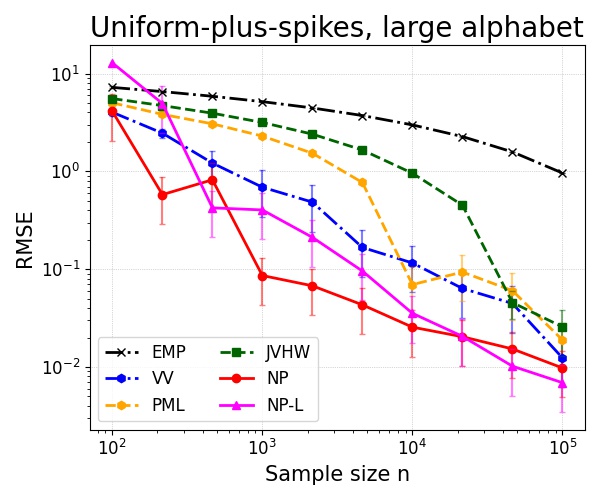}
    }
    \subfigure[Geometric]{
    \includegraphics[width=0.3\linewidth]{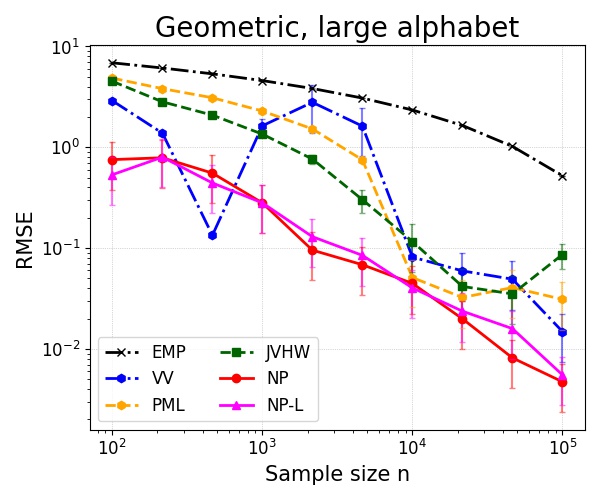}
    }
    \subfigure[Log-Series]{
    \includegraphics[width=0.3\linewidth]{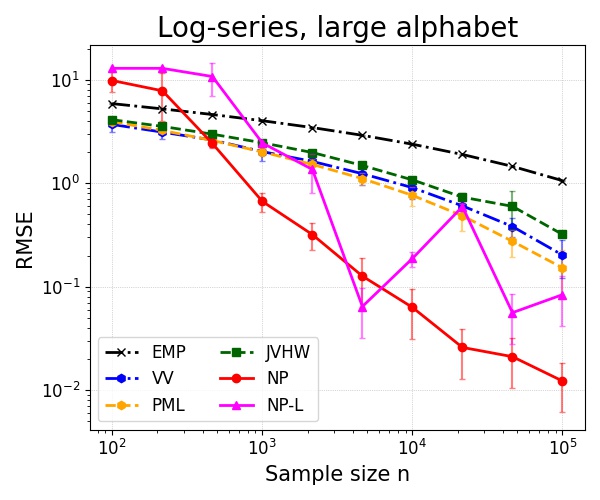}
    }
    \subfigure[Zipf(1)]{
    \includegraphics[width=0.3\linewidth]{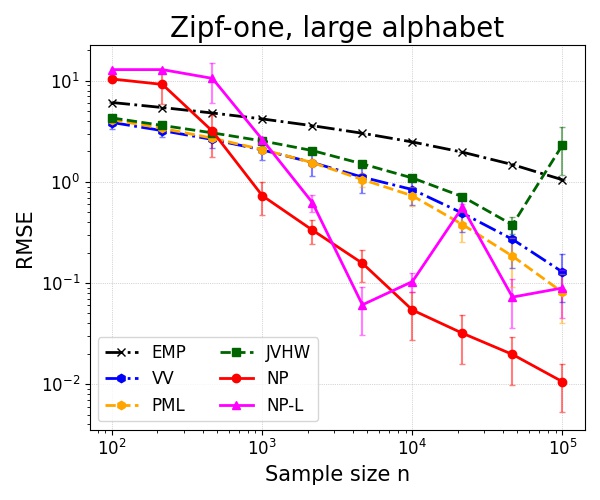}
    }
    \hspace{.0in}
    \caption{0.5-R\'enyi entropy estimation under the large-alphabet regime.}
    \label{fig:fig-renyi-large-k}
\end{figure}

\subsection{Details of experiments on LLMs}
\label{app:llm}

We provide experimental details of the Section~\ref{sec:llm} in this subsection.

\paragraph{General procedures.}
Given each (model, dataset) pair, the experiment proceeds as follows.
\begin{enumerate}
    \item \textit{Content generation.} Randomly select $n_0$ questions from the dataset.
Each question undergoes a 3-stage sampling process:
(1) generate a \textit{reference answer} at low temperature ($T = 0.1$) as the stable baseline\footnote{The reference answer may itself be incorrect due to missing complementary information in the pretraining procedure, which may require additional knowledge or external tools. Nevertheless, we focus solely on the model’s robustness in terms of output uncertainty, while consistently wrong outputs (e.g., arising from training on erroneous data) are tolerated.};
(2) sample $m_1$ \textit{testing answers} at high temperature ($T = 1$) to  obtain a ground-truth label for whether the model hallucinates on the problem;
(3) sample $m_2$ \textit{observed answers} at high temperature for entropy estimation. In our experiment, we set $n_0=200$, $m_1=50$, and $m_2=10$.

\item \textit{Embedding.} Use the \texttt{multilingual-e5-base} model to embed the reference answer and $m_1$ testing answers into 768-dimensional unit vectors, denoted by $v_i^\star$ and $\{v_{i,j}\}_{j=1}^{m_1}$ for the $i\Th$ question.
Semantically similar answers in general yield close embeddings.
The ground-truth label is then defined with a threshold hyperparameter $\gamma\in(0,1)$ as 
\begin{align*}
    u_i = \indc{\frac{1}{m_1}\sum_{j=1}^{m_1} \langle v_i^\star, v_{i,j}\rangle > \gamma},
\end{align*}
where $u_i=0$ indicates hallucination and $u_i=1$ otherwise.  
The hyperparameter $\gamma$ is chosen as the lower $q$-th quantile of the collection of cosine similarities $\{\langle v_i^\star, v_{i,j}\rangle\}_{i\in[n_0],j\in[m_1]}$ across all questions. We set $q=0.35$ to balance the number of positive and negative labels, and clamp the threshold within $[0.75,0.95]$ to ensure its reasonability.

\item \textit{Evaluation.}
Obtain semantic labels of the observed answers via an entailment-based clustering model. Next, apply Shannon entropy estimators to the resulting semantic vector of length $m_2$, and the estimates are then used for the classification task. The performance is evaluated by the ROC and AUC metrics for the binary event.
\end{enumerate}

\paragraph{Entailment algorithm.}
We adopt the bidirectional entailment clustering algorithm from \cite[Algorithm 1]{farquhar2024detecting}, summarized in Algorithm~\ref{alg:bi_entail_clustering}. This method prompts ChatGPT-3.5 to classify the relationship between pairs of answers as ``entailment,'' ``neutral'', or ``contradiction''. 
Two answers are assigned to the same cluster if they mutually entail each other. 
As noted in \cite{farquhar2024detecting}, LLM-based raters achieve performance comparable to human raters.
Semantic clusters are then formed by greedily aggregating answers with equivalent meaning.

\begin{algorithm}[htbp]
\caption{Bi-directional Entailment Clustering}
\label{alg:bi_entail_clustering}
\begin{algorithmic}[1]
\STATE \textbf{Input:} Context $\boldsymbol{x}$, sequences $\{\mathbf{s}^{(2)},\dots,\mathbf{s}^{(M)}\}$, classifier $\mathcal{M}$, initial cluster $C=\{\{\mathbf{s}^{(1)}\}\}$
\FOR{$m = 2$ \textbf{to} $M$}
    \FOR{each $c \in C$} 
        \STATE $\mathbf{s}^{(c)} \gets c_0$
        \STATE $ \texttt{left} \gets \mathcal{M}(\mathbf{s}^{(c)}, \mathbf{s}^{(m)})$
        \STATE $\texttt{right} \gets \mathcal{M}(\mathbf{s}^{(m)}, \mathbf{s}^{(c)})$
        \IF{\texttt{left} = \texttt{entailment} \textbf{and} \texttt{right} = \texttt{entailment}}
            \STATE $c \gets c \cup \mathbf{s}^{(m)}$
        \ENDIF
    \ENDFOR
    \STATE $C \gets C \cup \{\mathbf{s}^{(m)}\}$
\ENDFOR
\STATE \textbf{Output:} $C$
\end{algorithmic}
\end{algorithm}

\paragraph{Prompts.}
We instruct the models to generate answers with the following prompt:
\begin{verbatim}
Answer the question as briefly as possible of no more than 15 words.
Do not add explanations or extra information.
\end{verbatim}
The prompt used for the entailment model is as follows:
\begin{verbatim}
We are evaluating answers to the question "{question}"
Here are two possible answers:
Possible Answer 1: {text1}
Possible Answer 2: {text2}
Does Possible Answer 1 semantically entail Possible Answer 2?
Respond only with entailment, contradiction, or neutral.
Response:
\end{verbatim}

\paragraph{Proposed estimators.}
We compare the performance of NP with two baseline methods EMP and TOK as developed in \cite{farquhar2024detecting}.
For a fair comparison across sequences of varying lengths, we first apply length normalization by taking the arithmetic mean of the log-probabilities of all tokens conditioned on previous tokens. The sequence probabilities are then aggregated according to their semantic labels and normalized into  unit vector, which represents the occurrence probabilities of each semantic category. Finally, TOK is computed as the Shannon entropy of this semantic distribution.
Without using the logit bits, EMP and NP are simply plug-in estimates \eqref{eq:F_hat} based on the empirical histogram and NPMLE given the frequency counts of observed semantic labels, respectively. In particular, NP is implemented with an alphabet size $k = \floor{2.5 m_2}$ to account for unseen semantic categories.

\end{document}